\documentclass[a4paper,11pt]{amsart}
\usepackage[left=1.2in,right=1.2in,top=1.2in,bottom=1.2in]{geometry}  
\usepackage{amsfonts,amssymb,amsthm,amsmath,amscd}
\usepackage{mathtools}
\usepackage{enumitem}
\usepackage[percent]{overpic}
\usepackage{tikz}
\usepackage{mathrsfs}
\usepackage{xargs}
\usepackage{enumerate}
\usepackage{comment}
\usepackage[colorlinks,linkcolor={black},citecolor={black},urlcolor={black}]{hyperref}
\hypersetup{
	linkcolor={black!30!blue},
	citecolor={black!30!green},
	urlcolor={black!30!blue}
}
\usepackage[capitalise]{cleveref}
\usepackage{changepage}
\usepackage{times}
\newcommand{\spec}{\mathrm{Sp}}
\DeclareMathOperator{\capacity}{Cap}
\DeclareMathOperator{\energy}{Enrg}
\newcommand{\MH}{\mathcal{M}_{\mathrm{H}}}
\newcommand*{\dd}{{\,\mathrm{d}}}
\newcommand*{\dist}{{\mathrm{dist}}}
\definecolor{rev1}{HTML}{cb270f}
\definecolor{rev2}{HTML}{1c8235}

\numberwithin{equation}{section}
\newtheorem{theorem}{Theorem}
\newtheorem{lemma}{Lemma}
\newtheorem{corollary}{Corollary}
\newtheorem{proposition}{Proposition}
\newtheorem{definition}{Definition}
\theoremstyle{definition}

\newtheorem{example}{Example}

\numberwithin{theorem}{section}
\numberwithin{lemma}{section}
\numberwithin{corollary}{section}
\numberwithin{proposition}{section}
\numberwithin{definition}{section}
\numberwithin{example}{section}

\title[Computing Spectral Size]{\vspace{-5mm}Computing Spectral Size:\\Rigorous Algorithms and the Limits of Computation}
\date{}

\author{\vspace{-4mm}Matthew J. Colbrook} 
\address{Department of Applied Mathematics and Theoretical Physics, University of Cambridge}
\email{m.colbrook@damtp.cam.ac.uk}

\author{Mark Embree} 
\address{Department of Mathematics, Virginia Tech, Blacksburg, VA 24061, USA}
\email{embree@vt.edu}

\author{Jake Fillman} 
\address{Department of Mathematics, Texas A\&M University, College Station, TX 77845, USA.}
\email{fillman@tamu.edu}

\begin{document}
\vspace{-5mm}

\begin{abstract}\vspace{-2mm}
Many structures in mathematical physics and dynamics exhibit intricate fractal geometry. Such behavior appears prominently in quantum mechanics and materials science through spectra of aperiodic and quasicrystalline operators, where questions of ``size'' (Lebesgue measure, fractal dimension, spectral gaps, etc.) are central. Yet the lack of rigorous computational tools for analyzing these quantities limits both theory and application. Naïve truncation often fails, and there is no overarching framework to explain what can, and cannot, be computed.

We develop a unified program for the rigorous computation of spectral size for bounded self-adjoint operators, based on local spectral exclusions and adaptive covers. This constructive framework yields algorithmically optimal methods (under natural computational assumptions) that bridge spectral theory with computation to address problems previously deemed intractable. Their complexity is classified within the Solvability Complexity Index (SCI) hierarchy, extending Smale's program on the limits of computation. Sharp computational lower bounds are established through impossibility results for limit-periodic Schrödinger operators constructed from adversarial potentials. The methods enable state-of-the-art rigorous computations for one- and two-dimensional aperiodic systems, and pinpoint problems where numerics can feed directly into computer-assisted proofs. Beyond spectral analysis, they apply broadly to computing measures of size for general closed sets, opening new directions in the computational study of complex geometric structures.
\end{abstract}\vspace{-2mm}


\maketitle
\vspace{-10mm}
\renewcommand*\contentsname{Contents\vspace{-3mm}}
\small
\setcounter{tocdepth}{2}
\renewcommand{\baselinestretch}{0.8}
\tableofcontents
\renewcommand{\baselinestretch}{1.1}
\normalsize

\linespread{1.1}

\section{Introduction}\label{intro}
In the 1980s, Shechtman \textit{et al.} \cite{SBGC1984PRL} discovered quasicrystals, an exciting type of matter with unique physical properties.\footnote{\small A natural quasicrystal icosahedrite was discovered in 2009 \cite{bindi2009natural}. Shechtman subsequently received the Nobel Prize in Chemistry in 2011 for work that ``led to a paradigm shift within chemistry.''} 
Unlike ordinary crystals, quasicrystals lack periodic structure; instead, they exhibit emergent hierarchical structures, leading to interesting spectral features linked to distinctive physical properties \cite{dean2013hofstadter,novoselov20162d,bandres2016topological,tanese2014fractal,levi2011disorder,stadnik2012physical}. Applications of quasicrystals now include hydrogen storage for renewable energy \cite{kweon2022quantitative}, superconductivity \cite{kamiya2018discovery},
surfaces and composites \cite{thiel2008quasicrystal},
high-strength materials \cite{liu2022growth},
low-friction materials \cite{yadav2018quasicrystal},
scattering and metamaterials \cite{gerke2010aperiodic,wang2022mechanical}, and
photonics \cite{vardeny2013optics}.
More generally, almost-periodic operators often display a fractal spectrum and spectral properties that are hallmarks of exotic physical behavior, such as unusual transport phenomena where ordinary diffusion fails, the quantum Hall effect, and Cantor-like (fractal) spectra. These systems have garnered considerable interest, and their analysis is highly intricate, requiring sophisticated mathematical tools to prove fundamental results; see, e.g., \cite{avila2009ten, avila2017spectral, avila2006reducibility, AvilaYouZhou2017DMJ, BourgainGold2000Ann, BouJit2002Invent}.

Computational observations have frequently guided theory in this area, raising questions that intrigue both mathematicians and physicists. A classic example is provided by the \emph{almost Mathieu operator} on $\ell^2(\mathbb{Z})$:
\begin{equation}
[H_{\lambda,\alpha,\theta}\,\psi](n)
=\psi(n-1) + \psi(n+1) + 2\lambda\cos(2\pi n\alpha+\theta)\,\psi(n),
\end{equation}
where $\lambda,\alpha,\theta\in\mathbb{R}$. Its spectrum $\mathrm{Sp}_+(\alpha,\lambda)$ forms the Hofstadter butterfly (\cref{AMfig1}, left), one of the most iconic fractal images in physics. The \emph{Aubry--Andr\'e conjecture} \cite{aubry1980analyticity}, that the Lebesgue measure of the spectrum for irrational $\alpha$ equals $4|1 - |\lambda||$, was based on numerical evidence (see the thinning of the spectrum in the right panel of \cref{AMfig1} as $\lambda \uparrow 1$).  Its pursuit sparked significant breakthroughs in mathematics, with contributions from many authors \cite{HelffSjos1989MSMF,avron1990measure,Last1993CMP,last1994zero,JK02} leading to its final resolution in \cite{avila2006reducibility}. This problem is by no means unique in this regard, and further examples are discussed in \cref{sec:math_prelims}.

\begin{figure}
\centering
\raisebox{-0.5\height}{\includegraphics[width=0.48\textwidth,trim={0mm 0mm 0mm 0mm},clip]{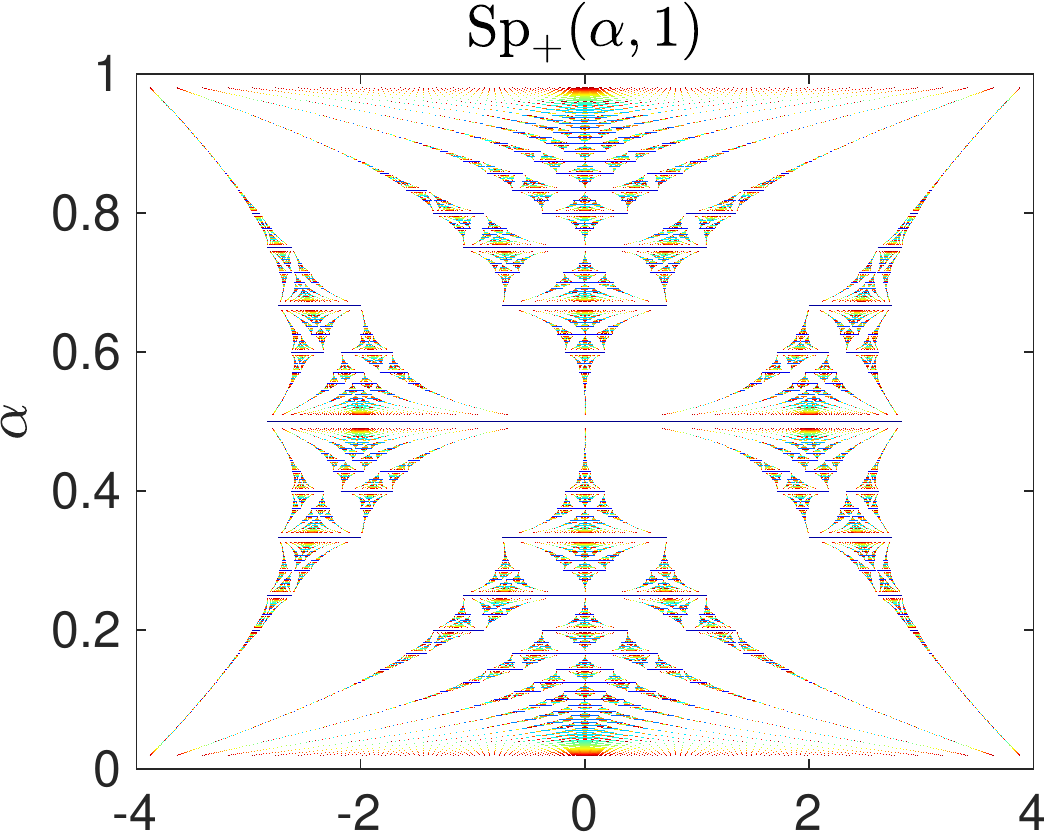}}
\hfill
\raisebox{-0.5\height}{\includegraphics[width=0.48\textwidth,trim={0mm 0mm 0mm 0mm},clip]{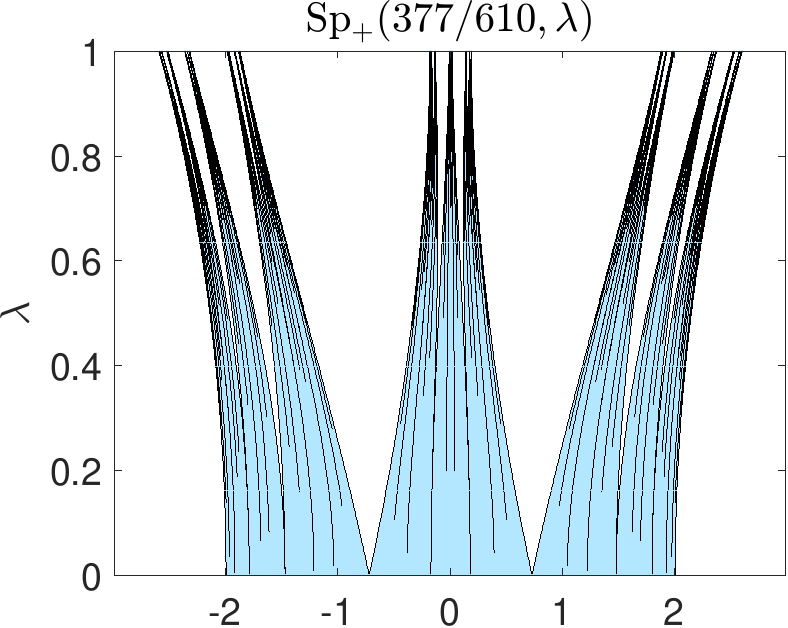}}\vspace{-2mm}
\caption{Left: A Hofstadter butterfly consisting of $\spec_{+}(\alpha,1)$ (union of spectra defined in \eqref{AM_union_spectrum}) as $\alpha$ varies. The colors correspond to the values of $q$ (shown up to $q=51$), where $\alpha=p/q$ with $p$ and $q$ coprime. Right: The set $\spec_{+}(377/610,\lambda)$ for various $\lambda$ and a rational approximation $377/610$ to the irrational number $(\sqrt{5}-1)/2$.}
\label{AMfig1}
\end{figure}

Despite such successes, most computations rely on finite truncations of the underlying operator, and there remains a lack of rigorous computational theory or convergence analysis to guide such delicate spectral calculations. This is particularly true in physical dimensions higher than one, where analytical tools are few and far between. Across the extensive applications literature, practitioners often depend on heuristics rather than guarantees of accuracy or convergence. For instance, \cref{fig:Pen_gaps} illustrates how standard truncation methods fail to capture the spectral gap distribution for Penrose-tile models of quasicrystals, the canonical two-dimensional quasicrystal system. Notably, the first provably convergent algorithm for computing spectra of general two-dimensional quasicrystals appeared only in 2019 \cite{colb1}. This leaves open fundamental questions about which spectral quantities can be computed reliably and how numerical methods can assist in proving theorems, either by providing compelling evidence or through direct computer-assisted proofs. To address these challenges, we develop a rigorous computational framework for quantifying spectral size.

Mathematically, the ``size'' of a set can be quantified in many ways---for example, by its Lebesgue measure (how much length or volume it covers), by its number of connected components or gaps, by its capacity (ability to hold electric charge), or by its fractal dimensions. A classic case is the middle-third Cantor set, which has zero Lebesgue measure but a positive fractal dimension, illustrating how different notions of size can reveal different aspects of a set. In this paper, we study the problem of computing various notions of the size of the spectrum of a bounded self-adjoint operator. Our results are not limited to aperiodic operators: the techniques extend to general operators with applications throughout the sciences, including Krylov methods \cite{nevanlinna2012convergence,nevanlinna1995hessenberg,nevanlinna1990linear,miekkala1996iterative,DTT98}, transport properties of dynamical systems \cite{han1994critical, barbaroux2001fractal, schulz1998anomalous, ketzmerick1992slow, ketzmerick1997determines}, multilayer materials (e.g., bilayer graphene) \cite{dean2013hofstadter,geim2013van,hunt2013massive,ponomarenko2013cloning}, strained materials \cite{naumis2017electronic,roman2014spectral}, and laser stability \cite{rivera2018fractal,new2001diffractive,berry2001fractal,berry2001theory,berry2004physics}. Although our focus is on spectral problems, the methods introduced here provide general tools for computing notions of size for any compact set satisfying the conditions we establish.

\begin{figure}[t]
\centering
\raisebox{-0.5\height}{\includegraphics[width=0.32\textwidth,trim={0mm 0mm 0mm 0mm},clip]{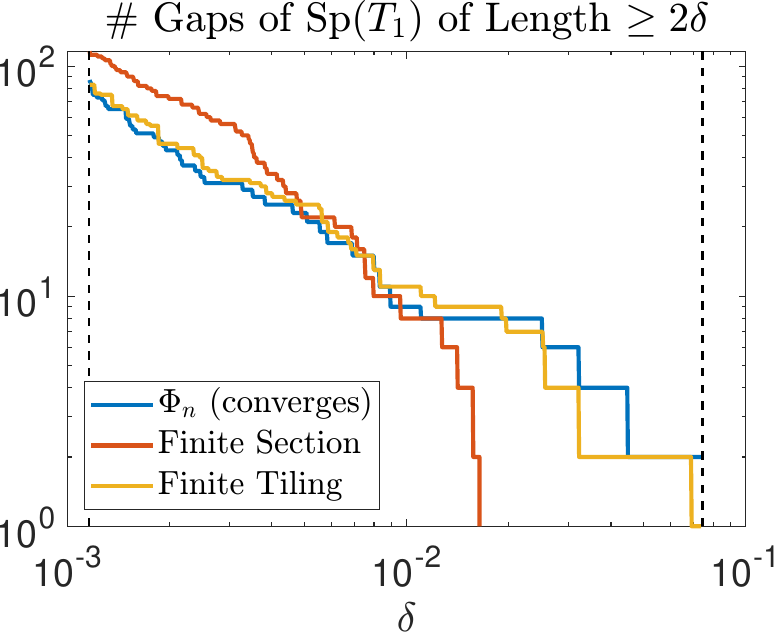}}
\hfill
\raisebox{-0.5\height}{\includegraphics[width=0.32\textwidth,trim={0mm 0mm 0mm 0mm},clip]{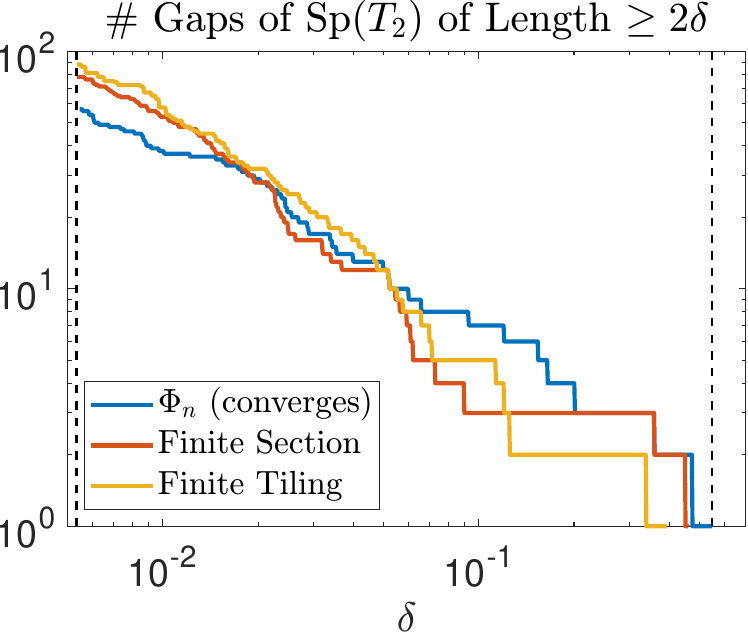}}
\raisebox{-0.5\height}{\includegraphics[width=0.32\textwidth,trim={0mm 0mm 0mm 0mm},clip]{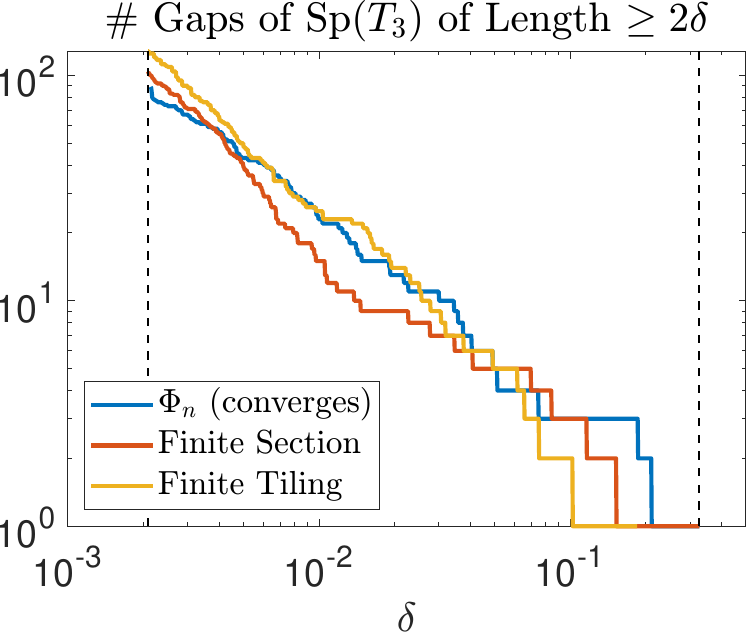}}
\caption{The number of gaps in the spectrum of a given size for the three Penrose tile models considered in \cref{sec:examples2}. The method we propose based on $\Phi_n$ converges, whereas the finite section (truncation of infinite matrix) and tiling (Laplacian of truncated tile) methods have distorted counts due to spectral pollution. The dashed lines show converged regions.}
\label{fig:Pen_gaps}
\end{figure}

Given a notion of size $\mathcal{Q}$ and a class $\Omega$ of bounded self-adjoint operators, we consider the following

\vspace{1mm}

\begin{adjustwidth}{+5mm}{+5mm}
\noindent{}\textbf{Open problem:} \textit{Are there algorithms that approximate $\mathcal{Q}(\mathrm{Sp}(A))$ for all $A\in\Omega$?
If so, what algorithms are ``optimal''? How do these algorithms exploit possible additional structure on $\Omega$?}
\end{adjustwidth}

\vspace{1mm}

\noindent{}However, to understand what can be computed and what cannot, we need a framework from computational theory. To resolve this problem, we use the Solvability Complexity Index (SCI) hierarchy \cite{ben2020can,Hansen_JAMS,colbrook2022foundations}. In simple terms, the SCI hierarchy classifies computational problems by how many limits or iterations an algorithm needs to converge to the answer.

The SCI hierarchy has recently been used to resolve the classical problem of computing spectra of general bounded operators on separable Hilbert spaces \cite{ben2020can,Hansen_JAMS}, which originated from work by Szeg{\H o} \cite{Szego} on finite section approximations and Schwinger \cite{Schwinger} on finite-dimensional approximations to quantum systems. Surprisingly, computing the spectrum often requires several successive limits, as exemplified by the algorithm $\Gamma_{n_3,n_2,n_1}$ from \cite{Hansen_JAMS}, which converges to the spectrum in the Hausdorff metric:
\begin{equation}
\label{SCI_mult_lim_example}
\lim_{n_3\rightarrow\infty}\lim_{n_2\rightarrow\infty}\lim_{n_1\rightarrow\infty}\Gamma_{n_3,n_2,n_1}(A)=\mathrm{Sp}(A)\quad\forall A\in\mathcal{B}(\ell^2(\mathbb{N})),
\end{equation}
where $\mathcal{B}(\ell^2(\mathbb{N}))$ denotes the class of bounded operators on $\ell^2(\mathbb{N})$. It is impossible to reduce the number of successive limits via any algorithm \cite{ben2020can}.
Traditional methods often rely on single-limit techniques (e.g., computing eigenvalues of increasingly large finite-dimensional matrices generated by reducing a mesh size, or increasing a truncation dimension). In contrast, the necessity of multiple limits reflects the intrinsic analytical complexity of infinite-dimensional spectral computations: it guides the development of nuanced algorithms that navigate the three-limit structure, and yields an infinite classification theory. For some subclasses of operators, one can design single-limit algorithms that converge to the spectrum with explicit error control \cite{colb1,colbrook2022foundations}. The SCI hierarchy thus delineates which computational problems are solvable, for which classes of operators, and with what types of algorithms, thus identifying the precise boundaries of computability and algorithmic optimality (i.e., the lowest SCI class necessary for the problem, also in certain cases the SCI hierarchy has connections with classical hierarchies \cite{colbrook2022computation}).

Classifying broader spectral problems (such as the above open problem) and developing a comprehensive library of optimal algorithms remain largely unexplored. Navigating this landscape requires a diverse set of methods to accommodate the wide range of structures and spectral behaviors exhibited by different classes of operators. Each structure presents distinct opportunities for tailored computational approaches. By integrating analytical results from spectral theory (e.g., \cite{avron1990measure,damanik2008fractal,Suto1987CMP}) into the SCI framework, we classify the complexity of computing spectral size and design new algorithms that determine these quantities rigorously, paving the way for their broad application. These tools also provide a foundation for formulating new mathematical conjectures (\cref{sec:conjectures}) and, in suitable settings, for establishing theorems via computer-assisted proofs. For instance, see~\cite{DEFM202XUEXM} for a catalogue of flat bands in certain two-dimensional quasicrystals. Our hope is that these techniques will not only advance spectral theory and mathematical physics, but also provide new tools for computational mathematics at large, potentially impacting areas like dynamical systems, optimization, and even the emerging interface of AI and mathematics.

\subsection{Contributions and roadmap}

In \cref{sec:math_prelims}, we provide the necessary mathematical preliminaries, before detailing two key scenarios for bounded self-adjoint operators that underpin our computational assumptions, and extend beyond the spectral applications upon which we focus.
\begin{itemize}[leftmargin=12mm]
	\item[\textbf{(S1)}:] In the first scenario, we assume that one can compute, for each $n$, a finite union of disjoint compact real intervals
\[
\Gamma_n^{\spec}(A)=\bigcup_{m=1}^{M_n(A)}[a_{m,n}(A),b_{m,n}(A)]
\]
such that $\spec(A)\subset \Gamma_n^{\spec}(A)$ and each interval intersects the spectrum. Crucially, we require these covers to come with explicit, computable error bounds 
\[
d_{\mathrm H}(\Gamma_n^{\spec}(A),\spec(A)) \le E_n^{\spec}(A)
\]
satisfying $E_n^{\spec}(A)\rightarrow 0$ as $n\rightarrow\infty$,
where $d_{\mathrm H}(\cdot,\cdot)$ denotes the Hausdorff distance (see the discussion surrounding \cref{gamma_cover1}).
Thus \textbf{(S1)} requires certified two-sided Hausdorff control of the spectrum by interval enclosures. Motivated by the classical construction of Cantor sets via nested covers, this requirement is demanding for general self-adjoint operators. However, it is satisfied by many important one-dimensional quasicrystal models, where such interval covers already play a central role in their analysis. We show that these dynamically generated covers are also highly effective for spectral computation. 
\item[\textbf{(S2)}:] In the second scenario, we assume that one can compute functions $\Phi_n(z,A)$ satisfying
\[
\mathrm{dist}(z,\spec(A)) \le \Phi_n(z,A),
\qquad
\lim_{n\to\infty}\Phi_n(z,A)=\mathrm{dist}(z,\spec(A)),
\]
with uniform convergence in $z$ on compact subsets of $\mathbb{R}$.\ \  Thus \textbf{(S2)} provides certified upper bounds on the distance-to-spectrum that converge to the exact distance, but without producing enclosing interval covers or explicit Hausdorff error control. This property holds for a wide class of operators, though achieving it typically requires more refined techniques than finite section methods, which can suffer from spectral pollution and thereby violate the required upper-bound property.
\end{itemize}

Typically, \textbf{(S1)} requires strong global assumptions about the operators, whereas \textbf{(S2)} relies only on local knowledge of the operators, e.g., in cases where infinite matrices represent them (see \cref{sec:bounded_disp}). This distinction is reflected in the SCI classifications of the spectral quantities we consider. The two scenarios \textbf{(S1)} and \textbf{(S2)} have applications beyond the computational problems addressed in this paper, and unlock the door to a menagerie of problems associated with spectra and, more generally, compact sets \cite{colbrook2020PhD}.

In \cref{sec:Imethods}, we develop algorithms under the \textbf{(S1)} framework. We begin by demonstrating that many properties of interest remain continuous under limits of \textbf{(S1)} covers. This fact seems to have gone largely unnoticed in the spectral theory community, and we exploit it to compute properties like the Lebesgue measure of spectra with error control. However, when dealing with fractal dimensions (which are not continuous under limits of \textbf{(S1)} covers), it is crucial to connect the resolution of spectral approximations to the diameters of mesh coverings. We show how this can be done for the box-counting and Hausdorff dimensions, leading to the first practical and rigorously convergent algorithms for these challenging quantities. The section concludes with a lower bound in the SCI hierarchy -- an \textit{impossibility theorem} -- using the class of limit-periodic Schr\"odinger operators to prove the sharpness (optimality) of our algorithms, thus marking the boundary of what can be achieved under these conditions. This proof is based on the explicit construction of an adversarial potential.

\Cref{sec:examples1} then showcases state-of-the-art computations for aperiodic Schr\"odinger operators on $\ell^2(\mathbb{Z})$. These results uncover previously unseen spectral properties, offering new insights into the behavior of these operators.

\begin{table}[t!]
\caption{Summary of the main results, with theorem numbers in brackets. The various classes are defined in \cref{subsec:SCI} (for the $\Pi$ and $\Sigma$ classes, the subscript denotes the number of limits and the final limit is monotonic, for the $\Delta$ classes, the subscript denotes one more than the number of limits). Each classification has a corresponding lower bound (e.g., \cref{sec:limit_per_lower}). In general, moving from \textbf{(S1)} to \textbf{(S2)} costs a limit. At the time of writing, \textbf{(S1)} applies mainly to one-dimensional models, whereas \textbf{(S2)} applies mainly to higher-dimensional models. $^*$The specific decision problem ``Is $\spec(A)$ an interval?'' gives an SCI one lower than the general case (i.e., $\Pi_1^A$ for \textbf{(S1)} and $\Pi_2^A$ for \textbf{(S2)}).} 
\label{SCI_tab}
\renewcommand{\arraystretch}{1}
\begin{center}
\begin{tabular}{|c|c|c|}
\hline
\textbf{Problem}&\textbf{Classification under (S1)}&\textbf{Classification under (S2)}\\
 \hline
 \hline
\# connected components & $\Sigma_1^A$ (\cref{thm:cor_assumI_lims}) & $\Sigma_2^A$ (\cref{cor_assumII_extra_lims})\\
Lebesgue measure & $\Pi_1^A$ (\cref{thm:cor_assumI_lims}) & $\Pi_2^A$ (\cref{cor_assumII_extra_lims})\\
Capacity & $\Pi_1^A$ (\cref{thm:cor_assumI_lims}) & $\Pi_2^A$ (\cref{cor_assumII_extra_lims})\\
Decision problems$^*$ & $\Pi_2^A$ (\cref{thm:cor_assumI_lims}) & $\Pi_3^A$ (\cref{cor_assumII_extra_lims})\\
Lower box-counting dim. & $\Sigma_2^A$ (\cref{thm:box_counting}) & $\Sigma_3^A$ (\cref{fractal_dim_final})\\
Upper box-counting dim. & $\Pi_2^A$ (\cref{thm:box_counting}) & $\Pi_2^A$ (\cref{fractal_dim_final})\\
Box-counting dimension & $\Delta_2^A$ (\cref{thm:box_counting}) & $\Pi_2^A$ (\cref{fractal_dim_final})\\
Hausdorff dimension & $\Sigma_2^A$ (\cref{thm_Haus}) & $\Sigma_3^A$ (\cref{fractal_dim_final})\\
\hline
\end{tabular}
\end{center}
\renewcommand{\arraystretch}{1.1}
\end{table}

In \cref{sec:IImethods}, we develop algorithms under the \textbf{(S2)} framework and introduce the \textit{Swiss cheese method}. Our approach computes spectral covers that meet the \textbf{(S1)} conditions but at the cost of an additional limiting process. These covers are constructed by removing open intervals from an initial interval, akin to the Swiss cheese sets used in approximation theory \cite{roth1938approximationseigenschaften}. The positions and diameters of these ``holes'' are dictated by the functions that satisfy \textbf{(S2)}. In essence, we approximate the spectrum by identifying points (or gaps) that do not belong to it. A major challenge arises in aperiodic systems of dimension higher than one, where techniques such as transfer matrices are absent. Our method allows us to extend the results from \cref{sec:Imethods}, albeit with an increase in the SCI classification by one level (i.e., adding one more limit). In \cref{sec:bounded_disp}, we demonstrate that many operator classes, including higher-dimensional quasicrystal models, satisfy the \textbf{(S2)} condition. This fact allows us to rigorously and practically compute their spectral properties. Surprisingly, to meet the requirements of \textbf{(S2)}, we rely on singular values rather than eigenvalues.

Finally, in \cref{sec:examples2}, we apply our algorithms to Penrose tile models of quasicrystals, demonstrating how standard truncation methods (e.g., finite section techniques) can fail by producing misleading results, while careful computations reveal new spectral properties. For example, \cref{fig:Pen_gaps} demonstrates how truncation methods typically fail to correctly compute the spectral gap distribution, whereas our new method achieves convergence.

In summary, we establish optimal, convergent algorithms for computing spectral size and classify the difficulty of these problems, as shown in \cref{SCI_tab}. These algorithms are also practical, enabling state-of-the-art numerical experiments on one- and two-dimensional quasicrystals. Their design clarifies the assumptions required for different models. We conclude the paper with further remarks and conjectures motivated by our computational examples.

\subsection{Connections with previous work}\label{sec:prev_work}

A motivation for this paper is the discovery by Chandler--Wilde, Chonchaiya, and Lindner \cite{chandler2024spectral} that banded operators on $\ell^2(\mathbb{Z})$ admit certain spectral covers under suitable conditions, leading to algorithms for computing spectra and rigorous spectral exclusions. This result provides a concrete example of \textbf{(S1)}. The authors of \cite{chandler2024spectral} assume access to the collection of all finite patches of a given size for the operator. Their results can be viewed as an extension of Gershgorin's theorem \cite{gershgorin1931uber} to a family of inclusion sets for the spectra of bi-infinite tridiagonal matrices. In \cref{sec:examples1}, we explore different techniques to produce covers for one-dimensional aperiodic operators. Our approximations are typically built around periodic approximations of the operators. We also refer the reader to the work of Ben-Artzi, Marletta, and R\"osler on computing scattering resonances via spectral covers \cite{Jonathan_res,benartzi2020computing}.

An exciting problem arises from our work: constructing \textbf{(S1)} covers for higher-dimensional quasicrystal models. Recently, Hege, Moscolari and Teufel \cite{hege2026computing} have shown that for discrete short-range operators with finite local complexity, the spectrum of normal operators admits two-sided error control (i.e., $\Delta_1^A$-computability, defined in Section~\ref{sec:math_prelims}) provided one is given local patch information. Their results therefore establish a rigorous pathway to \textbf{(S1)} in this setting, and this program has been carried out explicitly for the two-dimensional Ammann--Beenker tiling via a cut-and-project construction~\cite{PhysRevB.106.155140}. In view of this work, the remaining challenge is not the abstract existence of such covers, but their explicit construction and effective implementation for concrete two-dimensional quasicrystal Hamiltonians of physical interest. In particular, developing computationally sharp and practically implementable \textbf{(S1)} covers for models such as the Penrose tiling—analogous to what has been achieved for Ammann--Beenker—would move these classes from \textbf{(S2)} to \textbf{(S1)} within our framework.

When the finite section method is successful for computing spectra, it typically provides no error control or verification. Examples of such work include studies by B{\"o}ttcher, Brunner, Iserles \& N{\o}rsett \cite{Arieh2}; B{\"o}ttcher, Grudsky \& Iserles \cite{bottcher_grudsky_iserles_2011}; Marletta \cite{Marletta_pollution}; and Marletta \& Scheichl \cite{marletta2012eigenvalues}. These latter works also discuss when the finite section method fails. When applied to the more complicated spectral features studied in this paper, the finite section method rarely converges due to issues such as spectral pollution. See also the work of B{\"o}ttcher~\cite{Albrecht_Fields,Bottcher_pseu}, B{\"o}ttcher \& Silbermann~\cite{Bottcher_book,bottcher2006analysis}, Laptev \& Safarov~\cite{Laptev}, and Brown~\cite{brown2007quasi,Brown_2006,Brown_Memoars}. B{\"o}ttcher \& Silbermann~\cite{Albrecht1983} pioneered the combination of spectral computation and $C^*$-algebras.

Olver, Townsend, and Webb have established a critical framework for infinite-dimensional numerical linear algebra and computational methods for infinite data structures, contributing both theoretical insights and practical algorithms \cite{Olver_Townsend_Proceedings,Olver_SIAM_Rev,webb_thesis,webb2021spectra}. Recent work has also begun to explore the locality and computability of spectral or topological invariants in non-periodic quantum systems \cite{lu2022existence,loring2024locality,lu2024algebraic}. The paper \cite{colbrook2022computation} addressed similar questions for different classes of operators, most notably diagonal operators, using different techniques. By linking \textbf{(S1)} and \textbf{(S2)} through a covering lemma (\cref{lem:cover1}), our proofs provide deeper insights into the underlying structures and introduce a systematic approach for designing algorithms. Notably, the techniques developed in this paper are distinct: the methods in \cite{colbrook2022computation} do not yield optimal or practical algorithms for quasicrystal models.

The multiple limit structure in \cref{SCI_mult_lim_example} appears in other areas of computational mathematics. An early instance is Smale's polynomial root-finding problem with rational maps \cite{smale_question}, which involves several successive limits as shown by McMullen \cite{McMullen1,mcmullen1988braiding} and Doyle \& McMullen~\cite{Doyle_McMullen}. These results can be expressed in terms of the SCI hierarchy \cite{ben2020can}, which generalizes Smale's seminal work on the foundations of scientific computing and the existence of algorithms \cite{smale1981fundamental,Smale_Acta_Numerica,BCSS}. The SCI is now being used to explore the foundations of computation in diverse areas, including resonance \cite{Jonathan_res,benartzi2020computing}, inverse \cite{AdcockHansenBook}, and optimization \cite{SCI_optimization} problems, the foundations of AI \cite{colbrook2022difficulty}, and data-driven dynamics \cite{colbrook2024limits}.

The SCI hierarchy is further motivated by computer-assisted proofs, which have become an essential tool in modern mathematics \cite{AIM}. Perhaps unexpectedly, non-computable problems, such as those arising in the Dirac--Schwinger conjecture on the asymptotic behaviour of ground states of certain Schrödinger operators \cite{fefferman1990,fefferman1992,fefferman1996interval} or in the proof of Kepler's conjecture (Hilbert’s 18th problem) \cite{hales_Pi,Hales_Annals}, play a central role in these developments. The SCI hierarchy helps explain this apparent paradox, particularly through the classes $\Sigma^A_1$ and $\Pi^A_1$ (i.e., verifiable classes) described below. Many of the problems we study in this paper also fall within $\Sigma^A_1 \cup \Pi^A_1$, indicating that they are suitable candidates for computer-assisted proofs.

\subsection*{Acknowledgements}
We thank the Cecil King Foundation and
the London Mathematical Society for funding a trip of Colbrook to Virginia Tech, which initiated his project. Fillman was supported in part by National Science Foundation grants DMS-2213196 and DMS-2513006 and by Simons Foundation Grant MPS-TSM 00013720.
Embree and Fillman thank the American Institute of Mathematics for hospitality during a recent SQuaRE program, which facilitated related work on quasiperiodic models. We are grateful to the referees for their many helpful suggestions.

\section{Mathematical Preliminaries}
\label{sec:math_prelims}

\subsection{A primer on the solvability complexity index} \label{subsec:SCI}

What does it mean to compute a quantity in mathematics? Computation is often interpreted as providing an expression in closed form, perhaps involving some commonly accepted ``elementary'' functions or operations. However, this interpretation is overly restrictive. In many instances, the class of solutions of interest cannot be expressed in such a manner. For example, a general quintic equation cannot be solved using only roots and arithmetic operations, and many analytic functions lack ``elementary'' antiderivatives. Despite these limitations, these quantities can be calculated to any desired precision almost instantaneously.
A more robust and useful notion of computation should focus on \emph{algorithms}.

The Solvability Complexity Index (SCI) hierarchy~\cite{colbrook2020PhD,colbrook2022foundations,Hansen_JAMS}~\cite[Chapter 2]{colb_book} provides a framework for classifying the difficulty of computing various quantities and proving the optimality of algorithms. The SCI sets expectations for the relative difficulty of such problems (e.g.,  comparing the computation of the measure of the spectrum to its fractal dimension) and highlights opportunities for computations to be expedited by incorporating the structure associated with a given family of operators. We will see an example in this paper: for discrete Schr\"odinger operators with certain aperiodic potentials, one can exploit the operator's structure to obtain a convergent upper bound on the spectrum. The availability of such an upper bound elevates derived spectral questions to a more favorable SCI, making computationally tractable some problems that are generally much more difficult to tackle.

Throughout this paper, we use the notation of the SCI.\ \ 
We briefly introduce the SCI hierarchy to establish this notation and provide context for our results. We begin with a formal definition of a computational problem, discuss the structure of algorithms we can use to approximate a solution, and then explain how to classify the problem's difficulty.\vspace{2mm}

\begin{definition}[Computational problem]
\label{CHAP2_def:comp_prob}
A computational problem is a collection $\{\Xi,\Omega,\mathcal{M},\Lambda\}$ consisting of:
\begin{itemize}[leftmargin=0.7cm]
	\item An \textit{input class} (set of input instances), $\Omega$;
	\item A \textit{metric space} $(\mathcal{M},d)$;
	\item A \textit{problem function} $\Xi:\Omega\rightarrow\mathcal{M}$;
	\item An \textit{evaluation set}, $\Lambda$, of complex-valued functions on $\Omega$.
\end{itemize}
The evaluation set $\Lambda$ must separate elements of $\Omega$, to the degree of separation achieved by $\Xi$:
\begin{equation}
\label{CHAP2_eq:Lambda_determines_it}
\text{if }A,B\in\Omega\text{ with }f(A)= f(B)\,\,\forall f\in\Lambda, \text{ then }\Xi(A)= \Xi(B).
\end{equation}
In other words, any two inputs that cannot be distinguished by evaluation functions in $\Lambda$ have the same output under $\Xi$. (Otherwise, it is impossible to recover $\Xi$ from $\Lambda$.)
\end{definition}

The problem function $\Xi$ specifies the quantity to be computed, and the metric $d$ specifies what it means to approximate $\Xi(A)$. The evaluation set $\Lambda$ encodes the information about the input $A\in\Omega$ that an algorithm may query. That is, given $A\in\Omega$, we seek to compute an approximation of $\Xi(A)\in\mathcal{M}$ in the sense of the distance $d$, using an algorithm that can only access information in $\Lambda$.

\begin{example}[Lebesgue measure of the spectrum]
\label{CHAP2_ex:measure_problem}
Let $\Omega$ be a class consisting of bounded self-adjoint operators on $\ell^2(\mathbb{Z})$. Consider the computational problem
\[
\Xi(A)=|\spec(A)|,
\qquad 
\mathcal{M}=\mathbb{R}_{\ge 0},
\qquad 
d(x,y)=|x-y|,
\]
where $|\spec(A)|$ denotes the Lebesgue measure of the compact set $\spec(A)\subset\mathbb{R}$. A canonical evaluation set is given by the matrix-element functionals
\[
\Lambda=\{f_{j,k}:\Omega\to\mathbb{C}\}_{j,k\in\mathbb{Z}},
\qquad 
f_{j,k}(A)=\langle Ae_k,e_j\rangle,
\]
where $\{e_j\}_{j\in\mathbb{Z}}$ is the standard basis of $\ell^2(\mathbb{Z})$.
If instead $\Omega$ is restricted to discrete Schr\"odinger operators
$
(Au)_j = u_{j-1}+u_{j+1}+V_j u_j,
$
then $A$ is determined by the potential sequence $\{V_j\}_{j\in\mathbb{Z}}$, and it suffices to take
\[
\Lambda=\{f_j:\Omega\to\mathbb{C}\}_{j\in\mathbb{Z}},
\qquad 
f_j(A)=\langle Ae_j,e_j\rangle = V_j.
\]
In both cases, the separation condition \eqref{CHAP2_eq:Lambda_determines_it} is satisfied: if two operators cannot be distinguished by the evaluation set $\Lambda$, then they coincide within the prescribed input class $\Omega$, and hence have the same spectrum and the same Lebesgue measure of the spectrum.\hfill$\blacksquare$
\end{example}

An \emph{algorithm} for the problem $\{\Xi,\Omega,\mathcal{M},\Lambda\}$ is then a procedure that, given oracle access to the values $\{f(A):f\in\Lambda\}$ for $A\in\Omega$, outputs an element of $\mathcal{M}$ intended to approximate $\Xi(A)$ (with accuracy measured in the metric $d$).

\begin{definition}[General algorithm]
\label{CHAP2_def:Gen_alg_NEW}
A general algorithm $\Gamma:\Omega\rightarrow \mathcal{M}$ for solving the  problem $\{\Xi,\Omega,\mathcal{M},\Lambda\}$ is a map with the following property. For every $A\in\Omega$, there exists a non-empty finite subset of evaluations $\Lambda_\Gamma(A) \subset\Lambda$ such that if $B\in\Omega$ satisfies $f(A)=f(B)$ for every $f\in\Lambda_\Gamma(A)$, then $\Lambda_\Gamma(A)=\Lambda_\Gamma(B)$ and $\Gamma(A)=\Gamma(B)$.
\end{definition}

An algorithm can only access any $A \in \Omega$ through a \emph{finite number} of evaluation functions $f \in \Lambda_\Gamma(A) \subset \Lambda$. The set $\Lambda_\Gamma(A)$ could be adaptively chosen on the fly by the algorithm and its size could vary significantly with the input $A$. However, it must be the case that if $f(A)=f(B)$ for all $f\in\Lambda_\Gamma$, then $\Gamma(A)=\Gamma(B)$. We emphasize that the set $\Lambda$ plays an important role, defining realistic ways in which the information in each $A\in\Omega$ can be accessed.  In particular, both $\Lambda$ and $\Omega$ can incorporate special structure.  For example, if $\Omega$ is the set of banded Laurent operators on $\ell^2(\mathbb{Z})$, then $\Lambda$ would naturally give access to the symbol function for the operator~\cite{Bottcher_book}.  For more general bounded linear operators on $\ell^2(\mathbb{Z})$, $\Lambda$ could give access to samples $f_{j,k}(A) = \langle A e_k, e_j \rangle$ only for $-n \leq j,k \leq n$, known as the \emph{finite section method}.  In this case no single algorithm (e.g., a finite section for one value of $n$) will enable the computation of the spectrum; one or more \emph{limits} of algorithms are necessary.  The extent of such limits is described by \emph{towers of algorithms}.

\begin{definition}[Tower of algorithms]
A tower of algorithms of height $k$ for a computational problem $\{\Xi,\Omega,\mathcal{M},\Lambda\}$ is a collection of functions 
$$
\Gamma_{n_k, \ldots, n_1},\,
\Gamma_{n_k, \ldots, n_2},\, \ldots,\,\Gamma_{n_k}:\Omega \rightarrow \mathcal{M}, \quad n_k,\ldots,n_1 \in \mathbb{N},
$$
where the lowest level $\{\Gamma_{n_k, \ldots, n_1}\}$ comprises general algorithms (\cref{CHAP2_def:Gen_alg_NEW}) and for every $A \in \Omega$, the following convergence holds in $(\mathcal{M},d)$:
\begin{align*}
\lim_{n_1 \rightarrow \infty} \Gamma_{n_k, \ldots, n_1}(A)&=\Gamma_{n_k, \ldots, n_2}(A),\quad
&&\lim_{n_2 \rightarrow \infty} \Gamma_{n_k, \ldots, n_2}(A)=\Gamma_{n_k, \ldots, n_3}(A),\\
&\ldots,\quad
&&\lim_{n_k \rightarrow \infty} \Gamma_{n_k}(A)=\Xi(A).
\end{align*}
\end{definition}

When building such algorithms, one typically expects each of the index variables $n_1,\ldots,n_k$ to reflect a distinct limiting procedure, such as the size of a finite-dimensional truncation, or a grid width used for approximating the box-counting dimension (see Section~\ref{ssec:S1:bcd}). However, for our impossibility results, these indices can represent anything.  Each higher level in the tower consists of algorithms that result from taking the limit to infinity of each one of these indices, and thus, in principle, each of these algorithms can access infinite data.

We discuss two different approaches, distinguished in the SCI categories by a superscript $\alpha$ (e.g., $\Delta_0^\alpha$ below):  when $\alpha=A$, the algorithms in question use a practical model of real arithmetic and comparison operations; when $\alpha=G$, the algorithms may use more powerful, general operations.
To obtain the strongest theorems possible, we use the realistic $\alpha=A$ model to construct practical algorithms that give upper bounds on complexity.  (Specifically, we use the model of Smale et al.~\cite{BCSS}, though other models of arithmetic (e.g.,  \cite{tucker2011validated, turing1937computable}) are also possible; see~\cite[Section~2.3]{colbrook2020PhD} for details.)
In contrast, we use the more powerful $\alpha=G$ model to give rigorous lower bounds.
The proofs of impossibility results involve analysis of generic operators in the class $\Omega$, attributing the non-computability to fundamental mathematical properties at the interface of the problem and the operator class, rather than the limitations of any given model of computation for the operations themselves.

\vspace{2mm}

\noindent\textbf{SCI hierarchy:} The SCI for a problem $\{\Xi,\Omega,\mathcal{M},\Lambda\}$ is determined in terms of the tower of algorithms required to solve it.
\begin{itemize}[leftmargin=4mm]
\item If the problem can be solved exactly in finitely many operations, then its SCI is zero; we write $\{\Xi,\Omega,\mathcal{M},\Lambda\} \in \Delta_0^\alpha$.
\item If the problem can be solved with a set of algorithms
$\{\Gamma_n\}$ such that $d(\Gamma_n(A),\Xi(A)) \le 2^{-n}$ for all $A\in\Omega$
(i.e., we can compute $\Xi(A)$ in one limit with full control of the error),
we write $\{\Xi,\Omega,\mathcal{M},\Lambda\} \in \Delta_1^\alpha$. (Naturally, by taking a subsequence of $\{\Gamma_n\}$ we can replace $\{2^{-n}\}$ with any convergent sequence while maintaining the same $\Delta_1^\alpha$ classification.)
\item If the problem can be solved for all $A\in\Omega$
with a tower of algorithms of height $k \geq 1$ or less, we write 
$\{\Xi,\Omega,\mathcal{M},\Lambda\} \in \Delta_{k+1}^\alpha$.
\end{itemize}
Note that if a problem lies in $\Delta_2^\alpha$, then it can still be computed via a one-limit procedure, but not necessarily with error control in the sense of $\Delta_1^\alpha$.

When $(\mathcal{M},d)$ is totally ordered, we shall also find it useful to consider problems that admit towers of algorithms that converge to $\Xi(A)$ from below ($\Sigma_k^\alpha$) and above ($\Pi_k^\alpha$).\footnote{\small One can also define these notions for convergence in the Hausdorff metric, for example, when computing the spectrum \cite[Section 2.2]{colbrook2020PhD}. However, these are not needed in this paper.}
\begin{itemize}[leftmargin=4mm]
\item We define $\Sigma_0^\alpha = \Pi_0^\alpha = \Delta_0^\alpha$.
\item For $k\ge 1$, a problem $\{\Xi,\Omega,\mathcal{M},\Lambda\}$ is said to be in 
$\Sigma_k^\alpha$ if it is in $\Delta_{k+1}^\alpha$ and there exists a tower of
algorithms $\{\Gamma_{n_k,\ldots,n_1}\}$ such that 
$\Gamma_{n_k}(A) \uparrow \Xi(A)$ for all $A \in \Omega$.
\item For $k\ge 1$, a problem $\{\Xi,\Omega,\mathcal{M},\Lambda\}$ is said to be in 
$\Pi_k^\alpha$ if it is in $\Delta_{k+1}^\alpha$ and there exists a tower of
algorithms $\{\Gamma_{n_k,\ldots,n_1}\}$ such that 
$\Gamma_{n_k}(A) \downarrow \Xi(A)$ for all $A \in \Omega$.
\end{itemize}
\noindent{}These classes capture notions of verification. Even though problems in $(\Sigma_1^\alpha \cup \Pi_1^\alpha) \backslash \Delta_1^\alpha$ are considered non-computable, they can still be used in computer-assisted proofs. For example, consider a conjecture stating that $|\spec(A)| < 1$ for an operator $A$.\ \ If we use a $\Pi_1$-tower $\{\Gamma_n\}$ and the theorem is true, then $\Gamma_n(A) < 1$ for sufficiently large $n$. As soon as we observe $\Gamma_n(A) < 1$, we can conclude that the theorem must be true.

The SCI of a given problem measures its complexity: it helps us gauge the relative difficulty of problems that might otherwise appear similar and understand how some extra structure might accelerate the solution of a problem by lowering its SCI.\ \ For instance, one of the beautiful aspects of $\Pi_1^A$ algorithms for the spectrum is that they lower the SCI for some of the problems we discuss in this paper. At the start of the collaboration between the authors, the calculus of the SCI (in this case, bringing together $\Sigma_1$ and $\Pi_1$ algorithms for the same class of problems, and noting $\Sigma_1\cap\Pi_1=\Delta_1$ \cite[Proposition 2.2.8]{colbrook2020PhD}) made spotting this opportunity to lower the SCI to $\Delta_1$ immediate.

\subsection{Assumptions on computing spectra}
\label{sec:assumptions_on_spec}

The kinds of algorithms we discuss are motivated by the following fundamental questions about the \emph{size} of compact sets $S\subset \mathbb{R}$.
\begin{itemize}[leftmargin=0.7cm]
\item What is the Lebesgue measure of $S$? Is it zero?
\item What are the fractal dimensions of $S$? Are they zero? Are they one?
\item What is the logarithmic capacity of $S$? Is it zero?
\item Does $S$ possess gaps (bounded connected components of $\mathbb{R} \setminus S$) or is it an interval?
\item Are there infinitely many gaps? Are they dense? Is $S$ a \emph{Cantor set} (a perfect, nowhere dense set)?
\end{itemize}
The computational assumptions \textbf{(S1)} and \textbf{(S2)} we make below can be applied to arbitrary compact $S\subset\mathbb{R}$. Due to its richness and variety, we illustrate these problems on spectral theory, where
\begin{equation}
S=\spec(A) = \{ \lambda \in \mathbb{C} : A - \lambda I \text{ is not boundedly invertible} \},
\end{equation}
for a bounded self-adjoint operator $A$.\ \ Problems related to the spectrum's Cantor nature, Lebesgue measure, and fractal dimensions are extremely well-studied, particularly for aperiodic operators. For example, Kac's \textit{Ten Martini Problem}, that the spectrum of the almost Mathieu operator is a Cantor set for all irrational frequencies, was conjectured by Azbel \cite{azbel1964energy}. It attracted a host of numerical and analytical work before being proven \cite{avila2009ten}. Another well-studied operator is the Fibonacci Hamiltonian \cite{sutHo1989singular}, which has a Cantor spectrum of zero Lebesgue measure. For further examples of operators and numerical approximations of the Lebesgue measure, see the references in \cite{avila2017spectral,benza1991band,sire1989electronic}. Similarly, problems related to the existence and finiteness of the number of spectral gaps have been examined extensively, with recent work emphasizing the use of topological obstructions to spectral gap formation, e.g., \cite{ADG2023GAFA, DamFil2023CMP}. In many examples of interest in mathematical physics, $A$ is a \emph{Schr\"odinger operator}. The notions of size $\mathcal{Q}(\mathrm{Sp}(A))$ have been studied extensively over the years for Schr\"odinger operators \cite{DamanikFillman2022ESO}. Many of these examples have underlying Markov models or so-called dynamically defined potentials, where $V(n)=f(T^n\omega)$ for an invertible transformation $T$ on some space $\mathcal{X}$ and $f:\mathcal{X}\rightarrow\mathbb{R}$. This connection with dynamical systems is extremely fruitful, allowing the unified study of diverse systems such as random operators, almost periodic operators, periodic operators, limit-periodic operators, and quasiperiodic operators, each corresponding to a suitable choice of underlying ergodic dynamics.

We consider classes, $\Omega$, of bounded self-adjoint operators and an appropriate set of evaluation functions $\Lambda$. All the algorithms we build are local and can be extended to general self-adjoint operators, where the Hausdorff metric space $(\mathcal{M}_{\mathrm{H}},d_{\mathrm{H}})$ of non-empty compact subsets of $\mathbb{C}$ is replaced by the Attouch--Wets metric space of non-empty closed subsets of $\mathbb{C}$. One can then compute the quantities locally over the spectrum. We focus on two scenarios of interest that correspond to different SCI classifications of computing the spectrum.
\begin{itemize}
	\item[\textbf{(S1)}] \textbf{\emph{Spectral Covers.}} 
We assume that, using $\Lambda$, we can compute $\{\Gamma_n^\spec\}$, which is a $\Delta_1^A$-tower for $\{\spec,\Omega,\MH,\Lambda\}$, along with positive numbers $\{E_n^{\spec}\}$ such that, for all $A\in\Omega$:
\begin{itemize}[leftmargin=0.7cm]
	\item \textit{Unions of intervals:} Each output $\Gamma_n^\spec(A)$ is given as a finite disjoint union
	\begin{equation}
\label{gamma_cover1}
\Gamma_n^\spec(A)=\bigcup_{m=1}^{M_n(A)}[a_{m,n}(A),b_{m,n}(A)]\subset\mathbb{R},
\end{equation}
with $[a_{m,n}(A),b_{m,n}(A)]\cap\spec(A)\neq\emptyset$ for $m=1,\ldots, M_n(A)$;
\item \textit{Spectral cover:} $\spec(A)\subset\Gamma_n^\spec(A)$;
\item \textit{Error control:} $\lim\limits_{n\rightarrow\infty}E_n^{\spec}(A)= 0$ and $\max\limits_{z\,\in\,\Gamma_n^{\spec}(A)}\mathrm{dist}(z,\spec(A))\leq E_n^{\spec}(A)$.
\end{itemize}

\item[\textbf{(S2)}]  \textbf{\emph{Spectral Distance.}} 
Using $\Lambda$, we assume that $\Phi_n:\mathbb{R}\times \Omega\rightarrow\mathbb{R}$ ($n\in\mathbb{N}$) can be computed in finitely many arithmetic operations and comparisons such that, with uniform convergence on compact subsets of $\mathbb{R}$,
$$
\mathrm{dist}(z,\spec(A))\leq \Phi_n(z,A),\quad \lim_{n\rightarrow \infty}\Phi_n(z,A)=\mathrm{dist}(z,\spec(A))\quad\forall (z,A)\in\mathbb{R}\times\Omega.
$$
By taking successive minima of the $\Phi_n$, we may assume without loss of generality that $\Phi_{n+1}(z,A)\leq \Phi_n(z,A)$ for all $n\in\mathbb{N}$.
\end{itemize}

It is worth discussing how these assumptions interact with the difficulty of computing spectra. If \textbf{(S1)} holds, then
$
d_{\mathrm{H}}(\Gamma_n^\spec(A),\spec(A))\leq E_n^{\spec}(A).
$
Hence, $\{\spec,\Omega,\MH,\Lambda\}\in\Delta_1^A$. For the reverse implication, the following example shows that, given a $\Delta_1^A$-tower for $\{\spec,\Omega,\MH,\Lambda\}$, we can typically compute \textbf{(S1)} covers.

\begin{example}[Almost equivalence of \textbf{(S1)} and $\{\spec,\Omega,\MH,\Lambda\}\in\Delta_1^A$]
\label{CHAP8_example:delta_to_pi}
Given a class $\Omega$ of bounded self-adjoint operators and an evaluation set $\Lambda$, suppose there exists a $\Delta_1^A$-tower $\{\widetilde{\Gamma}_n^\spec\}$ with $d_{\mathrm{H}}(\widetilde{\Gamma}_n^\spec(A),\spec(A))\leq 2^{-n}$ for all $A\in\Omega$.\ \  Additionally, assume that the outputs of this tower are finite point sets,
$$
\widetilde{\Gamma}_n^\spec(A)=\{c_{m,n}\}_{m=1}^{N(n)}\subset\mathbb{R},\quad N(n)<\infty.
$$
We can use $\widetilde{\Gamma}_n^\spec(A)$ to form \textbf{(S1)} covers by taking the set $\cup_{m=1}^{N(n)}[c_{m,n}-2^{-n},c_{m,n}+2^{-n}]$ and expressing it as a union of disjoint intervals. A similar process applies if $\widetilde{\Gamma}_n^\spec(A)$ consists of a finite collection of compact intervals.\hfill$\blacksquare$
\end{example}

The examples of this paper correspond to covers that arise from a dynamical formalism, such as the trace map for the Fibonacci Hamiltonian in \cref{sec:Fibonacci_numerics}. (For practical problems these covers will be obtained via Floquet--Bloch theory, which requires the solution of finite-dimensional self-adjoint eigenvalue problems. Such calculations are $\Delta_1^A$ problems~\cite[Corollary 6.9]{colbrook2022foundations}.) Since computing the spectrum of a general discrete Schr\"odinger operator on $\ell^2(\mathbb{Z})$ does not fall in $\Delta_1^G$, the sets $\Omega$ that we consider under \textbf{(S1)} necessarily consist of special subclasses of these operators.

If \textbf{(S2)} holds, then $\{\spec,\Omega,\mathcal{M}_{\mathrm{H}},\Lambda\}\in \Sigma_1^A$ \cite{colbrook2022foundations}. Many classes of operators satisfy this assumption \cite{colb1}, including those that we study in \cref{sec:examples2}. The following example shows that the converse typically holds.

\begin{example}[Almost equivalence of \textbf{(S2)} and $\{\spec,\Omega,\MH,\Lambda\}\in\Sigma_1^A\}$]
\label{CHAP8_example:sigma_to_pi}
Suppose that $\{\widetilde{\Gamma}_n^\spec\}$ is a $\Sigma_1^A$-tower for $\{\spec,\Omega,\MH,\Lambda\}$ so that $\dist(z,\spec(A))\leq 2^{-n}$ for all $z\in\widetilde{\Gamma}_n^\spec(A)$. The functions
$$
\Phi_n(z,A)=2^{-n}+\dist(z,\widetilde{\Gamma}_n^\spec(A))
$$
satisfy the conditions laid out in \textbf{(S2)}.
If $\widetilde{\Gamma}_n^\spec(A)$ is a finite collection of points or compact intervals, $\Phi_n(z,A)$ can be computed using finitely many arithmetic operations and comparisons.\hfill$\blacksquare$
\end{example}

It is straightforward to show that \textbf{(S1)} is a stronger assumption than \textbf{(S2)}. That is, if \textbf{(S1)} holds, then \textbf{(S2)} holds. Moreover, if $\{\spec,\Omega,\mathcal{M}_{\mathrm{H}},\Lambda\}\in\Pi_1^A$ and \textbf{(S2)} holds (so that $\{\spec,\Omega,\mathcal{M}_{\mathrm{H}},\Lambda\}\in\Sigma_1^A$), then $\{\spec,\Omega,\mathcal{M}_{\mathrm{H}},\Lambda\}\in\Delta_1^A$ \cite[Proposition 2.2.8]{colbrook2020PhD}.

\section{Methods Under (S1)}
\label{sec:Imethods}

We now demonstrate how \textbf{(S1)} leads to methods that compute various spectral properties of operators in the class $\Omega$. The classifications provided in this section are lower in the SCI hierarchy than in \cref{sec:IImethods}, as \textbf{(S1)} represents a stronger assumption than \textbf{(S2)}. Consequently, this allows for more extensive computational capabilities. We prove lower bounds for discrete Schr\"odinger operators using limit-periodic operators in \cref{sec:limit_per_lower}.

\subsection{Problem functions semicontinuous in the limit under \textbf{(S1)}}

We first consider problem functions that behave continuously under assumption \textbf{(S1)}. This case already includes interesting examples such as the Lebesgue measure of the spectrum, the number of connected components in the spectrum, the logarithmic capacity of the spectrum, and so forth. Suppose, then, that \textbf{(S1)} holds and that the problem function $\Xi:\Omega\rightarrow\mathbb{R}$ is of the form
\begin{equation}
\label{limit_well-behaved}
\begin{split}
\Xi(A)=f(\spec(A))\,\, \forall A\in\Omega \quad\text{where}\quad f:\mathcal{M}_{\mathrm{H}}\rightarrow\mathbb{R}\cup\{\pm\infty\}\\
\text{and} \,\,\lim_{n\rightarrow\infty}f(\Gamma_n^{\spec}(A))=f(\spec(A))\,\,\forall A\in\Omega.
\end{split}
\end{equation}
If we can compute each $f(\Gamma_n^{\spec}(A))$ to any given accuracy in finitely many arithmetic operations and comparisons, then we immediately obtain a $\Delta_2^A$ classification for $\Xi$. Moreover, if the limit in \eqref{limit_well-behaved} is monotonic, we obtain a $\Sigma_1^A$ or $\Pi_1^A$ classification. The following lemma makes this precise.

\begin{lemma}
\label{thm_warm_up}
Suppose that \textbf{\rm\textbf{(S1)}} holds and let $\Xi$ be a problem function such that \eqref{limit_well-behaved} holds. Suppose that we can compute each $f(\Gamma_n^{\spec}(A))$ to any given accuracy using $\Lambda$ and finitely many arithmetic operations and comparisons. Then $\{\Xi,\Omega,\mathbb{R}\cup\{\pm\infty\},\Lambda\}\in\Delta_2^A.$ Moreover, if the limit in \eqref{limit_well-behaved} is from above or below, then $\{\Xi,\Omega,\mathbb{R}\cup\{\pm\infty\},\Lambda\}\in\Pi_1^A$ or $\{\Xi,\Omega,\mathbb{R}\cup\{\pm\infty\},\Lambda\}\in\Sigma_1^A$, respectively. Finally, if the convergence in \eqref{limit_well-behaved} is effective, i.e., there exists an algorithm (using $\Lambda$) that, given $\varepsilon>0$ and $A\in\Omega$, returns $N=N(A,\varepsilon)\in\mathbb{N}$ such that
\[
n\ge N \quad\Rightarrow\quad \bigl|f(\Gamma_n^{\spec}(A)) - f(\spec(A))\bigr|<\varepsilon,
\]
then $\{\Xi,\Omega,\mathbb{R}\cup\{\pm\infty\},\Lambda\}\in\Delta_1^A$.
\end{lemma}

\begin{proof}
We prove the $\Sigma_1^A$ classification result; the others are similar.
Let $\widetilde{\Gamma}_n(A)$ be an approximation of $f(\Gamma_n^{\spec}(A))$, computed to accuracy $1/n$ using $\Lambda$.\ \  We then set $\Gamma_n(A)=\widetilde{\Gamma}_n(A)-1/n$.
\end{proof}

\cref{thm_warm_up} is a simple observation; however, it encompasses several nontrivial
problem functions. The following three examples are of interest.
\begin{itemize}[leftmargin=0.4cm]
	\item Let $\Xi_{\mathrm{Lm}}(A)=|\spec(A)|$, where we write $|\cdot|$ for the Lebesgue measure on the real line. Often we are interested in whether $\spec(A)$ is Lebesgue-null. Hence, we also consider the decision problem:
$$
\Xi_{\mathrm{Lm}}^{\mathrm{dec}}(A)=\begin{cases}
1,\quad\text{if }|\spec(A)|=0;\\
0,\quad\text{otherwise}.
\end{cases}
$$ 
\item Another example, this time mapping into $\mathbb{N}\cup\{\infty\}$ (one-point compactification of $\mathbb{N}$), is the problem function
$$
\Xi_{\mathrm{cc}}(A)=\begin{dcases}
\text{number of connected components of $\spec(A)$,}\quad&\text{if this is finite};\\
+\infty,\quad&\text{otherwise}.
\end{dcases}
$$
We also consider the decision problem
$$
\Xi_{\mathrm{cc}}^{\mathrm{dec}}(A)=\begin{dcases}
1,\quad&\text{if }\Xi_{\mathrm{cc}}(A)<\infty;\\
0,\quad&\text{otherwise}.
\end{dcases}
$$
\item The logarithmic capacity of $\spec(A)$, $\Xi_{\mathrm{cap}}(A)=\capacity(\spec(A))$, is equal to the transfinite diameter and the Chebyshev constant, since $\spec(A)$ is a compact set \cite{Ran95,saff2010logarithmic}. Thus, capacity is a measure of the size of a set and is related to polynomial approximation theory. In fact, Halmos~\cite{halmos1971capacity} showed that
\begin{align*}
\capacity(\spec(A))&=\inf_{\text{monic polynomial }p}\|p(A)\|^{\frac{1}{\mathrm{deg}(p)}}\\
&=\lim_{d\rightarrow\infty}\inf\left\{\|p(A)\|^{\frac{1}{d}}:\text{monic polynomial }p,\mathrm{deg}(p)=d\right\}.
\end{align*}
The capacity can also be viewed via minimizers of the logarithmic energy:
$$
\energy(\mu)= -\iint \log|x-y| \dd\mu(x) \dd\mu(y),
$$
where $\mu$ is a Borel probability measure having compact support. One then defines the capacity of a compact set $K$ to be
$$
\capacity(K) =  \exp(-\inf\{ \energy(\mu): \mu \text{ is a probability measure supported on }K\}).
$$
If $\capacity (K)=0$, $K$ is called polar. For any  non-polar compact $K$, there is a unique probability measure $\rho = \rho_K$ satisfying $\energy(\rho) =-\log\capacity(K).$ For a one-dimensional Schr\"odinger operator, the capacity of the spectrum is always at least $1$ \cite{Simon2007IPI}, which gives a universal lower bound on the size of the spectrum.\footnote{\small The capacity is precisely $1$ if and only if the density of states measure is the potential-theoretic equilibrium measure of the spectrum. This happens for a large class of one-dimensional quasicrystal models, namely those one-dimensional Schr\"odinger operators generated by a locally constant sampling of a subshift satisfying Boshernitzan's criterion for unique ergodicity; see \cite{DamLen2006DMJ} for details and definitions.} As well as $\Xi_{\mathrm{cap}}$, we consider the decision problem
$$
\Xi_{\mathrm{cap}}^{\mathrm{dec}}(A)=\begin{cases}
1,\quad\text{if }\Xi_{\mathrm{cap}}(A)=0;\\
0,\quad\text{otherwise}.
\end{cases}
$$
\end{itemize}

\begin{theorem}
\label{thm:cor_assumI_lims}
Suppose that \textbf{\rm\textbf{(S1)}} holds. Then
\begin{itemize}
	\item Number of connected components:
	$
	\{\Xi_{\mathrm{cc}},\Omega,\mathbb{N}\cup\{\infty\},\Lambda\}\in\Sigma_1^A.
	$
\item Lebesgue measure:
$
\{\Xi_{\mathrm{Lm}},\Omega,\mathbb{R}_{\geq0},\Lambda\}\in\Pi_1^A.
$
\item Capacity:
$
\{\Xi_{\mathrm{cap}},\Omega,\mathbb{R}_{\geq0},\Lambda\}\in\Pi_1^A.
$
\item Decision problems: For $\Xi=\Xi_{\mathrm{cc}}^{\mathrm{dec}},\Xi_{\mathrm{Lm}}^{\mathrm{dec}}$, $\Xi_{\mathrm{cap}}^{\mathrm{dec}}$,
$
\{\Xi,\Omega,\{0,1\},\Lambda\}\in\Pi_2^A.
$
\end{itemize}
\end{theorem}

\begin{proof}
Consider the output of $\Gamma_n^{\spec}(A)$ as $M_n(A)$ closed intervals in \eqref{gamma_cover1}. The collection of intervals produced is pairwise disjoint, each interval $[a_{m,n}(A),b_{m,n}(A)]$ intersects $\spec(A)$, and the union of the intervals covers $\spec(A)$. Hence, we have $M_n(A)\leq \Xi_{\mathrm{cc}}(A).$ If $\Xi_{\mathrm{cc}}(A)<\infty$, then $\spec(A)$ is a finite union of compact intervals. In this case, since $\Gamma_n^\spec(A)\rightarrow\spec(A)$, $M_n(A)\geq \Xi_{\mathrm{cc}}(A)$ for large $n$. If $\Xi_{\mathrm{cc}}(A)=\infty$, then $\mathbb{R}\backslash\spec(A)$ can be written as a countably infinite union of open intervals. The midpoint of any such interval must eventually be in $\mathbb{R}\backslash\Gamma_n^\spec(A)$ and $\Gamma_n^\spec(A)$ covers $\spec(A)$. This implies that $M_n(A)\rightarrow\infty$ and hence 
$
\{\Xi_{\mathrm{cc}},\Omega,\mathbb{N}\cup\{\infty\},\Lambda\}\in\Sigma_1^A.$
Note that
$$
\Gamma_n(A)=\begin{cases}1,\quad &\text{if }\Gamma_n^{\spec}(A)\text{ is connected};\\
0,\quad &\text{otherwise};
\end{cases}
$$
provides a $\Pi_1^A$ algorithm for the decision problem $\Xi:\text{``Is $\spec(A)$ an interval?''}$

Since each $\Gamma_n^{\spec}(A)$ is closed with $\spec(A)\subset\Gamma_n^{\spec}(A)$ and $\lim_{n\rightarrow\infty}\Gamma_n^{\spec}(A)=\spec(A)$, the Lebesgue measure of the spectrum satisfies \eqref{limit_well-behaved}. For this problem, the output of the associated algorithm is
$$
\Gamma_n(A)=\sum_{m=1}^{M_n(A)}\big(b_{m,n}(A)-a_{m,n}(A)\big),
$$
which provides a $\Pi_1^A$ algorithm for $\Xi_{\mathrm{Lm}}(A)$. Similarly, the capacity is right-continuous as a set function, meaning that for compact sets $K_n, K$ with $K_n\downarrow K$, $\capacity(K_n)\downarrow\capacity(K)$. Hence, we can compute the capacity of each output set $\Gamma_n^{\spec}(A)$ to obtain a $\Pi_1^A$ algorithm. In practice and the examples below, we use the conformal mapping method in \cite{liesen2017fast} to compute capacities of finite unions of intervals.\footnote{\small Computing the capacity of a finite union of intervals with error control is far from trivial. For instance, if one wanted to use the conformal mapping method in \cite{liesen2017fast}, it would require an error analysis of the integral equations and linear systems that are solved.}

Finally, we consider the decision problems. For the problem $\Xi_{\mathrm{Lm}}^{\mathrm{dec}}$, we set
$$
\Gamma_{n_2,n_1}(A)=\begin{dcases}
1,\quad&\text{if }\min_{1\leq n\leq n_1}|\Gamma_n^\spec(A)|<1/n_2;\\
0,\quad&\text{otherwise}.
\end{dcases}
$$
A straightforward check shows that this provides a $\Pi_2^A$-tower. Similarly, we can build $\Pi_2^A$ algorithms for the decision problems $\Xi_{\mathrm{cc}}^{\mathrm{dec}}$ and $\Xi_{\mathrm{cap}}^{\mathrm{dec}}$.
\end{proof}

An interesting consequence of \cref{thm:cor_assumI_lims} is that the associated decision problems are strictly more complex in the SCI hierarchy than the computation of the spectral size functionals themselves. While quantities such as the number of connected components, the Lebesgue measure, and the capacity lie at the $\Sigma_1^A$ or $\Pi_1^A$ level, deciding threshold properties of these quantities requires an additional limit, placing such decision problems in $\Pi_2^A$. Thus, passing from quantitative evaluation to a yes/no statement increases the computational complexity. This illustrates a subtle but important phenomenon: determining an exact value may be computationally simpler than deciding whether this value satisfies a prescribed condition.

In \cref{thm_warm_up,thm:cor_assumI_lims}, we approximate the property of $\spec(A)$ by computing it for the given covering $\Gamma_n^{\spec}(A)$. Computing fractal dimensions is necessarily more complicated since the fractal dimension of $\Gamma_n^{\spec}(A)$ given in \eqref{gamma_cover1} is always one (unless the intervals in the covering are degenerate, in which case the spectrum is a finite point set and has zero fractal dimension). Hence, we must use a different strategy.

\subsection{Adaptive resolution for the box-counting dimension} \label{ssec:S1:bcd}

We now consider the box-counting dimension of $\spec(A)$. We use the following definition of the upper and lower box-counting dimensions \cite[Chapter 3]{falconer2004fractal}.

\begin{definition}
A $\delta$-mesh interval is an interval of the form $[m\delta,(m+1)\delta]$ for $m\in\mathbb{Z}$. Let $F\subset\mathbb{R}$ be bounded and $N_\delta(F)$ denote the number of $\delta$-mesh intervals that intersect $F$. Let $\{\delta_k\}$ be a decreasing sequence such that $\delta_{k+1}\geq c\,\delta_k$ for some $c\in(0,1)$. 
Then the lower box-counting dimension $\underline{\mathrm{dim}}_{\mathrm{B}}(F)$ 
and upper box-counting dimension $\overline{\mathrm{dim}}_{\mathrm{B}}(F)$ 
of $F$ are defined as:
$$
\underline{\mathrm{dim}}_{\mathrm{B}}(F)=\liminf_{k\rightarrow\infty}\frac{\log(N_{\delta_k}(F))}{\log(1/\delta_k)},\qquad\overline{\mathrm{dim}}_{\mathrm{B}}(F)=\limsup_{k\rightarrow\infty}\frac{\log(N_{\delta_k}(F))}{\log(1/\delta_k)}
.
$$
If $\underline{\mathrm{dim}}_{\mathrm{B}}(F)=\overline{\mathrm{dim}}_{\mathrm{B}}(F)$, we write the limit as $\mathrm{dim}_{\mathrm{B}}(F)$, the box-counting dimension of~$F$.
\end{definition}

The definition of box-counting dimensions is independent of the choice of decreasing sequence $\{\delta_k\}$, subject to the condition that $\delta_{k+1}\geq c\,\delta_k$ for some $c\in(0,1)$. The following theorem classifies the computation of $\underline{\mathrm{dim}}_{\mathrm{B}}(\spec(A))$ and $\overline{\mathrm{dim}}_{\mathrm{B}}(\spec(A))$.

\begin{theorem}[Box-counting dimension can be computed in one limit]
\label{thm:box_counting}
Suppose that \textbf{\rm\textbf{(S1)}} holds. Then
$$
\{\underline{\mathrm{dim}}_{\mathrm{B}}(\spec(\cdot)),\Omega,[0,1],\Lambda\}\in\Sigma_2^A,\quad \{\overline{\mathrm{dim}}_{\mathrm{B}}(\spec(\cdot)),\Omega,[0,1],\Lambda\}\in\Pi_2^A.
$$
Let $\widehat{\Omega}\subset\Omega$ be the subset of operators $A\in\Omega$ for which $\overline{\mathrm{dim}}_{\mathrm{B}}(\spec(A))=\underline{\mathrm{dim}}_{\mathrm{B}}(\spec(A))$. Then
$$
\{\mathrm{dim}_{\mathrm{B}}(\spec(\cdot)),\widehat{\Omega},[0,1],\Lambda\}\in\Delta_2^A.
$$ 
\end{theorem}

The essence of the proof is adaptively choosing the parameter $n$ in the cover $\Gamma_n^{\spec}(A)$ according to the mesh size $\delta$. We compute the spectrum to an error tolerance that scales to zero as $\delta\downarrow 0$, allowing us to compute the box-counting dimension in one limit.

\begin{proof}
Given $\delta>0$, we can compute $N_{\delta}(\spec(A))$ to within a factor of three as follows. First, since $\spec(A)\subset\Gamma_n^{\spec}(A)$, we must have
$$
N_{\delta}(\spec(A))\leq N_{\delta}\left(\Gamma_n^{\spec}(A)\right) \rlap{\quad$\forall n\in\mathbb{N}$.}
$$
Now we successively compute the error bound $E_n^{\spec}(A)$, and let $n(\delta)$ be the smallest positive integer such that $E_{n(\delta)}(A)\leq \delta$. Any covering of $\spec(A)$ by $\delta$-mesh intervals can be extended to a covering of $\Gamma_{n(\delta)}^\spec(A)$. We can do this by extending any interval in the covering at its left and right endpoints to three adjacent $\delta$-mesh intervals. It follows that
$$
N_{\delta}\left(\Gamma_{n(\delta)}^{\spec}(A)\right)\leq 3 N_{\delta}(\spec(A)).
$$
Hence, taking $\delta_k=2^{-k}$ (for example), we have
\begin{align*}
\underline{\mathrm{dim}}_{\mathrm{B}}(\spec(A))&=\liminf_{k\rightarrow\infty}\frac{\log\left(N_{\delta_k}\left(\Gamma_{n(\delta_k)}^{\spec}(A)\right)\right)}{\log(1/\delta_k)},\\
\overline{\mathrm{dim}}_{\mathrm{B}}(\spec(A))&=\limsup_{k\rightarrow\infty}\frac{\log\left(N_{\delta_k}\left(\Gamma_{n(\delta_k)}^{\spec}(A)\right)\right)}{\log(1/\delta_k)}.
\end{align*}
Let $a_{k,j}$ be an approximation of ${\log(N_{\delta_k}(\Gamma_{n(\delta_k)}^{\spec}(A)))}/{\log(1/\delta_k)}$ computed to an accuracy $1/j$. (We use these approximations since the logarithm is a transcendental function and we require an arithmetic algorithm.) If we define $
\Gamma_{n_2,n_1}(A)=\max_{n_2\leq k\leq n_1}a_{k,n_1}$, then
$$
\lim_{n_1\rightarrow\infty}\Gamma_{n_2,n_1}(A)= \sup_{ k\geq n_2}\frac{\log\left(N_{\delta_k}\left(\Gamma_{n(\delta_k)}^{\mathrm{Sp}}(A)\right)\right)}{\log(1/\delta_k)}.
$$

\noindent
Hence $\lim_{n_2\rightarrow\infty}\lim_{n_1\rightarrow\infty}\Gamma_{n_2,n_1}(A)=\overline{\mathrm{dim}}_{\mathrm{B}}(\spec(A))$, with the final limit from above. This shows that $\{\overline{\mathrm{dim}}_{\mathrm{B}}(\spec(\cdot)),\Omega,[0,1],\Lambda\}\in\Pi_2^A$. An analogous argument leads to the classification $\{\underline{\mathrm{dim}}_{\mathrm{B}}(\spec(\cdot)),\Omega,[0,1],\Lambda\}\in\Sigma_2^A$. For the final part of the theorem, note that
$$
\mathrm{dim}_{\mathrm{B}}(\spec(A))=\lim_{k\rightarrow\infty}\frac{\log\left(N_{\delta_k}\left(\Gamma_{n(\delta_k)}^{\spec}(A)\right)\right)}{\log(1/\delta_k)}
$$
for any $A\in\widehat{\Omega}$ and set $\Gamma_n(A)=a_{n,n}$.
\end{proof}

\subsection{Computing the Hausdorff dimension of the spectrum}

A possible drawback of the box-counting dimension is its lack of countable stability. For example, $F=\{0,1,1/2,1/3,\ldots\}$ is a countable union of singletons, each of which has box-counting dimension~0, yet $\mathrm{dim}_{\mathrm{B}}(F)=1/2$. Moreover, the box-counting dimension does not always exist since we may have $\underline{\mathrm{dim}}_{\mathrm{B}}(F)<\overline{\mathrm{dim}}_{\mathrm{B}}(F)$. The Hausdorff dimension provides a more robust yet complicated notion of fractal dimension. We use the following definition of the Hausdorff dimension, which agrees with the usual one for compact sets \cite[Theorem 3.13]{fernandez2014fractal}.

\begin{definition}
Let $\rho_k$ be the set of all closed dyadic intervals 
$
[2^{-k}m,2^{-k}(m+1)], m\in\mathbb{Z}
$ and let $F\subset\mathbb{R}$ be compact. Set
$$
\mathcal{A}_k(F)=\left\{\{U_i\}_{i\in I}:I\text{ is finite},F\subset\bigcup_{i\in I}U_i,U_i\in\bigcup_{\ell\geq k}\rho_\ell\right\}
$$
and define
$$
H^{d}_{k}(F)=\inf\left\{\sum_i|U_i|^d:\{U_i\}_{i\in I}\in\mathcal{A}_{k}(F)\right\},\quad H^{d}(F)=\lim_{k\rightarrow\infty}H^{d}_{k}(F).
$$
The Hausdorff dimension of $F$ is the unique number $\mathrm{dim}_\mathrm{H}(F)\in[0,1]$ such that $H^{d}(F)=0$ for $d>\mathrm{dim}_\mathrm{H}(F)$ and $H^{d}(F)=\infty$ for $d<\mathrm{dim}_\mathrm{H}(F)$.
\end{definition}

While the box-counting dimension of $\spec(A)$ need not exist, the Hausdorff dimension of $\spec(A)$ is always well-defined \cite[Section 2.2]{falconer2004fractal}. Moreover, it satisfies:
$$
\mathrm{dim}_\mathrm{H}(\spec(A))\ \leq\   \underline{\mathrm{dim}}_{\mathrm{B}}(\spec(A))\ \leq\ \overline{\mathrm{dim}}_{\mathrm{B}}(\spec(A)).
$$
In general, computing a set's Hausdorff dimension analytically or numerically is more challenging than computing its box-counting dimension. The following theorem reflects this added difficulty.

\begin{theorem}[The Hausdorff dimension can be computed in two limits]
\label{thm_Haus}
Suppose that \textbf{\rm\textbf{(S1)}} holds. Then
$\{\mathrm{dim}_{\mathrm{H}}(\spec(\cdot)),\Omega,[0,1],\Lambda\}\in\Sigma_2^A.$ 
\end{theorem}

\begin{proof}
We define
\begin{align}
\mathcal{A}_{j,k}(F)&=\left\{\{U_i\}_{i\in I}:I\text{ is finite},F\subset\bigcup_{i\in I}U_i,U_i\in\bigcup_{j\leq \ell\leq k}\rho_\ell\right\},\label{A_haus_def}\\
H^{d}_{j,k}(F)&=\inf\left\{\sum_i|U_i|^d:\{U_i\}_{i\in I}\in\mathcal{A}_{j,k}(F)\right\}.\notag
\end{align}
We choose $n=n(k)$ adaptively to control the error in the covers, $E_{n(k)}(A)\leq 2^{-k}$. From the same arguments used to prove \cref{thm:box_counting}, we then have
$$
H^{d}_{j,k}(\spec(A))\leq H^{d}_{j,k}\left(\Gamma_{n(k)}^\spec(A)\right)\leq 3H^{d}_{j,k}(\spec(A)).
$$
We set
$$
h_{n_2,n_1}(A,d)=\max_{1\leq j\leq n_2} \min_{1\leq k\leq n_1}H^{d}_{j,k}\left(\Gamma_{n(k)}^\spec(A)\right).
$$
Due to taking minima over $k$, $\lim_{n_1\rightarrow\infty}h_{n_2,n_1}(A,d)=h_{n_2}(A,d)$ exists. Since
$$
H^{d}_{n_2}(\spec(A))=\inf_{n_1}H^{d}_{n_2,n_1}(\spec(A)),
$$
we must have
\begin{equation}
\label{eq:Haus_sandwich}
H^{d}_{n_2}(\spec(A))\leq h_{n_2}(A,d)\leq 3H^{d}_{n_2}(\spec(A)).
\end{equation}
Define
$$
\Gamma_{n_2,n_1}(A)=\max\left\{\frac{j}{2^{n_2}}:1\leq j\leq2^{n_2},h_{n_2,n_1}(A,k/2^{n_2})\geq1\text{ for }k=1,\ldots,j\right\},
$$
where we define the maximum over the empty set to be $0$. If $h_{n_2}(A,d)\geq1$, then $h_{n_2,n_1}(A,d)\geq1$ for all $n_1$; otherwise $h_{n_2,n_1}(A,d)<1$ for large $n_1$. Hence
$$
\lim_{n_1\rightarrow\infty}\Gamma_{n_2,n_1}(A)=\max\left\{\frac{j}{2^{n_2}}:1\leq j\leq2^{n_2},h_{n_2}(A,k/{2^{n_2}})\geq1\text{ for }k=1,\ldots,j\right\},
$$
where we denote this limit by $\Gamma_{n_2}(A)$. Using the monotonicity of $h_{n_2}(A,d)$ in $d$ and \eqref{eq:Haus_sandwich}, it follows that
$$
\lim_{n_2\rightarrow\infty}\Gamma_{n_2}(A)=\mathrm{dim}_\mathrm{H}(\spec(A)).
$$
Since $h_{n_2}$ is non-decreasing in $n_2$, the set $\{1/2^{n_2},2/2^{n_2},\ldots,1\}$ refines itself, convergence is monotonic from below. Hence, we obtain the stated $\Sigma_2^A$ classification.
\end{proof}

In practice, we use dynamic programming to rapidly compute the optimum covers in our approximation of the Hausdorff dimension, exploiting the hierarchical nature of the dyadic intervals.

\subsection{Lower bounds through limit-periodic operators}
\label{sec:limit_per_lower}

We now prove lower bounds, demonstrating that the above algorithms are sharp under assumption \textbf{(S1)}. We do this for the class $\Omega_{\mathrm{DS}}$ of discrete Schr\"odinger operators on $\ell^2(\mathbb{Z})$; the proofs for other classes of operators are entirely analogous. A discrete Schr\"odinger operator $A\in\Omega_{\mathrm{DS}}$ on $\ell^2(\mathbb{Z})$ is given by
$$
[A \psi](n) 
= \psi(n-1) + \psi(n+1) + V(n) \psi(n),
$$
where we assume that the function $V$ is real and bounded. To make our result precise, we assume that $\Lambda$, the information our algorithm is allowed to access, is a sequence $\{\Gamma_n^\spec\}$ such that
\begin{equation}
\label{delta_1_spec_info}
d_{\mathrm{H}}\big(\Gamma_n^\spec(A),\spec(A)\big)\leq 2^{-n}
   \quad \forall A\in\Omega_{\mathrm{DS}}, n\in\mathbb{N}.
\end{equation}
Recall from \cref{sec:assumptions_on_spec} that such a sequence can easily be altered to satisfy \textbf{(S1)}.
We want algorithms that work for all possible choices of $\{\Gamma_n^\spec\}$ that satisfy \cref{delta_1_spec_info}. Hence, we view the sequences $\{\{\Gamma_n^\spec(A)\}:A\in\Omega_{\mathrm{DS}},\{\Gamma_n^\spec\}\text{ s.t. \cref{delta_1_spec_info} holds}\}$ as the inputs to our putative algorithms. We will denote this class by $\Omega_{\mathrm{DS}}^{\rm\textbf(S1)}$ and the readable information as $\Lambda_{\rm\textbf(S1)}$.

The arguments for showing that our algorithms in this section are sharp fall into two categories:
\begin{itemize}
	\item For those problems where we have proven a $\Sigma_1^A$ or $\Pi_1^A$ upper bound, we will show that the problem does not lie in $\Delta_1^G$.
	\item For those problems where we have proven a $\Sigma_2^A$ or $\Pi_2^A$ upper bound, we will show that the problem does not lie in $\Delta_2^G$.
\end{itemize}
We will demonstrate the arguments for $\Xi_{\mathrm{Lm}}$ and $\mathrm{dim}_{\mathrm{H}}(\spec(\cdot))$, respectively. The other problem functions follow analogously.

We will also use some facts about limit-periodic operators that can be found in \cite{damanik2018spectral} (see also \cite{Avi2009CMP}). A discrete Schr\"odinger operator $A\in\Omega_{\mathrm{DS}}$ on $\ell^2(\mathbb{Z})$ is \textit{limit-periodic} if there is a sequence of periodic potentials $V_1,V_2,V_3,\ldots$ such that $\lim_{k \rightarrow\infty}\|V - V_k\|_{\infty}=0$. Equivalently \cite{Avi2009CMP}, we fix a monothetic Cantor group\footnote{\small A Cantor group is a compact, totally disconnected group with no isolated points. A monothetic group is a topological group with a dense cyclic subgroup.} $\mathcal{G}$ and an $\alpha\in \mathcal{G}$ so that $\{n\alpha : n\in\mathbb{Z}\}$ is dense in $\mathcal{G}$.\ \ If $f:\mathcal{G}\rightarrow\mathbb{R}$ is continuous, then the potential $V_f(n) = f(n\alpha)$ is limit-periodic. Let $H_f$ denote the discrete Schr\"odinger operator with potential $V_f$. We shall use the following facts.
\begin{itemize}
	\item There is a dense set $\mathcal{C}\subset C(\mathcal{G},\mathbb{R})$ such that $\underline{\mathrm{dim}}_{\mathrm{B}}(\spec(H_f))=0$ (and hence $\spec(H_f)$ has infinitely many connected components, Lebesgue measure 0, and Hausdorff dimension~0) for all $f\in\mathcal{C}$.
	\item There is a dense set $\mathcal{P}\subset C(\mathcal{G},\mathbb{R})$ such that $V_f$ is periodic for all $f\in\mathcal{P}$.\ \  Hence, if $f\in\mathcal{P}$, $\spec(H_f)$ is a union of finitely many nondegenerate
closed intervals (and hence has nonzero Lebesgue measure, fractal dimension 1, and finitely many connected components).
\end{itemize}
We are now ready to prove the lower bounds. We begin with the more straightforward argument for problems not in $\Delta_1^G$.

\begin{proof}[Proof that $
\{\Xi_{\mathrm{Lm}},\Omega_{\mathrm{DS}}^{\rm{(S1)}},\mathbb{R}_{\geq0},\Lambda_{\rm{(S1)}}\}\notin\Delta_1^G
$]
Suppose this were false and $\{\Gamma_n\}$ is a $\Delta_1^G$-tower solving the problem. We consider an input $\{\Gamma_n^\spec(H_f)\}$, where $f\in \mathcal{P}$, that satisfies \cref{delta_1_spec_info} with strict inequality. Since $\Xi_{\mathrm{Lm}}(H_f)>0$, there exists $n_0\in\mathbb{N}$ such that
$$
\Gamma_{n_0}(\{\Gamma_n^\spec(H_f)\})> 2^{-n_0}.
$$
However, $\Gamma_{n_0}(\{\Gamma_n^\spec(H_f)\})$ can only depend on $\{\Gamma_n^\spec(H_f)\}_{n=1}^N$ for some finite $N$. We may choose $f'\in\mathcal{C}$ sufficiently close to $f$ so that for any $n=1,\ldots, N$,
\begin{align*}
d_{\mathrm{H}}\big(\Gamma_n^\spec(H_f),\spec(H_{f'})\big)&\leq d_{\mathrm{H}}\big(\Gamma_n^\spec(H_f),\spec(H_{f})\big)+d_{\mathrm{H}}\big(\spec(H_{f}),\spec(H_{f'})\big)\\
&\leq d_{\mathrm{H}}\big(\Gamma_n^\spec(H_f),\spec(H_{f})\big) +\|f-f'\|_\infty < 2^{-n}.
\end{align*}
Here, we have used that the spectrum of self-adjoint operators is $1$-Lipschitz with respect to the operator norm. Since the above bound is consistent with \cref{delta_1_spec_info} for $A=H_{f'}$ and $n=1,\ldots,N$, we may extend $\{\Gamma_n^\spec(H_f)\}_{n=1}^N$ to an input $\{\Gamma_n^\spec(H_{f'})\}_{n=1}^\infty$ so that \cref{delta_1_spec_info} holds with $A=H_{f'}$ for all $n$. By consistency of general algorithms,
$$
\Gamma_{n_0}(\{\Gamma_n^\spec(H_{f'})\})=\Gamma_{n_0}(\{\Gamma_n^\spec(H_f)\})> 2^{-n_0}.
$$
But this contradicts $\{\Gamma_n\}$ being a $\Delta_1^G$-tower since $\Xi_{\mathrm{Lm}}(H_{f'})=0$ .
\end{proof}

We now consider the more complicated argument for problems not in $\Delta_2^G$. Given a candidate $\Delta_2^G$-tower $\{\Gamma_n\}$, the idea will be to build an $f$ such that the sequence $\Gamma_n(\{\Gamma_m^\spec(H_{f})\})$ cannot converge.

\begin{proof}[Proof that {$
\{\mathrm{dim}_{\mathrm{H}}(\spec(\cdot)),\Omega_{\mathrm{DS}}^{\rm{(S1)}},[0,1],\Lambda_{\mathrm{(S1)}}\}\notin\Delta_2^G
$}]
Suppose this were false and $\{\Gamma_n\}$ is a $\Delta_2^G$-tower solving the problem. We first choose $f_1\in\mathcal{C}$ such that $\|f_1\|_{\infty}\leq 1$ and input $\{\Gamma_m^\spec(H_{f_1})\}$ that satisfies \cref{delta_1_spec_info} with strict inequality. Due to the convergence of $\{\Gamma_n\}$, there exists some $n_1$ such that
$$
\Gamma_{n_1}(\{\Gamma_m^\spec(H_{f_1})\})\leq 1/4.
$$
However, $\Gamma_{n_1}(\{\Gamma_m^\spec(H_{f_1})\})$ can only depend on $\{\Gamma_m^\spec(H_{f_1})\}_{m=1}^{N_1}$ for some finite $N_1$.
We may choose $f_2$ with $f_1+f_2\in\mathcal{P}$ and $\|f_2\|_{\infty}$ sufficiently small (and not larger than $1/2$) so that for any $m=1,\ldots, N_1$,
\begin{align*}
d_{\mathrm{H}}\big(\Gamma_m^\spec(H_{f_1}),\spec(H_{f_1+f_2})\big)&\leq d_{\mathrm{H}}\big(\Gamma_m^\spec(H_{f_1}),\spec(H_{f_1})\big)+d_{\mathrm{H}}\big(\spec(H_{f_1}),\spec(H_{f_1+f_2})\big)\\
&\leq d_{\mathrm{H}}\big(\Gamma_m^\spec(H_{f_1}),\spec(H_{f_1})\big) +\|f_2\|_\infty < 2^{-m}.
\end{align*}
It follows that we may set $\Gamma_m^\spec(H_{f_1+f_2})=\Gamma_m^\spec(H_{f_1})$ for $m=1,\ldots,N_1$ and extend to obtain $\Gamma_m^\spec(H_{f_1+f_2})$ for $m>N_1$ that satisfies \cref{delta_1_spec_info} with strict inequality. Due to the convergence of $\{\Gamma_n\}$, there exists some $n_2>n_1$ such that
$$
\Gamma_{n_2}\big(\{\Gamma_m^\spec(H_{f_1+f_2})\}\big)\geq 3/4.
$$
However, $\Gamma_{n_2}\big(\{\Gamma_m^\spec(H_{f_1+f_2})\}\big)$ can only depend on $\{\Gamma_m^\spec(H_{f_1+f_2})\}_{m=1}^{N_2}$ for some finite $N_2>N_1$. We can now choose $f_3$ with $f_1+f_2+f_3\in\mathcal{C}$ and $\|f_3\|_{\infty}$ sufficiently small (and not larger than $1/4$) so that we can take $\Gamma_m^\spec(H_{f_1+f_2+f_3})=\Gamma_m^\spec(H_{f_1+f_2})$ for $m=1,\ldots,N_2$ and extend to obtain $\Gamma_m^\spec(H_{f_1+f_2+f_3})$ for $m>N_2$ that satisfies \cref{delta_1_spec_info} with strict inequality. We now continue this process inductively, switching between
$$
\Gamma_{n_k}\big(\{\Gamma_m^\spec(H_{f_1+\cdots + f_k})\}\big)\leq 1/4
$$
for odd $k$ and
$$
\Gamma_{n_k}\big(\{\Gamma_m^\spec(H_{f_1+\cdots + f_k})\}\big)\geq 3/4
$$
for even $k$. Each time, we ensure that $\|f_k\|_{\infty}\leq 2^{-k}$ and is sufficiently small so that we may take
$$
\Gamma_m^\spec(H_{f_1+\cdots +f_k})=\Gamma_m^\spec(H_{f_1+\cdots+ f_{k-1}})\quad\text{for}\quad m=1,\ldots, N_{k-1}.
$$
Clearly $f=\sum_{k=1}^\infty f_k\in C(\mathcal{G},\mathbb{R})$ (the sum converges uniformly). Moreover, by the consistency of general algorithms, $\Gamma_{n_k}(\{\Gamma_m^\spec(H_{f_1+\cdots+ f_k})\})=\Gamma_{n_k}(\{\Gamma_m^\spec(H_{f})\})$ for all $k\in\mathbb{N}$. But then $\Gamma_{n}(\{\Gamma_m^\spec(H_{f})\})$ lies infinitely often in each of the intervals $[0,1/4]$ and $[3/4,1]$ as $n\rightarrow\infty$, and so cannot converge, which is the required contradiction.
\end{proof}

\section{Examples Using \textbf{(S1)}}
\label{sec:examples1}

We now consider two examples of the machinery developed in \cref{sec:Imethods}. Both examples are aperiodic discrete Schr\"odinger operators on $\ell^2(\mathbb{Z})$ (or direct products thereof), and we use various types of periodic approximations to realize the assumption \textbf{(S1)}. We will consider more intricate two-dimensional examples in \cref{sec:examples2}.

\subsection{The almost Mathieu operator}

First, we consider the almost Mathieu operator, defined on $\ell^2(\mathbb{Z})$ for $\lambda,\alpha,\theta\in\mathbb{R}$ by
\begin{equation} \label{eq:AMO}
[H_{\lambda,\alpha,\theta}\,\psi](n)
=\psi(n-1) + \psi(n+1) + 2\lambda\cos(2\pi n\alpha+\theta)\,\psi(n).
\end{equation}
This family of operators (in the critical case $\lambda =1$) arises after a suitable gauge transformation from the Hamiltonian associated with an electron on the two-dimensional square lattice subjected to a uniform perpendicular magnetic field with flux $\alpha$ \cite{peierls1933theorie}. For the one-dimensional operators in \eqref{eq:AMO}, the threshold $\lambda=1$ represents a transition from extended states for the subcritical case, $\lambda\in[0,1)$, to localization for the supercritical case $\lambda>1$ for almost all phases and frequencies. In fact, localization holds for a.e.\ $(\alpha,\theta)$ \cite{Jito1999Annals} with precise arithmetic conditions. Recent works have studied the phase transition line for frequencies $\alpha$ in the exceptional set, showing a sharp phase transition dictated precisely by the arithmetic of $\alpha$ \cite{AvilaYouZhou2017DMJ}. The almost Mathieu operator at the critical value $\lambda=1$ is one of the most important models in physics, where it is also known as the Harper or Azbel--Hofstadter model \cite{avron2003topological}; three of Barry Simon's fifteen problems about Schr\"odinger operators for the twenty-first century feature the almost Mathieu operator \cite{simon2000schrodinger}.

 If $\alpha$ is irrational, the spectrum of $H_{\lambda,\alpha,\theta}$ is independent of $\theta$. To deal with $\theta$-independent objects for every $\alpha$, one considers the following sets:
\begin{equation}
\label{AM_union_spectrum}
\spec_{+}(\alpha,\lambda)=\bigcup_{\theta\in\mathbb{R}}\spec(H_{\lambda,\alpha,\theta}).
\end{equation}
These are continuous in the Hausdorff metric as $\alpha$ varies. 
For any $\alpha$, $\spec_{+}(\alpha,\lambda) = \lambda\spec_+(\alpha,1/\lambda)$, which follows from Aubry duality \cite{aubry1978new, avron1983almost}, so we consider $\lambda\in[0,1]$.
\Cref{AMfig1} (left) shows the famous Hofstadter butterfly \cite{hofstadter1976energy}, namely, $\spec_{+}(\alpha,1)$ as $\alpha$ varies. If $\alpha=p/q\in\mathbb{Q}$ with $p,q$ coprime and $\lambda\in[0,1]$, then \cite{avron1990measure}
$$
\spec_{+}(\alpha,\lambda)=\spec(H_{\lambda,\alpha,0})\cup \spec(H_{\lambda,\alpha,\pi/q}).
$$
Hence, we can approximate $\spec_{+}(\alpha,\lambda)$ using rational approximations of $\alpha$, for which the operator is periodic. If the period is $q$, the spectrum of the periodic approximation is the union of $q$ real intervals. This spectrum can be characterized by Bloch--Floquet theory~(see, e.g., \cite[chap.~7]{Tes99}); the ends of the intervals can be computed by solving two Hermitian eigenvalue problems involving matrices of dimension $q$ with tridiagonal structure plus entries in the top-right and bottom-left entries. For details and a discussion of how to compute the spectrum of these periodic Jacobi operators in $\mathcal{O}(q^2)$ operations, see \cite{Lam97,puelz2015spectral}. To obtain a $\Delta_1^A$ algorithm for the spectrum when $\alpha\notin\mathbb{Q}$, we consider $\spec_{+}(\alpha',\lambda)$ for $\alpha'\in\mathbb{Q}$ with $|\alpha-\alpha'|$ sufficiently small and use \cite[Proposition 7.1]{avron1990measure}:
\begin{equation}
\label{AM_cty_spectrum}
d_{\mathrm{H}}\big(\spec_{+}(\alpha,\lambda),\spec_{+}(\alpha',\lambda)\big)
\leq c\sqrt{\lambda|\alpha-\alpha'|},
\end{equation}
where $c>0$ is an explicit universal constant. 
Thus, by absorbing this error bound, we can increase the size of the intervals of $\spec_{+}(\alpha',\lambda)$ to obtain a covering that satisfies assumption \textbf{(S1)} in \cref{sec:assumptions_on_spec}.

The spectrum $\spec_+(\alpha,\lambda)$ is a Cantor set for $\lambda\neq 0$ and $\alpha\notin\mathbb{Q}$ \cite{avila2009ten}, and thus has an infinite number of connected components. When $\alpha\notin\mathbb{Q}$, $|\spec(H_{\lambda,\alpha,\theta})|=4\big|1-|\lambda|\big|$. This was conjectured based on numerical work in \cite{aubry1980analyticity}, and proven by Avila \& Krikorian~\cite{avila2006reducibility}, with earlier partial results in \cite{HelffSjos1989MSMF, avron1990measure, Last1993CMP, last1994zero, JK02}. \Cref{AMfig1} (right) shows $\spec_{+}(p/q,\lambda)$ for the convergent $p/q=377/610$ of the continued fraction expansion of $\alpha = (\sqrt{5}-1)/2$. The thinning of the spectrum as $\lambda\uparrow 1$ is visible. \Cref{AMfig2} shows approximations of the Lebesgue measure (left), the number of connected components (middle), and the capacity (right) of $\spec_{+}(\alpha,\lambda)$. Instead of indexing the algorithm $\Gamma_n^\spec$ by $n$, we have indexed it by the convergents $p/q\rightarrow\alpha$. The plots show the monotonicity of the algorithms, which are decreasing, increasing, and decreasing, respectively, from left to right as $p/q\rightarrow\alpha$. The approximation of the Lebesgue measure converges down to $4|1-|\lambda||$, the number of connected components diverges to $+\infty$ for any fixed $\lambda$, and the capacity converges down to $1$ (which is the correct analytic value \cite{DamanikFillman2022ESO}). In view of our impossibility results, one cannot in general obtain two-sided ($\Delta_1$) error control for quantities such as the Lebesgue measure or the number of connected components of the spectrum without assumptions beyond \textbf{(S1)}; at best, one-sided ($\Pi_1$ or $\Sigma_1$) control is achievable.

\begin{figure}
\centering
\raisebox{-0.5\height}{\includegraphics[width=0.32\textwidth,trim={0mm 0mm 0mm 0mm},clip]{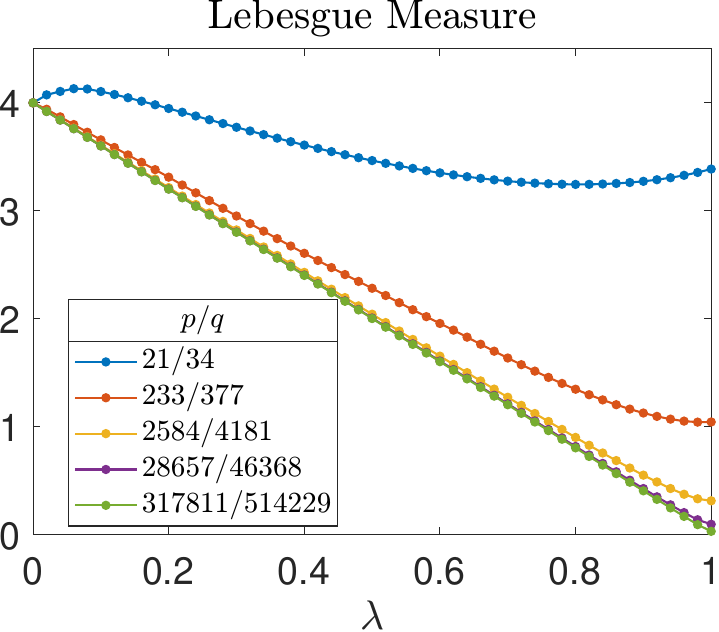}}
\hfill
\raisebox{-0.5\height}{\includegraphics[width=0.32\textwidth,trim={0mm 0mm 0mm 0mm},clip]{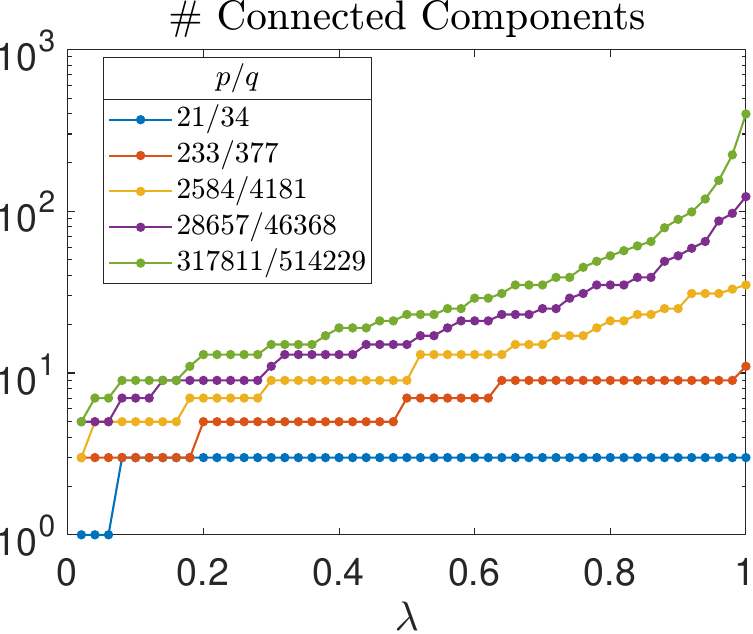}}
\hfill
\raisebox{-0.5\height}{\includegraphics[width=0.32\textwidth,trim={0mm 0mm 0mm 0mm},clip]{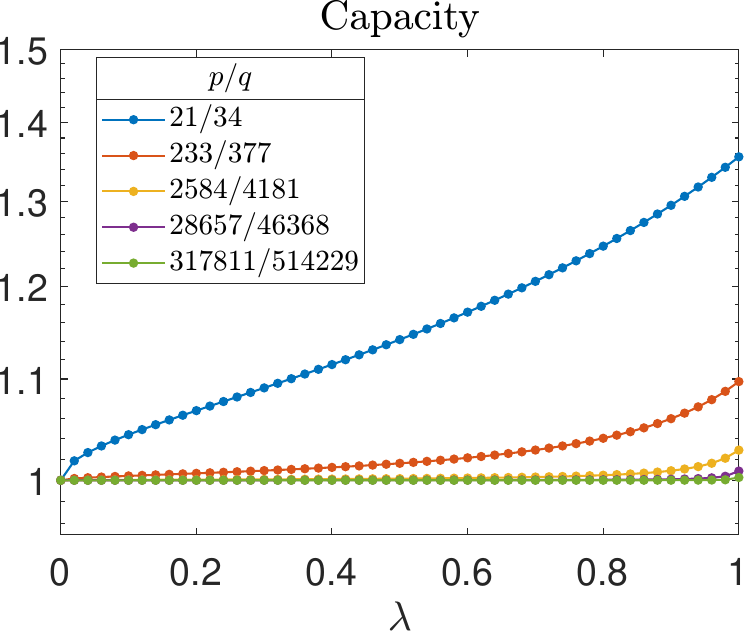}}
\caption{Results for the almost Mathieu operator with $\alpha=(\sqrt{5}-1)/2$. We have only shown the results for $\lambda\in[0,1]$ due to Aubry duality $\spec(H_{\lambda,\alpha,\theta})=\lambda\,\spec(H_{1/\lambda,\alpha,\theta})$. As discussed in the text, our results show that one-sided error control is optimal 
for these quantities. Left: Output of $\Pi_1^A$ algorithm for the Lebesgue measure of the spectrum. For a fixed $\lambda$, the approximations decrease monotonically as $|\alpha-p/q|$ decreases; the results are consistent with convergence to the known formula $4\big|1-|\lambda|\big|$. Middle: Output of the $\Sigma_1^A$ algorithm for the number of connected components of the spectrum. The approximations increase monotonically as $|\alpha-p/q|$ decreases, consistent with divergence to $+\infty$. Right: Output of the $\Pi_1^A$ algorithm for the logarithmic capacity of the spectrum. The approximations decrease monotonically and are consistent with convergence to~$1$.}
\label{AMfig2}
\end{figure}

We now fix $\lambda =1$ and consider various fractal dimensions of $\spec_{+}(\alpha,1)$ for irrational $\alpha$. (If $\alpha\in\mathbb{Q}$ or if $|\lambda|\neq 1$, then $\spec_{+}(\alpha,\lambda)$ has positive Lebesgue measure and hence fractal dimension $1$.) 
A past conjecture often attributed to Thouless \cite{wilkinson1994spectral} was that $\mathrm{dim}_{\mathrm{H}}(\spec_{+}(\alpha,1))=1/2$; see \cite{jitomirskaya2019critical_survey} for a discussion. Determining the fractal dimensions appeared in Simon's new list of problems \cite{simon2019fifty}. In \cite{jitomirskaya2019critical}, it was shown that $\mathrm{dim}_{\mathrm{H}}(\spec_{+}(\alpha,1))\leq 1/2$ for any irrational $\alpha$. The argument in \cite{jitomirskaya2019critical} can be extended to show that $\underline{\mathrm{dim}}_{\mathrm{B}}(\spec_{+}(\alpha,1))\leq 1/2$ (see also the earlier work \cite{last1994zero}) using the Lebesgue measure characterization of $\underline{\mathrm{dim}}_{\mathrm{B}}$ given in \cite[Proposition 3.2]{falconer2004fractal}. Other results include zero Hausdorff dimension for a subset of Liouville $\alpha$ \cite{last2016zero}, extended to all weakly Liouville $\alpha$ in \cite{shamis2023abominable}, and a dense positive Hausdorff dimension set of Diophantine $\alpha$ for which $\spec_+(\alpha,1)$ has positive Hausdorff dimension \cite{helffer2019positive}. For further studies, see  \cite{geisel1991new, ketzmerick1998covering, tang1986global, wilkinson1994spectral, stinchcombe1987hierarchical}.

We consider two values of $\alpha$, the (inverse) golden ratio $(\sqrt{5}-1)/2  = 0.61803\ldots$ and Cahen's constant:\footnote{\small The $s_j$ appearing in this definition are the (rapidly growing) terms in Sylvester's sequence, which arises in Euclid's proof that there are infinitely many prime numbers. }
$$
C=\sum_{j=0}^\infty\frac{(-1)^j}{s_j-1} = 0.64341\ldots,\quad\text{where } s_0=2\text{ and }s_{j+1}=1+\prod_{i=0}^js_i.
$$
These two values of $\alpha$ are chosen because of their different so-called \textit{irrationality measure}. The irrationality measure of $\alpha\in\mathbb{R}\backslash\mathbb{Q}$ is given by
$$
\mu(\alpha)=1+\limsup_{n\rightarrow\infty}\frac{\log(q_{n+1})}{\log(q_n)},
$$
where $p_n/q_n$ are the convergents from the continued fraction expansion of $\alpha$. We have
$\mu((\sqrt{5}-1)/2)=2 $ and $\mu(C)=3.$ \Cref{AMfig3} plots the covering numbers from the proof of \cref{thm:box_counting} against the interval length on a logarithmic scale. To avoid the gridding effect of the $\delta$-mesh intervals $\{[m\delta,(m+1)\delta]:m\in\mathbb{Z}\}$, we have perturbed the endpoints of the mesh by a uniform random variable and taken the average of the covering over 100 trials. For $\alpha=(\sqrt{5}-1)/2$, we get a linear fit with slope $\approx 0.500$, suggesting that $\mathrm{dim}_{\mathrm{B}}(\spec_{+}((\sqrt{5}-1)/2,1))=1/2$. This is slightly larger than the value of $0.498$ reported in \cite{wilkinson1994spectral}, though we warn the reader that it is, of course, impossible to be certain we have converged to that accuracy. An advantage of our approach is that we combine better rational approximations of $\alpha$ and use the argument in the proof of \cref{thm:box_counting} to ensure that $|p_n/q_n-\alpha|$ is sufficiently small for a given $\delta$. For $\alpha=C$, there is no apparent linear scaling, suggesting that the upper and lower box-counting dimensions disagree. More precisely, it appears that $\underline{\mathrm{dim}}_B(\spec_{+}(C,1))\approx1/2$ and $\overline{\mathrm{dim}}_B(\spec_{+}(C,1))\approx2/3$.

\begin{figure}
\centering
\raisebox{-0.5\height}{\includegraphics[width=0.441\textwidth,trim={0mm 0mm 0mm 0mm},clip]{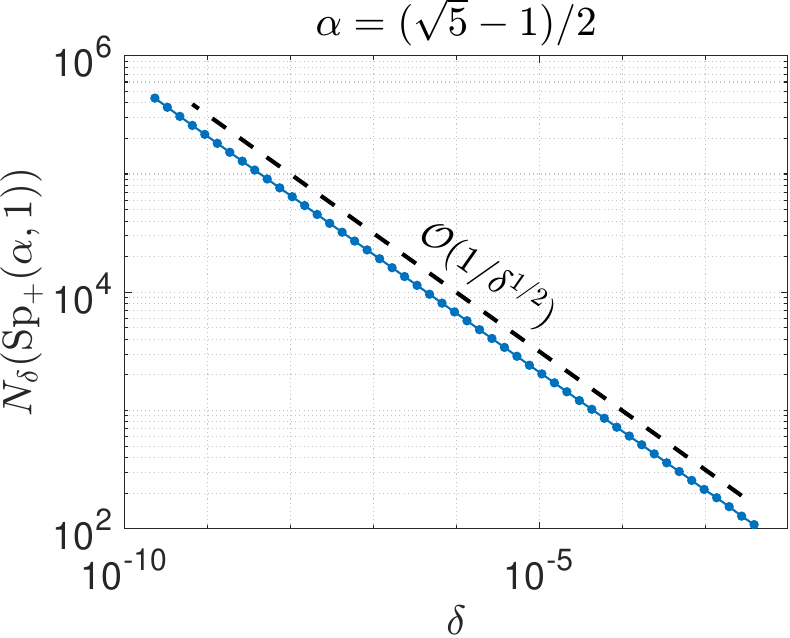}}
\hfill
\raisebox{-0.5\height}{\includegraphics[width=0.441\textwidth,trim={0mm 0mm 0mm 0mm},clip]{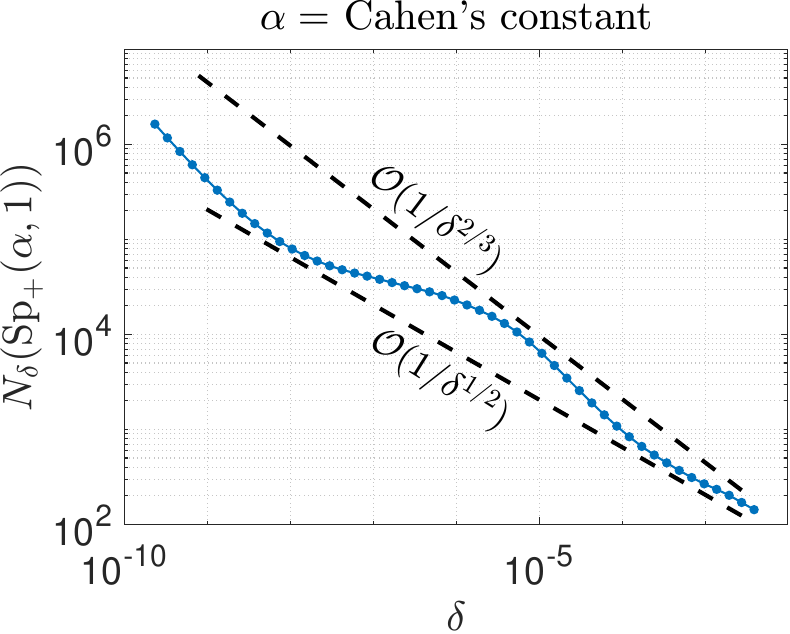}}
\caption{Covering numbers from the proof of \cref{thm:box_counting} for $\spec_{+}(\alpha,1)$ for two different values of $\alpha$. The left plot shows a scaling of exponent $1/2$, while the right plot shows an apparent lack of scaling and possible values for the upper and lower box-counting dimensions. It appears that $\underline{\mathrm{dim}}_B(\spec_{+}(C,1))\approx1/2$ and $\overline{\mathrm{dim}}_B(\spec_{+}(C,1))\approx2/3$.}
\label{AMfig3}
\end{figure}

\Cref{AMfig4} shows the output of the $\Sigma_2^A$ algorithm for computing the Hausdorff dimension of the spectrum. For $\alpha=(\sqrt{5}-1)/2$, there is strong evidence that $\mathrm{dim}_{\mathrm{H}}(\spec_{+}((\sqrt{5}-1)/2,1))\approx1/2$. Since the Hausdorff dimension provides a lower bound for the lower box-counting dimension, this also supports the conclusion $\mathrm{dim}_{\mathrm{B}}(\spec_{+}((\sqrt{5}-1)/2,1))\approx1/2$ from \cref{AMfig3}. For $\alpha=C$, \cref{AMfig4} provides evidence that $\mathrm{dim}_{\mathrm{H}}(\spec_{+}(C,1))<1/2$. However, we stress to the reader that computing the Hausdorff dimension of the spectrum requires two limits and requires one limit to provide a lower bound. We cannot, of course, be sure that the first limits in \cref{AMfig4} have converged.

\begin{figure}
\centering
\raisebox{-0.5\height}{\includegraphics[width=0.441\textwidth,trim={0mm 0mm 0mm 0mm},clip]{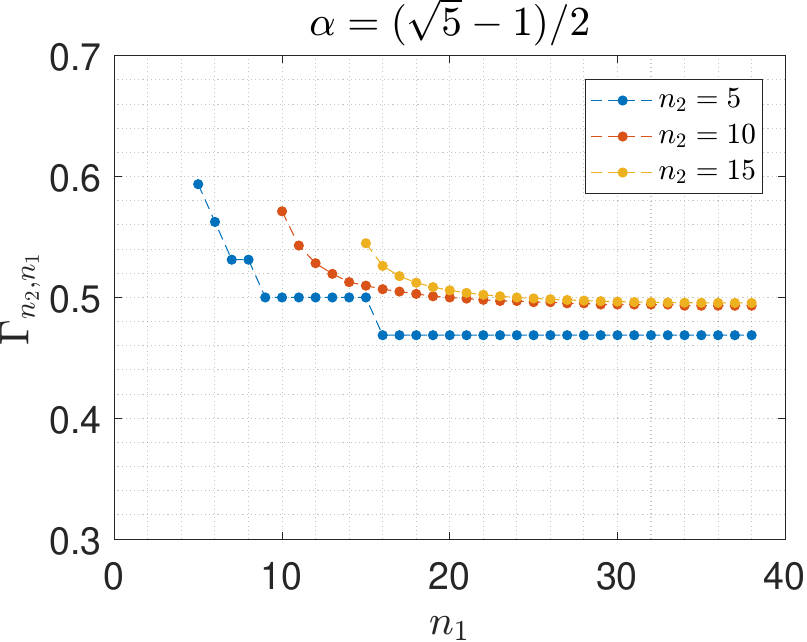}}
\hfill
\raisebox{-0.5\height}{\includegraphics[width=0.441\textwidth,trim={0mm 0mm 0mm 0mm},clip]{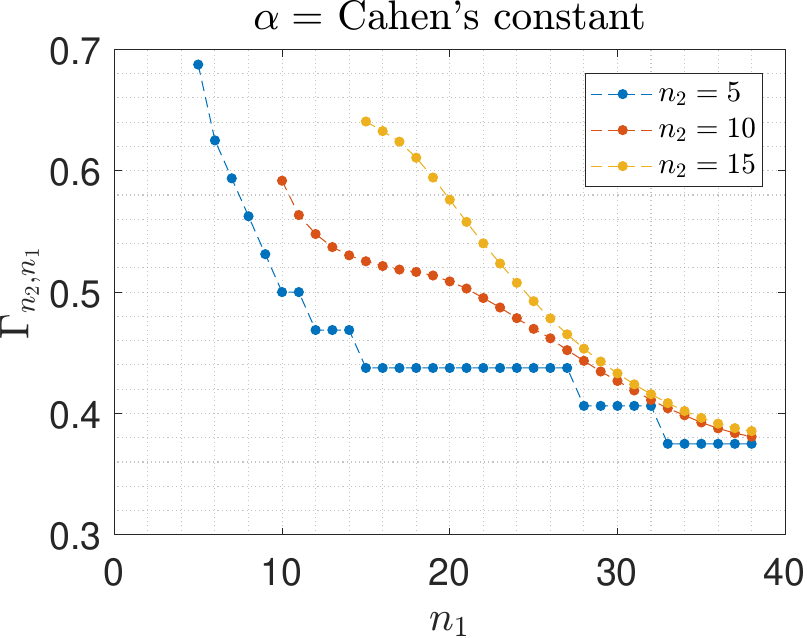}}
\caption{Results for the Hausdorff dimension of $\spec_{+}(\alpha,1)$ for two different values of $\alpha$. On the left, we see apparent convergence $\lim_{n_2\rightarrow\infty}\lim_{n_1\rightarrow\infty}\Gamma_{n_2,n_1}=1/2$. On the right, we see evidence that $\mathrm{dim}_{\mathrm{H}}(\spec_{+}(C,1))<1/2$. Both cases show the monotonicity in the successive limits and the $\Sigma_2^A$ classification from \cref{thm_Haus}.}
\label{AMfig4}
\end{figure}

\subsection{The Fibonacci Hamiltonian}
\label{sec:Fibonacci_numerics}

We now consider the Fibonacci Hamiltonian, the central model in the study of one-dimensional quasicrystals. For coupling constant $\lambda\geq 0$ and phase $\omega\in[0,1)$, the Hamiltonian is defined on $\ell^2(\mathbb{Z})$ by
\begin{equation}
\label{eq:fib_def}
[H_{\lambda,\omega}\psi](n)
=\psi(n-1)+\psi(n+1) + \lambda\chi_{[1-\alpha,1)}(n\alpha+\omega\,\,\mathrm{mod}\, 1)\psi(n),
\end{equation}
where $\alpha=(\sqrt{5}-1)/2$. Along with the almost Mathieu operator, the Fibonacci Hamiltonian is among the most heavily studied Schr\"odinger operators, with hundreds of mathematics and physics papers devoted to it. Its study was originally proposed in the physics literature in \cite{KohKadTan1983PRL} and \cite{OPRSS1983}, as this model can be explicitly dealt with via renormalization group techniques. The first mathematics papers on this operator were \cite{casdagli1986symbolic} and \cite{Suto1987CMP}, with the most comprehensive mathematical results in \cite{DamGorYes2016Invent}. The spectrum of $H_{\lambda,\omega}$ is independent of $\omega$ \cite{bellissard1989spectral}, which can be seen from strong operator convergence. Consequently, we set $\omega=0$ and write $H_{\lambda}=H_{\lambda,0}$ in what follows. 

Much is known about the spectral properties of $H_{\lambda}$, which are surveyed in \cite{damanik2012spectral, damanik2007strictly}. The spectrum $\spec(H_\lambda)$ is a Cantor set of zero Lebesgue measure for every $\lambda > 0$ \cite{sutHo1989singular} and is purely singular continuous \cite{damanik1999uniform}. Moreover, $\spec(H_\lambda)$ is a dynamically defined Cantor set \cite{DamGorYes2016Invent}, meaning that it belongs to a special class of Cantor sets with strong self-similarity properties \cite{palis1995hyperbolicity}. In particular, for any $E\in\spec(H_\lambda)$ and $\epsilon>0$,
\begin{align*}
\mathrm{dim}_{\mathrm{H}}\big((E-\epsilon,E+\epsilon)\cap\spec(H_\lambda)\big)&= \mathrm{dim}_{\mathrm{B}}\big((E-\epsilon,E+\epsilon)\cap\spec(H_\lambda)\big)\\
&= \mathrm{dim}_{\mathrm{H}}\big(\spec(H_\lambda)\big)= \mathrm{dim}_{\mathrm{B}}\big(\spec(H_\lambda)\big),
\end{align*}
and $\mathrm{dim}_{\mathrm{H}}(\spec(H_\lambda))$ is an analytic function of $\lambda$. There are constants $C_1, C_2 > 0$ such that $1-C_1\lambda\leq \mathrm{dim}_{\mathrm{H}}(\spec(H_\lambda))\leq 1-C_2\lambda$ for sufficiently small $\lambda>0$ \cite{damanik2011spectral}, and for any $\lambda>0$, $0<\mathrm{dim}_{\mathrm{H}}(\spec(H_\lambda))<1$ \cite{cantat2009bers}. In \cite{damanik2008fractal} it was shown that
$$
\lim_{\lambda\rightarrow\infty} \mathrm{dim}_{\mathrm{H}}(\spec(H_\lambda))\cdot\log(\lambda)=\log\big(1+\sqrt{2}\big).
$$

From this operator, we can build a simple two-dimensional variant, the square Fibonacci Hamiltonian, defined on $\ell^2(\mathbb{Z}^2)$ by
\begin{align*}
[H_{\lambda}^{(2)} \psi](m,n)
&=\psi(m-1,n) + \psi(m+1,n) + \psi(m,n-1) + \psi(m,n+1) \\
&\quad+ \lambda\left[\chi_{[1-\alpha,1)}(m\alpha\,\,\mathrm{mod}\, 1)+\chi_{[1-\alpha,1)}(n\alpha\,\,\mathrm{mod}\, 1)\right]\psi(m,n),
\end{align*}
and the three-dimensional cubic Fibonacci Hamiltonian, defined on $\ell^2(\mathbb{Z}^3)$ by
\begin{align*}
&\hspace{-.5cm} [H_{\lambda}^{(3)} \psi](\ell,m,n) \\
&=\psi(\ell-1,m,n) + \psi(\ell+1,m,n) +\psi(\ell,m-1,n) \\
& \quad +\psi(\ell,m+1,n)+\psi(\ell,m,n-1) + \psi(\ell,m,n+1) \\
&\quad+ \lambda\left[\chi_{[1-\alpha,1)}(\ell\alpha\,\,\mathrm{mod}\, 1)+\chi_{[1-\alpha,1)}(m\alpha\,\,\mathrm{mod}\, 1)+\chi_{[1-\alpha,1)}(n\alpha\,\,\mathrm{mod}\, 1)\right]\psi(\ell,m,n).
\end{align*}
See~\cite{EL08} for early work on these higher dimensional operators in the physics literature, and \cite[Section~7.3]{damanik2012spectral} for further numerical calculations. Since $H_{\lambda}^{(2)}$ and $H_{\lambda}^{(3)}$ can be expressed via tensor products of the one-dimensional model $H_{\lambda}$, we have
$$
\spec(H_{\lambda}^{(2)})=\{a+b:a,b\in \spec(H_{\lambda})\},\quad \spec(H_{\lambda}^{(3)})=\{a+b+c:a,b,c\in \spec(H_{\lambda})\}.
$$
These relations immediately imply that
\begin{align*}
\overline{\mathrm{dim}}_{\mathrm{B}}\big(\spec(H_{\lambda}^{(2)})\big) &\leq \min\left\{2\,\,\overline{\mathrm{dim}}_{\mathrm{B}}(\spec(H_{\lambda})),1\right\},\\
\overline{\mathrm{dim}}_{\mathrm{B}}\big(\spec(H_{\lambda}^{(3)})\big) &\leq \min\left\{3\,\,\overline{\mathrm{dim}}_{\mathrm{B}}(\spec(H_{\lambda})),1\right\}.
\end{align*}
It is known that $\spec(H_\lambda^{(2)})$ is an interval for sufficiently small $\lambda>0$ \cite{damanik2011spectral}. However, for sufficiently large $\lambda$, both $\spec(H_{\lambda}^{(2)})$ and $\spec(H_{\lambda}^{(3)})$ are Cantor sets of Hausdorff dimension strictly smaller than one, and have zero Lebesgue measure. Moreover, Yessen \cite{yessen2014hausdorff} has shown that
$$
\mathrm{dim}_{\mathrm{H}}\big(\spec(H_{\lambda}^{(2)})\big) = \mathrm{dim}_{\mathrm{B}}\big(\spec(H_{\lambda}^{(2)})\big)=\min\left\{2\,\,\mathrm{dim}_{\mathrm{H}}\big(\spec(H_\lambda)\big),1\right\},
$$
for all but at most countably many $\lambda$. Nevertheless, in general, much less is known about the sets $\spec(H_{\lambda}^{(2)})$ and $\spec(H_{\lambda}^{(3)})$ than $\spec(H_{\lambda})$.

\begin{figure}
\begin{center}
\hspace*{0.15in} $\lambda = 3/4$ \hspace*{2.3in} $\lambda=5/4$

\smallskip
\includegraphics[width=0.441\textwidth]{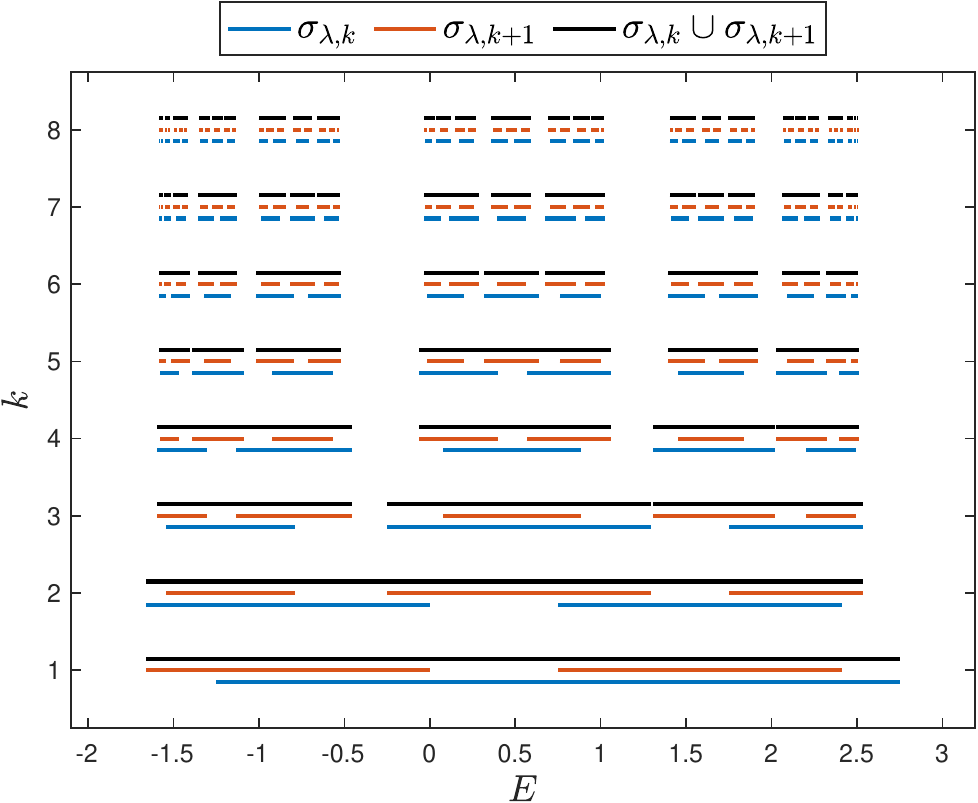}\quad
\includegraphics[width=0.441\textwidth]{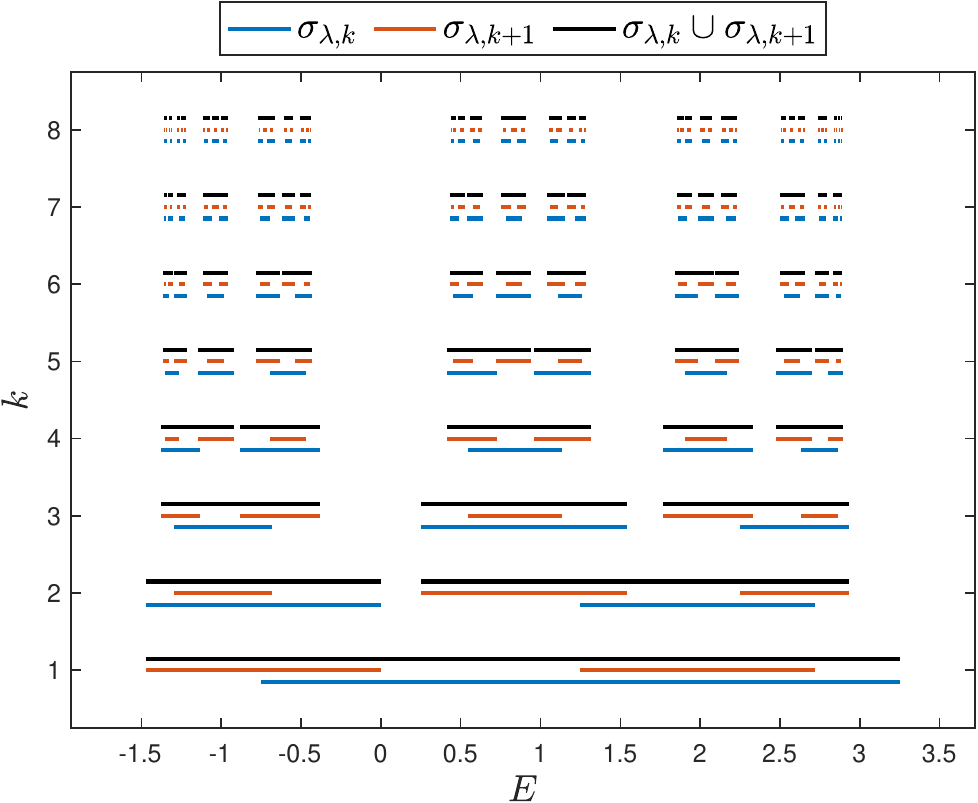}
\end{center}

\vspace*{-1em}
\caption{Covers of the spectrum of the Fibonacci Hamiltonian, $\spec(H_{\lambda})$, 
with coupling constant $\lambda=3/4$ (left) and $\lambda=5/4$ (right).
For each level $k$, $\sigma_{\lambda,k}$ denotes the spectrum of the 
$k$th periodic approximation to Fibonacci potential.  The union 
$\sigma_{\lambda,k} \cup \sigma_{\lambda,k+1}$ of two consecutive
periodic approximations covers the spectrum, $\spec(H_\lambda)$.
The fine-scale that emerges for the higher levels of $k$ here hints at
the Cantor nature of $\spec(H_\lambda)$. For $\lambda=3/4$, the $k=8$ cover has 26~disjoint intervals, the smallest of which has width $9.96\times 10^{-3}$; the smallest gap is $5.01\times 10^{-4}$.
 The analogous cover for $\lambda=5/4$ has 36~disjoint intervals, the smallest of which has width $7.73\times 10^{-3}$; the smallest gap is $2.56\times 10^{-5}$. (Zoom in to see fine structure for the larger values of $k$.)
}
\label{fig:fib_cover_plot}
\end{figure}

As early as 1987, S\"ut\H{o} showed that one can use periodic approximations of the aperiodic potential to obtain covers of the spectrum $\spec(H_\lambda)$  \cite{Suto1987CMP}.  Begin by defining $\alpha_k = F_{k-1}/F_k$ for $k=1,2,\ldots$, where $F_k$ denotes the $k$th Fibonacci number ($F_0=F_1=1$, $F_{k+1}=F_{k-1}+F_{k}$). One approximates the Fibonacci Hamiltonian~(\ref{eq:fib_def}) by replacing the irrational $\alpha$ with its rational approximation $\alpha_k$:
\begin{equation}
\label{eq:fibk_def}
[H_{\lambda,k}\psi](n)
=\psi(n-1) +\psi(n+1) + \lambda\chi_{[1-\alpha_k,1)}(n\alpha_k\,\,\mathrm{mod}\, 1)\psi(n).
\end{equation}
For $\lambda>0$, the operator $H_{\lambda,k}$ has a periodic potential with period~$F_k$, and thus its spectrum is the union of $F_k$ real intervals, which we denote by $\sigma_{\lambda,k}$. We can compute this spectrum using the same techniques we used for the almost Mathieu operator. S\"ut\H{o}~\cite[Proposition~5]{Suto1987CMP} showed that two consecutive periodic approximations give an upper bound for $\spec(H_\lambda)$:
\begin{equation} \label{eq:fib_cover}
   \spec(H_\lambda) \ \subset\  \sigma_{\lambda,k} \cup \sigma_{\lambda,k+1},
\end{equation}
and that these covers converge to $\spec(H_\lambda)$, in the sense that for any $N\ge 1$,
$$
\sigma_{\lambda,k+1} \cup \sigma_{\lambda,k+2}\ \subset\ \sigma_{\lambda,k} \cup \sigma_{\lambda,k+1},\quad
\spec(H_\lambda) = \lim_{k\rightarrow\infty} \big(\sigma_{\lambda,k} \cup \sigma_{\lambda,k+1}\big)=\bigcap_{k} \big( \sigma_{\lambda,k} \cup \sigma_{\lambda,k+1}\big).
$$
To satisfy \textbf{(S1)}, we use the fact that each band of $\sigma_k$ intersects $\spec(H_\lambda)$ \cite[Lemma 2]{damanik2008fractal}. Let $\ell_k$ be the maximum of lengths of the bands that form $\sigma_{\lambda,k}$ and $\sigma_{\lambda,k+1}$. Then
\begin{gather}
d_{\mathrm{H}}\big(\sigma_{\lambda,k} \cup \sigma_{\lambda,k+1},\spec(H_\lambda)\big)\leq \ell_k\\
d_{\mathrm{H}}\big((\sigma_{\lambda,k} \cup \sigma_{\lambda,k+1})+(\sigma_{\lambda,k} \cup \sigma_{\lambda,k+1}),\spec(H_\lambda^{(2)})\big)\leq 2\ell_k\\
d_{\mathrm{H}}\big((\sigma_{\lambda,k} \cup \sigma_{\lambda,k+1})+(\sigma_{\lambda,k} \cup \sigma_{\lambda,k+1})+(\sigma_{\lambda,k} \cup \sigma_{\lambda,k+1}),\spec(H_\lambda^{(3)})\big)\leq 3\ell_k.
\end{gather}
Since, for any $\lambda>0$, $\lim_{k\rightarrow\infty}\ell_k=0$, we immediately obtain a $\Delta_1^A$-tower of algorithms for any of $\spec(H_{\lambda})$, $\spec(H_{\lambda}^{(2)})$ and $\spec(H_{\lambda}^{(3)})$. \Cref{fig:fib_cover_plot} shows, for representative values $\lambda=3/4$ and $\lambda=5/4$, how the covers $\sigma_{\lambda,k} \cup \sigma_{\lambda,k+1}$ evolve as $k$ increases. These plots suggest how $\spec(H_\lambda)$ varies with $\lambda$, requiring us to adjust the approximation level $k$ to suit the coupling constant $\lambda$ in our subsequent computations of fractal dimension.

\begin{figure}
\centering
\raisebox{-0.5\height}{\includegraphics[width=0.49\textwidth,trim={0mm 0mm 0mm 0mm},clip]{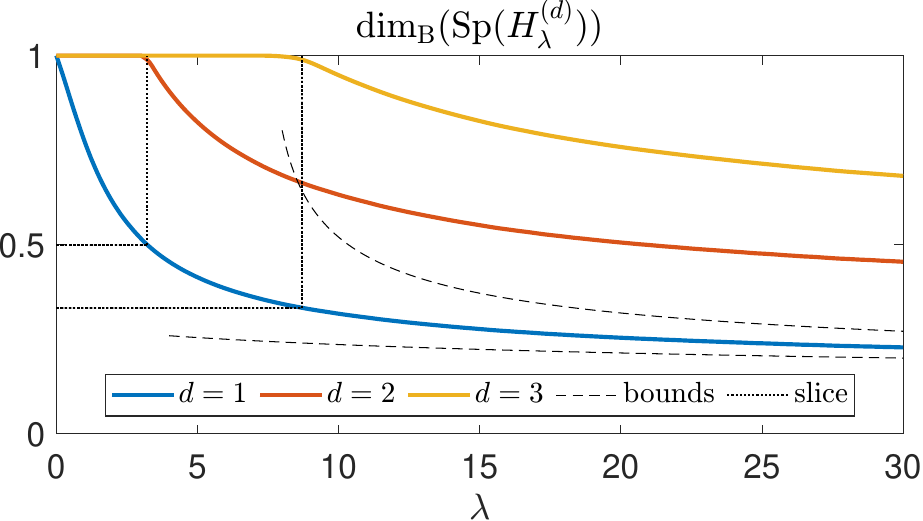}}\vspace{5mm}
\raisebox{-0.5\height}{\includegraphics[width=0.49\textwidth,trim={0mm 0mm 0mm 0mm},clip]{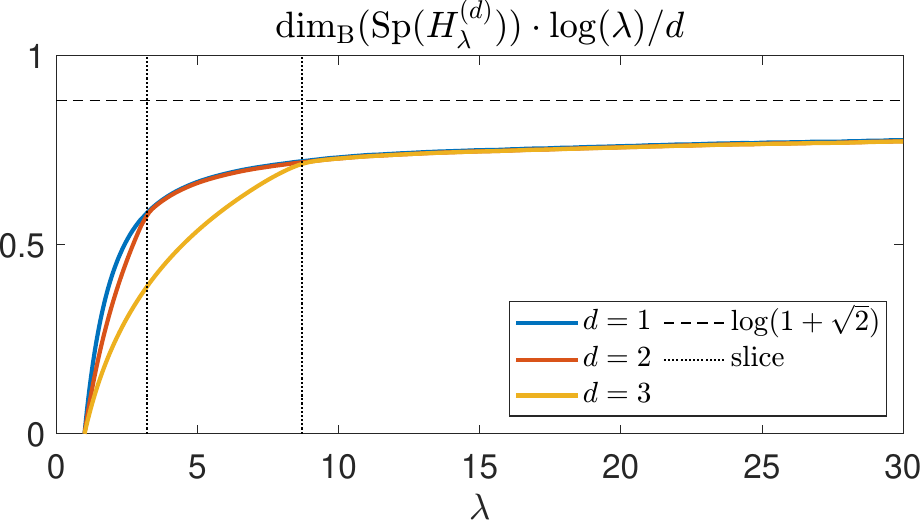}}

\vspace*{-15pt}
\caption{Fractal dimensions for the Fibonacci Hamiltonian ($H_\lambda \equiv H_\lambda^{(1)}$) and the square ($H_\lambda^{(2)}$) and cubic ($H_\lambda^{(3)}$) versions computed as a function of the coupling constant $\lambda$. In the top plot, we have included known bounds on the fractal dimension for the standard one-dimensional model. The slices correspond to where $\mathrm{dim}_{\mathrm{B}}(H_{\lambda})=1/2$ and $1/3$. The bottom plot provides evidence that the box-counting dimension scales like $d/\log(\lambda)$ as $\lambda$ increases.
}
\label{FIBfig1}
\end{figure}

For a desired resolution, the size of the required periodic approximations depends on $\lambda$. We chose approximations to ensure an error ranging from $10^{-6}$ for small $\lambda$ to $10^{-11}$ for large $\lambda$. This allowed us to obtain a scaling region for the box-counting dimension. \cref{FIBfig1} shows the computed fractal dimensions for the Fibonacci, square Fibonacci, and cubic Fibonacci Hamiltonians as $\lambda$ varies. The dashed lines in the top plot show the known bounds \cite{damanik2012spectral}:
\begin{align*}
\mathrm{dim}_{\mathrm{H}}(\spec(H_\lambda))&\leq \frac{\log(1+\sqrt{2})}{\log\left(\frac{1}{2}\left[\lambda-4+\sqrt{(\lambda-4)^2-12}\right]\right)},\quad &&\lambda\geq 8\\
\mathrm{dim}_{\mathrm{H}}(\spec(H_\lambda))&\geq \frac{\log(1+\sqrt{2})}{\log(2\lambda+22)},\quad &&\lambda\geq 4.
\end{align*}
For the Fibonacci Hamiltonian $H_{\lambda}$, the fractal dimensions are less than one for $\lambda>0$. A linear fit for small $\lambda$ suggests that $\mathrm{dim}_{\mathrm{B}}(\spec(H_\lambda))\approx 1-0.22\lambda$ for small $\lambda$. For the square Fibonacci Hamiltonian, $\mathrm{dim}_{\mathrm{B}}(H_{\lambda}^{(2)})=1$ for small $\lambda$ and becomes smaller than one for $\lambda\approx 3.2$. For the cubic Fibonacci Hamiltonian, $\mathrm{dim}_{\mathrm{B}}(H_{\lambda}^{(3)})=1$ for small $\lambda$ and becomes smaller than one for $\lambda\approx 8.7$. These transitions occur approximately where $\mathrm{dim}_{\mathrm{B}}(H_{\lambda})=1/2$ and $1/3$. (The discrepancy is likely due to computations using a finite value of $k$.) The bottom plot rescales the box-counting dimensions, providing evidence that
$$
\mathrm{dim}_{\mathrm{B}}(\spec(H_\lambda^{(d)}))=\min\left\{1,d\cdot\mathrm{dim}_{\mathrm{B}}(H_{\lambda})\right\}.
$$
It is not known whether this formula holds, apart from the case of $d=2$, where it is known to hold for all but countably many $\lambda$. We found that the Hausdorff dimension agreed with the box-counting dimension. For example, \cref{FIBfig2} shows the output of the $\Sigma_2^A$ algorithm for the Hausdorff dimension for $\lambda=15$, demonstrating the apparent convergence.

Finally, \cref{FIBfig4} shows the number of connected components for various covers. For the cubic Fibonacci Hamiltonian, there is a significant jump in the number of connected components around $\lambda=1+\sqrt{2}$ (shown as a black line). This provides evidence that there is a region of $\lambda$ for which the number of connected components is infinite, but the fractal dimension is~1.  
(See \cite[Tables~1 and~2]{damanik2012spectral} for numerical approximations of the
$\lambda$ values at which gaps in $\spec(H_\lambda^{(2)})$ and $\spec(H_\lambda^{(3)})$ 
emerge as $\lambda$ increases.)

\begin{figure}
\centering
\raisebox{-0.5\height}{\includegraphics[width=0.32\textwidth,trim={0mm 0mm 0mm 0mm},clip]{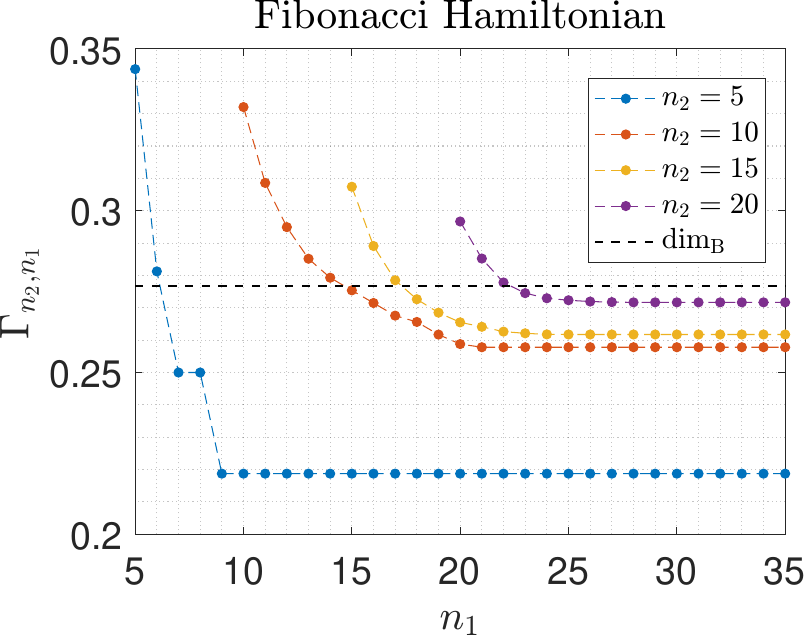}}\hfill
\raisebox{-0.5\height}{\includegraphics[width=0.32\textwidth,trim={0mm 0mm 0mm 0mm},clip]{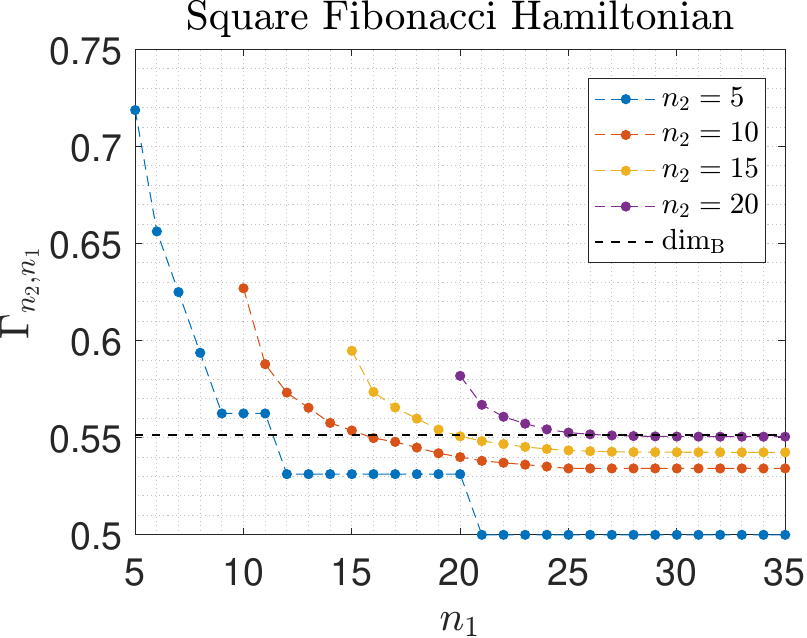}}\hfill
\raisebox{-0.5\height}{\includegraphics[width=0.32\textwidth,trim={0mm 0mm 0mm 0mm},clip]{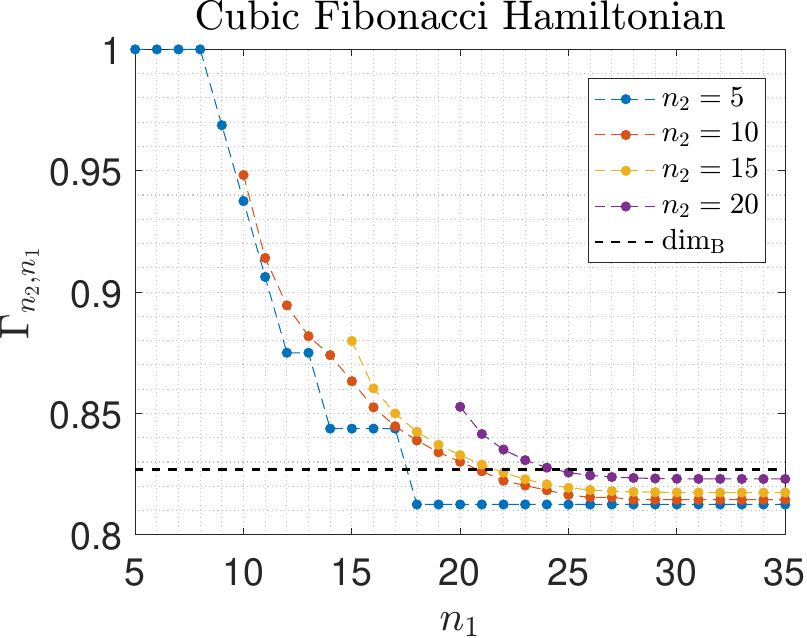}}
\caption{Output of $\Sigma_2^A$ algorithm for the Hausdorff dimension at $\lambda=15$.}
\label{FIBfig2}
\end{figure}

\begin{figure}
\centering
\raisebox{-0.5\height}{\includegraphics[width=0.49\textwidth,trim={0mm 0mm 0mm 0mm},clip]{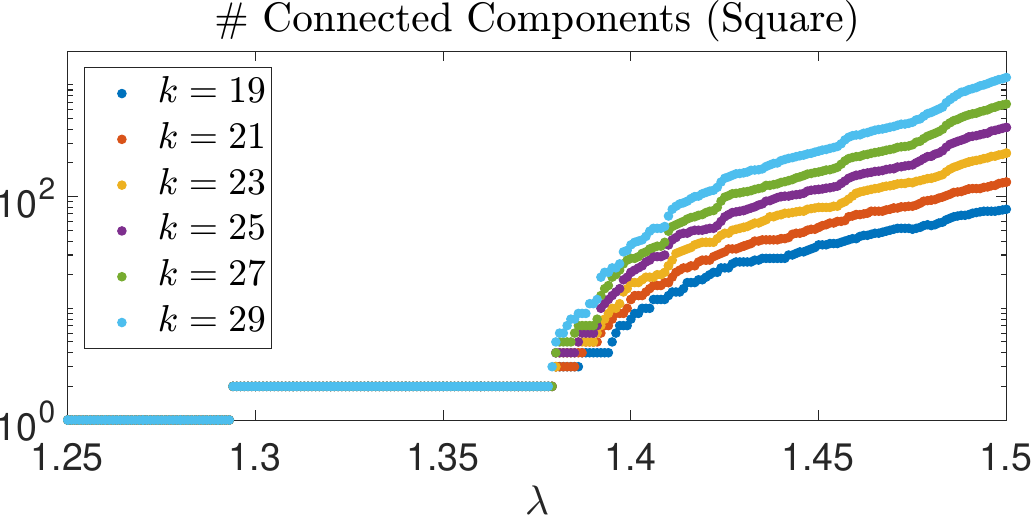}}\vspace{5mm}
\raisebox{-0.5\height}{\includegraphics[width=0.49\textwidth,trim={0mm 0mm 0mm 0mm},clip]{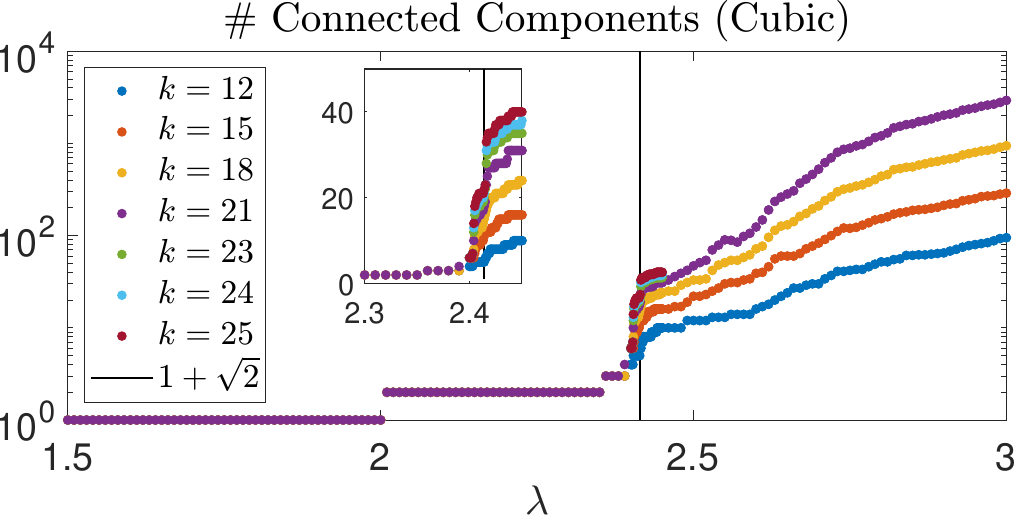}}
\vspace*{-15pt}
\caption{The number of connected components of the coverings of the spectrum for the square and cubic Fibonacci Hamiltonians. As $k$ increases, these approximations converge up to the true answer, providing a $\Sigma_1^A$ algorithm for the number of connected components.}
\label{FIBfig4}
\end{figure}

\section{Methods Under \textbf{(S2)}}
\label{sec:IImethods}

In \cref{sec:Imethods}, we showed how to compute various properties of the spectrum under the assumption \textbf{(S1)}. For many operators, it can be difficult to realize such an assumption. However, we can compute suitable covers under assumption \textbf{(S2)} at the cost of an additional limit. We will do this via a covering lemma, before showing how this leads to classifications of spectral problems. We then demonstrate an important operator class that satisfies assumption \textbf{(S2)}.

\subsection{Drilling holes in the spectrum: A covering lemma}
\label{sec:covering_lemma}

Given functions $\{\Phi_n\}$ that satisfy assumption \textbf{(S2)} (i.e., provide upper bounds on the distance to the spectrum), we compute covers of the spectrum that satisfy \textbf{(S1)} at the cost of an additional limit. For $\delta > 0$ and $R > 0$, define the grid
$$
G_{\delta,R}=\{z_j = j\cdot \delta:j\in\mathbb{Z}, |z_j|\leq R\}\subset\mathbb{R}.
$$
Given $n_1,n_2\in\mathbb{N}$, set
$$
\widehat{\Phi}_{n_1}(n_2;z,A)=\sup\left\{\frac{k}{n_2}:k\in\mathbb{Z}_{\geq 0},\;\frac{k}{n_2}\leq{\Phi}_{n_1}(z,A)\right\}=\frac{1}{n_2}\Big\lfloor n_2 \Phi_{n_1}(z, A) \Big\rfloor,
$$
and define
\begin{equation}
\label{S2_cover_alg}
\Gamma_{n_2,n_1}^\spec(A) = [-n_1,n_1]\Big\backslash \bigcup_{w\in G_{1/n_2,n_1}} D(w,\widehat{\Phi}_{n_1}(n_2;w,A)),
\end{equation}
where $D(x,r)$ is the open ball of radius $r$ centered at $x$ and is taken to be the empty set if $r=0$. Without loss of generality by computing successive minima, the functions $\Phi_n(z,A)$ are non-increasing in $n$. The intuition here is that we are removing balls not included in the spectrum from the interval $[-n_1,n_1]$ to obtain covers of the spectrum. This approach is particularly suitable for operators that have fractal-like spectra.

\begin{proposition}
\label{lem:cover1}
For any fixed $n_2$ and $A\in\Omega$, $\Gamma_{n_2,n_1}^\spec(A)$ is eventually constant for large $n_1$ and the limit set $\Gamma_{n_2}^\spec(A)=\lim_{n_1\rightarrow\infty}\Gamma_{n_2,n_1}^\spec(A)$ satisfies
$
\spec(A)\subset \Gamma_{n_2}^\spec(A) \subset \spec(A)+B_{2/n_2}(0),
$
where $B_r(0)$ denotes the closed ball of radius $r>0$ centered at the origin.
\end{proposition}

\begin{proof}
The convergence of the functions $\{\Phi_n\}$ from above implies that
$$
\lim_{n_1\rightarrow\infty}\!\widehat{\Phi}_{n_1}(n_2;z,A)\!=\!\widehat{\Phi}(n_2;z,A)\!=\!\sup\left\{\frac{k}{n_2}:k\in\mathbb{Z}_{\geq 0},\frac{k}{n_2}\!\leq\!\dist(z,\spec(A))\right\}.
$$
For any fixed $n_2$, $z$ and $A$, $\widehat{\Phi}_{n_1}(n_2;z,A)$ can only take values on a fixed grid of points separated by $1/n_2$, and so must be eventually constant for sufficiently large $n_1$. Let $L\in\mathbb{N}$ be such that $\spec(A)\subset {[-(L-2)},L-2]$. Since $\lim_{|z|\rightarrow\infty}\dist(z,\spec(A))=\infty$ and $\widehat{\Phi}_{n_1}(n_2;z,A)\geq \dist(z,\spec(A))-1/n_2$, $\Gamma_{n_2,n_1}^\spec(A)$ is constant for large $n_1$ with
$$
\Gamma_{n_2}^\spec(A)=\lim_{n_1\rightarrow\infty}\Gamma_{n_2,n_1}^\spec(A)=[-L,L]\Big\backslash \bigcup_{w\in G_{1/n_2,L}} D(w,\widehat{\Phi}(n_2;w,A)).
$$
If $z\in\spec(A)$ and $w\in G_{1/n_2,L}$,
$
|z-w|\geq \dist(w,\spec(A)) \geq \widehat{\Phi}(n_2;w,A)
$
and hence
$$
z\not\in \bigcup_{w\in G_{1/n_2,L}} D(w,\widehat{\Phi}(n_2;w,A)).
$$
It follows that $\spec(A)\subset \Gamma_{n_2}^\spec(A)$. On the other hand, if $z\in \Gamma_{n_2}^\spec(A)$, let $w$ be a nearest point in $G_{1/n_2,L}$ to $z$. Then
$
\widehat{\Phi}(n_2;w,A)<|w-z|\leq {1}/{(2n_2)}
$
so that
$$
\dist(z,\spec(A))\leq |w-z|+ \dist(w,\spec(A))\leq |w-z|+ \widehat{\Phi}(n_2;w,A)+\frac{1}{n_2}\leq \frac{2}{n_2}.
$$
The final part of the lemma follows.
\end{proof}

\subsection{Classification theorems}
\label{sec:covering_classifications}

We can now prove the analog of \cref{thm_warm_up} with assumption \textbf{(S2)} replacing assumption \textbf{(S1)}. The cost is an additional limit through the use of \cref{lem:cover1}. The lower bounds in \cite{colbrook2022computation} for diagonal operators (which satisfy \textbf{\rm\textbf{(S2)}}) show that the classifications we prove in this section are, in general, sharp.

\begin{lemma}
\label{thm_warm_up2}
Suppose that \textbf{\rm\textbf{(S2)}} holds and let $\Xi$ be a problem function such that \eqref{limit_well-behaved} holds, and we can compute $f(\Gamma_{n_2,n_1}^\spec(A))$ of a finite union of compact intervals to any given accuracy. Then $\{\Xi,\Omega,\mathbb{R}\cup\{\pm\infty\},\Lambda\}\in\Delta_3^A.$ Moreover, if the limit in \eqref{limit_well-behaved} is from above or below, then $\{\Xi,\Omega,\mathbb{R}\cup\{\pm\infty\},\Lambda\}\in\Pi_2^A$ or $\{\Xi,\Omega,\mathbb{R}\cup\{\pm\infty\},\Lambda\}\in\Sigma_2^A$, respectively.
\end{lemma}

\begin{proof}
Let $A\in\Omega$. Then \cref{lem:cover1} shows that
$$
\spec(A)\subset \Gamma_{n_2}^\spec(A),\quad \max_{z\in \Gamma_{n_2}^\spec(A)}\mathrm{dist}(z,\spec(A))\leq 2/n_2.
$$
By enlarging the intervals of $\Gamma_{n_2,n_1}^\spec(A)$ as necessary, we can ensure that the sequence $\{\Gamma_{n_2}^\spec\}$ satisfies the requirements of \textbf{(S1)}, apart from being an arithmetic algorithm. From \cref{lem:cover1}, we may compute an appropriate $\Gamma_n^{\mathrm{Sp}}$ using an additional limit, which is eventually constant. This means that all of the classifications in \cref{thm_warm_up} carry through with the penalty of an additional limit. Note that this argument relies on the crucial property that the arithmetic tower of algorithms in \cref{lem:cover1} is eventually constant.
\end{proof}

Given the examples following \cref{thm_warm_up}, the following corollary is immediate.

\begin{corollary}
\label{cor_assumII_extra_lims}
Suppose that \textbf{\rm\textbf{(S2)}} holds. Then
\begin{itemize}
	\item Number of connected components:
	$
	\{\Xi_{\mathrm{cc}},\Omega,\mathbb{N}\cup\{\infty\},\Lambda\}\in\Sigma_2^A.
	$
\item Lebesgue measure:
$
\{\Xi_{\mathrm{Lm}},\Omega,\mathbb{R}_{\geq0},\Lambda\}\in\Pi_2^A.
$
\item Capacity:
$
\{\Xi_{\mathrm{cap}},\Omega,\mathbb{R}_{\geq0},\Lambda\}\in\Pi_2^A.
$
\item For the decision problems $\Xi=\Xi_{\mathrm{cc}}^{\mathrm{dec}},\Xi_{\mathrm{Lm}}^{\mathrm{dec}}$ or $\Xi_{\mathrm{cap}}^{\mathrm{dec}}$:
$
\{\Xi,\Omega,\{0,1\},\Lambda\}\in\Pi_3^A.
$
\end{itemize}
\end{corollary}

A similar result holds for computing the fractal dimensions of the spectrum.

\begin{theorem}
\label{fractal_dim_final}
Suppose that \textbf{\rm\textbf{(S2)}} holds. Then
\begin{align*}
\{\underline{\mathrm{dim}}_{\mathrm{B}}(\spec(\cdot)),\Omega,[0,1],\Lambda\}\in\Sigma_3^A,\quad &\{\overline{\mathrm{dim}}_{\mathrm{B}}(\spec(\cdot)),\Omega,[0,1],\Lambda\}\in\Pi_2^A,\\
\{\mathrm{dim}_{\mathrm{H}}(\spec(\cdot)),\Omega,[0,1],\Lambda\}\in\Sigma_3^A,\quad&\{\mathrm{dim}_{\mathrm{B}}(\spec(\cdot)),\widehat{\Omega},[0,1],\Lambda\}\in\Pi_2^A,
\end{align*}
where $\widehat{\Omega}=\{A\in\Omega:\underline{\mathrm{dim}}_{\mathrm{B}}(\spec(A))=\overline{\mathrm{dim}}_{\mathrm{B}}(\spec(A))\}$.
\end{theorem}

\begin{proof}
The classifications for $\underline{\mathrm{dim}}_{\mathrm{B}}(\spec(\cdot))$ and $\mathrm{dim}_{\mathrm{H}}(\spec(\cdot))$ follow the same argument used in the proof of \cref{thm_warm_up2}, but now using \cref{thm:box_counting} and \cref{thm_Haus}, respectively. The only classifications left to prove are for the problem functions $\overline{\mathrm{dim}}_{\mathrm{B}}(\spec(\cdot))$ and $\mathrm{dim}_{\mathrm{B}}(\spec(\cdot))$. For any class $\Omega$, the problem function $\mathrm{dim}_{\mathrm{B}}(\spec(\cdot))$ reduces to $\overline{\mathrm{dim}}_{\mathrm{B}}(\spec(\cdot))$ when restricted to the subclass $\widehat{\Omega}$. Hence it is enough to consider the problem function $\overline{\mathrm{dim}}_{\mathrm{B}}(\spec(\cdot))$.

Let $\widetilde{\Gamma}_{n,m}=\Gamma_{2^{n},m}^\spec$ and $\widetilde{\Gamma}_{n}=\lim_{m\rightarrow\infty}\widetilde{\Gamma}_{n,m}$. \cref{lem:cover1}'s proof shows there exists $M\in\mathbb{N}$ with $\widetilde{\Gamma}_{n,m}(A)\subset\widetilde{\Gamma}_{n,m+1}(A)$ for all $m\geq M$, independent of $n$. The key point is that, once $m$ is sufficiently large, the grid of points $G_{1/n,m}$ in $[-L,L]$ used in the construction of $\widetilde{\Gamma}_{n,m}(A)$ stabilizes and those grid points outside $[-L,L]$ do not affect the set $\widetilde{\Gamma}_{n,m}(A)$  (the threshold depending only on the stabilization argument in \cref{lem:cover1}, and not on $n$). So increasing $m$ can only add points to the resulting enclosure, yielding the eventual nesting. Hence,
$$
\lim_{n_1\rightarrow\infty}\max_{n_2\leq k\leq n_1} \frac{\log(N_{2^{-k}}(\widetilde{\Gamma}_{k+1,n_1}(A)))}{\log(2^k)}=\sup_{k\geq n_2} \frac{\log(N_{2^{-k}}(\widetilde{\Gamma}_{k+1}(A)))}{\log(2^k)}.
$$
Let $a_{k,n_1}$ be an approximation of $\log(N_{2^{-k}}(\widetilde{\Gamma}_{k+1,n_1}(A)))/\log(2^k)$ computed to an accuracy $1/n_1$, and define $\Gamma_{n_2,n_1}(A) = \max_{n_2\leq k\leq n_1} a_{k,n_1}.$ Then
$$
\lim_{n_1\rightarrow\infty}\Gamma_{n_2,n_1}(A)=\Gamma_{n_2}(A)=\sup_{k\geq n_2} \frac{\log(N_{2^{-k}}(\widetilde{\Gamma}_{k+1}(A)))}{\log(2^k)}.
$$
\cref{lem:cover1} shows that $\dist(z,\spec(A))\leq 2^{-k}$ for any $z\in \widetilde{\Gamma}_{k+1}(A)$. We argue as in the proof of \cref{thm:box_counting} to deduce that
$
N_{2^{-k}}(\spec(A))\leq N_{2^{-k}}(\widetilde{\Gamma}_{k+1}(A))\leq 3 N_{2^{-k}}(\spec(A)).
$
It follows that
$$
\sup_{k\geq n_2}\frac{\log(N_{2^{-k}}(\spec(A)))}{\log(2^{k})}\leq\Gamma_{n_2}(A)\leq \sup_{k\geq n_2}\frac{\log(3)}{\log(2^{k})} +\frac{\log(N_{2^{-k}}(\spec(A)))}{\log(2^{k})}.
$$
Hence, $\Gamma_{n_2}(A)$ converges to $\overline{\rm dim}_{\rm B}(\spec(A))$ from above as $n_2\rightarrow\infty$.
\end{proof}

There is a surprise in \cref{fractal_dim_final}: the SCI of computing the upper box-counting dimension is one lower than what one would obtain by passing to an additional limit through \cref{thm:box_counting}.
However, the classification for $\underline{\mathrm{dim}}_{\mathrm{B}}(\spec(\cdot))$ cannot be lowered to $\Sigma_2^G$ (or $\Sigma_2^A$). Indeed, if this were the case, then computing the box-counting dimension (in the regime where the upper and lower box-counting dimensions coincide) would lie in $\Sigma_2^G \cap \Pi_2^G = \Delta_2^G$ (see \cite[Proposition 2.2.8]{colbrook2020PhD} for the equality). However, this is not possible in general. In particular, for suitable diagonal operators whose spectra have equal upper and lower box-counting dimensions, the computation of the box-counting dimension does not lie in $\Delta_2^G$.

\subsection{A natural class of operators}
\label{sec:bounded_disp}

To compute a suitable sequence of functions $\{\Phi_n\}$ and realize assumption \textbf{(S2)}, we make the following definition. This definition follows that of \cite{colb1} but is slightly more restrictive in that we consider bounded and self-adjoint operators.

\begin{definition}
\label{def:bounded_dispersion}
Given $f:\mathbb{N}{\rightarrow}\mathbb{N}$ with $f(n)\geq n+1$ and $A\in\mathcal{B}(\ell^2(\mathbb{N}))$ with $A=A^*$, define 
$$
D_{f\!,n}(A)=\|(I-\mathcal{P}_{f\!(n)}^*\mathcal{P}_{f\!(n)})A\mathcal{P}_n^*\|,
$$
where $\mathcal{P}_n$ is the orthogonal projection onto $\mathrm{span}\{e_1,\ldots, e_n\}$ (viewed as a map from $\ell^2(\mathbb{N})$ to $\mathbb{C}^n$, where $\{e_j\}_{j=1}^\infty$ is the canonical basis of $\ell^2(\mathbb{N})$). We say that $A$ has bounded dispersion with respect to $f$ if $\lim\limits_{n\rightarrow\infty}D_{f\!,n}(A)=0$ and set
$$
\Omega_{f,\mathrm{SA}}=\left\{A\in\mathcal{B}(\ell^2(\mathbb{N})): A=A^*,\lim_{n\rightarrow\infty}D_{f\!,n}(A)=0  \right\}.
$$
\end{definition}

Given $A\in\Omega_{f,\mathrm{SA}}$, we assume knowledge of a sequence $\{c_n\}_{n\in \mathbb{N}} \subset \mathbb{Q}$ with $D_{f\!,n}(A)\leq c_n$ and $\lim_{n\rightarrow\infty}c_n=0$.
One can show that for any self-adjoint $A\in\mathcal{B}(\ell^2(\mathbb{N}))$, there exists a suitable $f$ so that $A\in\Omega_{f,\mathrm{SA}}$. However, $f$ is generally unknown, and one can even show it is generally not computable from the matrix elements of $A$ \cite{ben2020can}. Bounded dispersion generalizes the concept of a matrix being banded or sparse by considering the decay of its columns. \cref{def:bounded_dispersion} can be extended to other Hilbert spaces. For instance, in the space $\ell^2(\mathbb{Z})$, one can adopt the same definition but with $\mathcal{P}_n$ defined as the orthogonal projection onto $\mathrm{span}\{e_{-n},\ldots, e_n\}$.

\begin{example}
Suppose that a self-adjoint $A\in\mathcal{B}(\ell^2(\mathbb{N}))$ is sparse with finitely many nonzero entries in each column and row. If we know $f$ such that $\langle A e_j,e_i\rangle=0$ if $i>f(j)$, then $A\in\Omega_{f,\mathrm{SA}}$ with $c_n=0$.\hfill$\blacksquare$
\end{example}

The following example encompasses quasicrystal models on lattices, including the examples in dimensions higher than one that will be the focus of \cref{sec:examples2}.

\begin{example}\label{CHAP3_example:graphs}
Operators on graphs or lattices are ubiquitous in mathematical physics. Consider a connected, undirected graph $\mathcal{G}$ such that the set of vertices $\mathcal{V}=\mathcal{V}(\mathcal{G})$ is countably infinite and any vertex has only finitely many edges. Suppose that a bounded self-adjoint operator is given by
$$
A=\sum_{v, w\in \mathcal{V}}\alpha(v,w)\,\delta_v^{}\delta_w^*,
$$
for some $\alpha: \mathcal{V}\times \mathcal{V} \rightarrow \mathbb{C}$. We identify any $v \in \mathcal{V}$ with the element $\delta_v \in \ell^2(\mathcal{V})$ such that $\delta_v(v) = 1$ and $\delta_v(w) = 0$ for $w \neq v$. For simplicity, assume that for any $v\in \mathcal{V}$, the set of vertices $w$ with $\alpha(v,w)\neq 0$ is finite. Also, with respect to a given enumeration $\{v_1,v_2,\ldots\}$ of $\mathcal{V}$, let $S:\mathbb{N}\rightarrow\mathbb{N}$ be such that if $m>S(n)$, then $\alpha(v_n,v_m)=\alpha(v_m,v_n)=0$. If we allow access to the functions $\alpha$ and $S$, then the induced isomorphism $\ell^2(\mathcal{V}(\mathcal{G}))\cong \ell^2(\mathbb{N})$ naturally leads to a choice of $f$ so that $A$ can be viewed as a member of $\Omega_{f,\mathrm{SA}}$. Starting with $v_1$, we list its neighbors as $S_1=\{v_1,\ldots,v_{q_1}\}$ for some finite $q_1$. We then list the neighbors of vertices (including themselves) in $S_1$ as $S_2$. This process is continued inductively to enumerate each $S_m$. For example, if $A$ is a nearest-neighbor operator on $\ell^2(\mathbb{Z}^d)$, then $\left|S_{m}\right|\sim\mathcal{O}(m^{d})$ while $\left|S_{m+1}\backslash S_m\right|\sim\mathcal{O}(m^{d-1})$. Hence, we can choose a suitable $f$ with $f(n)-n\sim \mathcal{O}(n^{\frac{d-1}{d}})$.\hfill$\blacksquare$
\end{example}

Recall that the class of problems~(S2) requires a family of functions $\{\Phi_n\}$ that bound the distance of a point $z \in \mathbb{C}$ from the spectrum.
To compute such functions for the class $\Omega_{f,\mathrm{SA}}$, our strategy will be to use the function $f$ to construct certain \textit{rectangular} matrix truncations of our operators and compute \textit{singular values}. This approach contrasts with methods that typically employ \textit{square} matrix truncations and compute \textit{eigenvalues}. We let $\sigma_{\mathrm{inf}}$ denote the injection modulus of a bounded operator $T:\mathcal{H}_1\rightarrow\mathcal{H}_2$ between Hilbert spaces:
$$
\sigma_{\mathrm{inf}}(T)=\inf\left\{\|Tx\|:x\in\mathcal{H}_1,\|x\|=1\right\}.
$$
In the case that $T$ is given by a rectangular matrix, $\sigma_{\mathrm{inf}}(T)$ is the smallest singular value of $T$. From the definition of $\Omega_{f,\mathrm{SA}}$, we see that if $A\in\Omega_{f,\mathrm{SA}}$, then
\begin{align}
\nonumber
\left|\sigma_{\mathrm{inf}}((A-zI)\mathcal{P}_n^*)-\sigma_{\mathrm{inf}}(\mathcal{P}_{f\!(n)}(A-zI)\mathcal{P}_n^*)\right|
&\leq \|(A-zI)\mathcal{P}_n^*-\mathcal{P}_{f\!(n)}^*\mathcal{P}_{f\!(n)}(A-zI)\mathcal{P}_n^*\|\\
\nonumber
&=\|(I-\mathcal{P}_{f\!(n)}^*\mathcal{P}_{f\!(n)})A\mathcal{P}_n^*\|\\
& =D_{f\!,n}(A)\leq c_n. \label{eq:spectralinjectionradiusboundcn}
\end{align}
Moreover, since increasing the domain of an operator can only decrease its injection modulus,
$$
\mathrm{dist}(z,\spec(A))=\sigma_{\mathrm{inf}}(A-zI)\leq\sigma_{\mathrm{inf}}((A-zI)\mathcal{P}_n^*), 
$$
where the equality holds since $A$ is self-adjoint. The previous two inequalities suggest setting
\begin{equation}
\label{thm_Phi_candidate}
\Phi_n(z,A)=\sigma_{\mathrm{inf}}(\mathcal{P}_{f\!(n)}(A-zI)\mathcal{P}_n^*)+c_n.
\end{equation}
The next theorem shows that these functions satisfy the requirements of assumption \textbf{(S2)}.

\begin{theorem}
Let $A\in\Omega_{f,\mathrm{SA}}$ and $\Phi_n(z,A)$ be given by \cref{thm_Phi_candidate}. Then
$$
\mathrm{dist}(z,\spec(A))\leq \Phi_n(z,A),\quad\text{and}\quad \lim_{n\rightarrow \infty}\Phi_n(z,A)=\mathrm{dist}(z,\spec(A)),
$$
where the convergence is uniform on compact subsets of $\mathbb{R}$. Moreover, we can compute $\Phi_n(z,A)$ to any given accuracy using finitely many arithmetic operations and comparisons. In particular, the functions $\{\Phi_n\}$ satisfy the requirements of assumption \textbf{\rm\textbf{(S2)}}.
\end{theorem}

\begin{proof}
We have already argued why $\mathrm{dist}(z,\spec(A))\leq \Phi_n(z,A)$. To prove the convergence, let $\epsilon>0$ and choose $x\in \ell^2(\mathbb{N})$ of norm 1 such that $\|(A-zI)x\|\leq \sigma_{\mathrm{inf}}(A-zI)+\epsilon$. Then
$$
\Phi_n(z,A)\leq \frac{\|\mathcal{P}_{f\!(n)}(A-zI)\mathcal{P}_n^*\mathcal{P}_nx\|}{\|\mathcal{P}_nx\|}+c_n.
$$
The first quantity on the right-hand side converges to $\|(A-zI)x\|$ as $n\rightarrow\infty$, while the second converges to zero. It follows that $\limsup_{n\rightarrow\infty}\Phi_n(z,A)\leq \|(A-zI)x\|\leq \sigma_{\mathrm{inf}}(A-zI)+\epsilon$. Since $\epsilon>0$ was arbitrary, the pointwise convergence now follows. Furthermore, we may upgrade the pointwise convergence of $\sigma_{\mathrm{inf}}(\mathcal{P}_{f\!(n)}(A-zI)\mathcal{P}_n^*)$ to local uniform convergence by first applying Dini's theorem to $\sigma_{\mathrm{inf}}((A-zI)\mathcal{P}_n^*)$ and then appealing to \eqref{eq:spectralinjectionradiusboundcn}. The local uniform convergence of $\Phi_n(z,A)$ follows.

It remains to prove the assertion on computability. Consider the self-adjoint matrix
$$
B_n(z)=\begin{pmatrix}
		0 & \mathcal{P}_{n}(A-zI)^*\mathcal{P}^*_{f\!(n)} \\
		\mathcal{P}_{f\!(n)}(A-zI)\mathcal{P}_{n}^* & 0
		\end{pmatrix}\in\mathbb{C}^{(n+f(n))\times (n+f(n))},
$$
and list the singular values of the rectangular matrix $\mathcal{P}_{f\!(n)}(A-zI)\mathcal{P}_{n}^*$ as $\{\sigma_1,\ldots, \sigma_n\}$. The eigenvalues of $B_n(z)$ are $\{\sigma_1,\ldots, \sigma_n,-\sigma_1,\ldots, -\sigma_n,0,\ldots,0\}$ where $0$ appears $f(n)-n$ times \cite[Section 8.6.1]{golub2013matrix}. The final part of the theorem now follows from the computability of the eigenvalues of finite self-adjoint matrices \cite[Corollary 6.9]{colbrook2022foundations}.
\end{proof}

\section{Models of Two-Dimensional Quasicrystals}
\label{sec:examples2}

We now turn to models of two-dimensional quasicrystals. Spectral analysis in dimensions higher than one is often substantially more intricate than in the one-dimensional setting \cite{damanik2012spectral}. While a rich body of analytic results exists for certain one-dimensional Schr\"odinger operators, comparatively little is known in general or in higher dimensions \cite{damanik2015absolutely}. (An exception:  tensor products of one-dimensional models, such as the square and cubic Fibonacci Hamiltonians; see, e.g., \cite{damanik2012spectral,EL08}.) This regime is of particular interest, since physically realistic models of quasicrystals are typically two- or three-dimensional.

As discussed in \cref{sec:prev_work}, there has been significant recent progress toward establishing \textbf{(S1)}-type spectral enclosures for higher-dimensional models under suitable structural assumptions. By contrast, \textbf{(S2)} is applicable under substantially milder hypotheses. In the present section, we instead work under \textbf{(S2)}, which introduces an additional limiting step in the computational procedure. Nevertheless, the framework developed in \cref{sec:bounded_disp} allows us to quantify the achievable spectral resolution and to extract fine structural information from the approximations.

\subsection{The three models}

As model problems for this paper, we consider the Laplacians of three graphs associated with an infinite Penrose--Robinson triangle tiling of the two-dimensional plane, generated by repeated application of inflation rules from a seed of 20~tiles arranged in a decagon; see~\cite[sect~6.2]{BaakeGrimm2013Vol1}, \cite{Fra08} for details about the inflation rules, and \cref{fig:circ_tri_tiling} for an illustration.

The model $T_1$ comes from treating each tile as a graph node and connecting nodes associated with tiles that share a common edge. All interior nodes have degree~3. The model $T_2$ comes directly from the sides and vertices of the triangles in the tiling; in this case, vertices can have degrees as large as~10. The model $T_3$ is derived from $T_2$ by casting out the longest and shortest triangular edges, resulting in a tiling made up of two distinct rhombus prototiles. \Cref{fig:circ_tri_tiling} gives representative examples of all three graphs.

\begin{figure}
\begin{center}
\includegraphics[width=1.4in]{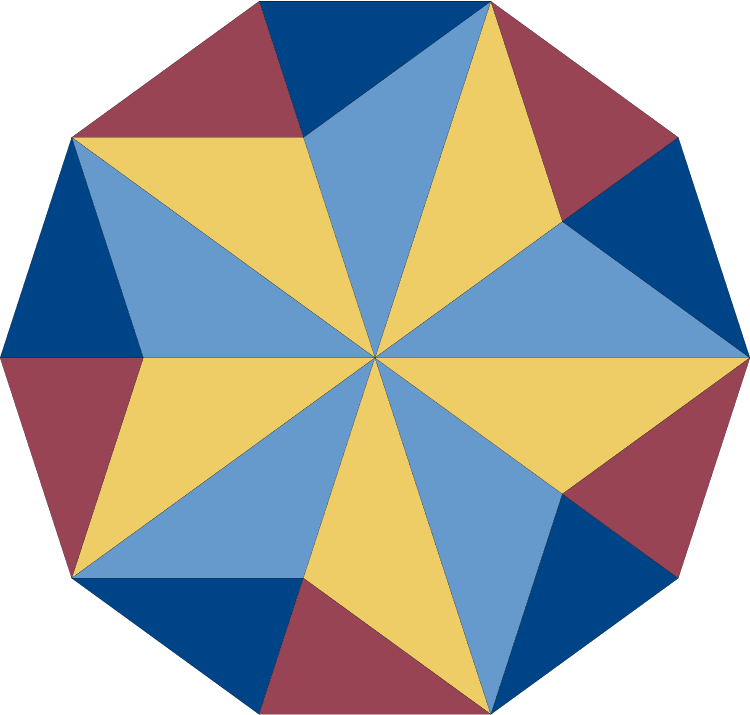}\qquad
\includegraphics[width=1.4in]{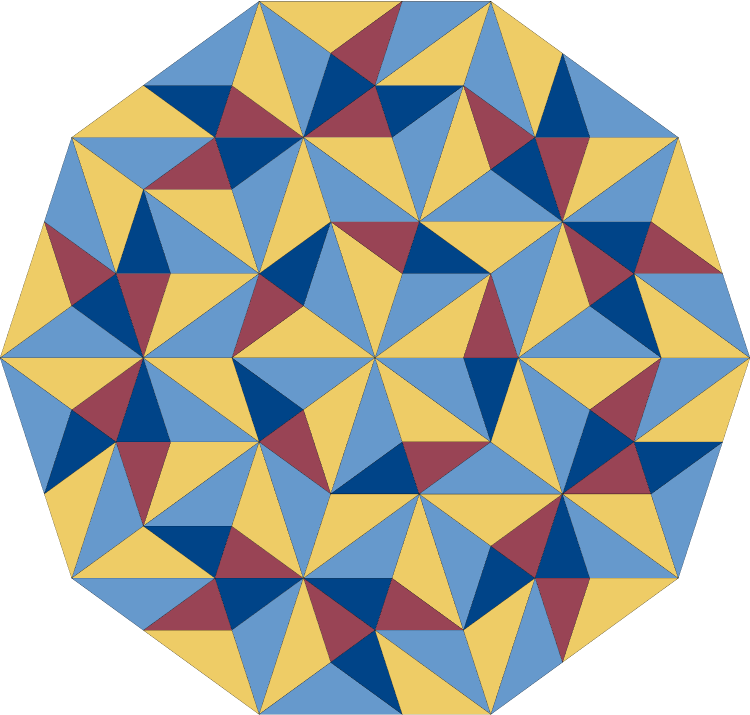}\qquad
\includegraphics[width=1.4in]{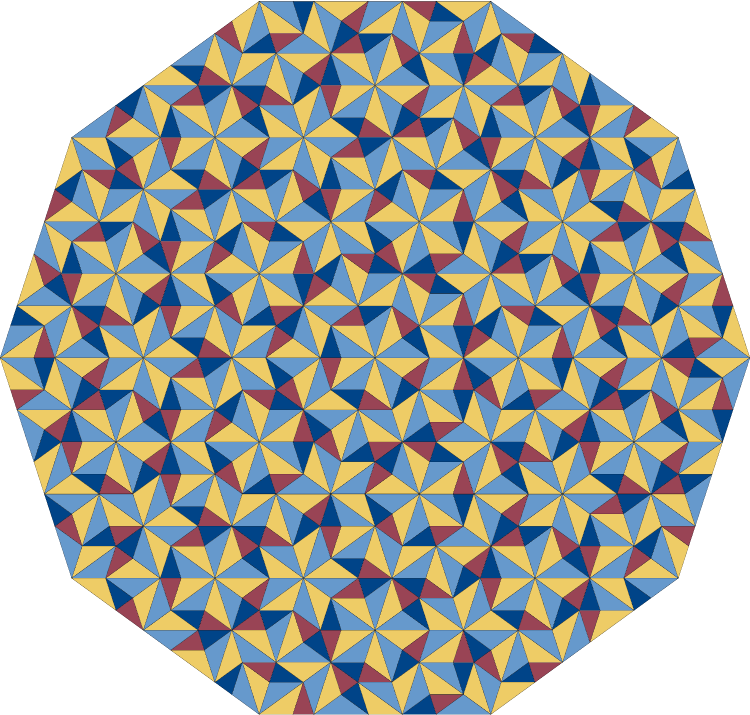}
\end{center}
\vspace*{5pt}
\begin{center}
\begin{minipage}{1.4in}
\begin{center}
\includegraphics[width=1.4in]{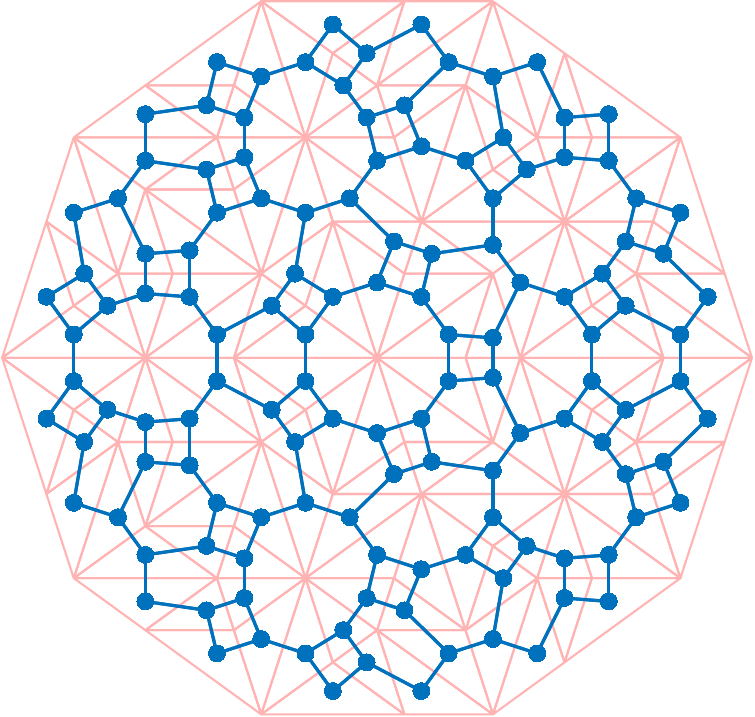}\\
\small $T_1$ tiling, level~2
\end{center}
\end{minipage}\qquad
\begin{minipage}{1.4in}
\begin{center}
\includegraphics[width=1.4in]{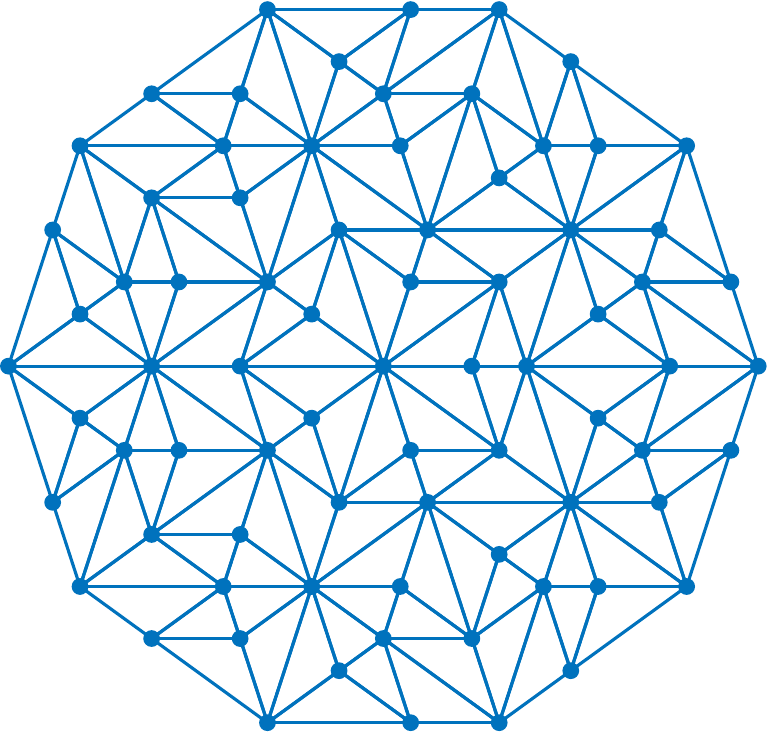}\\
\small $T_2$ tiling, level~2
\end{center}
\end{minipage}\qquad
\begin{minipage}{1.4in}
\begin{center}
\includegraphics[width=1.4in]{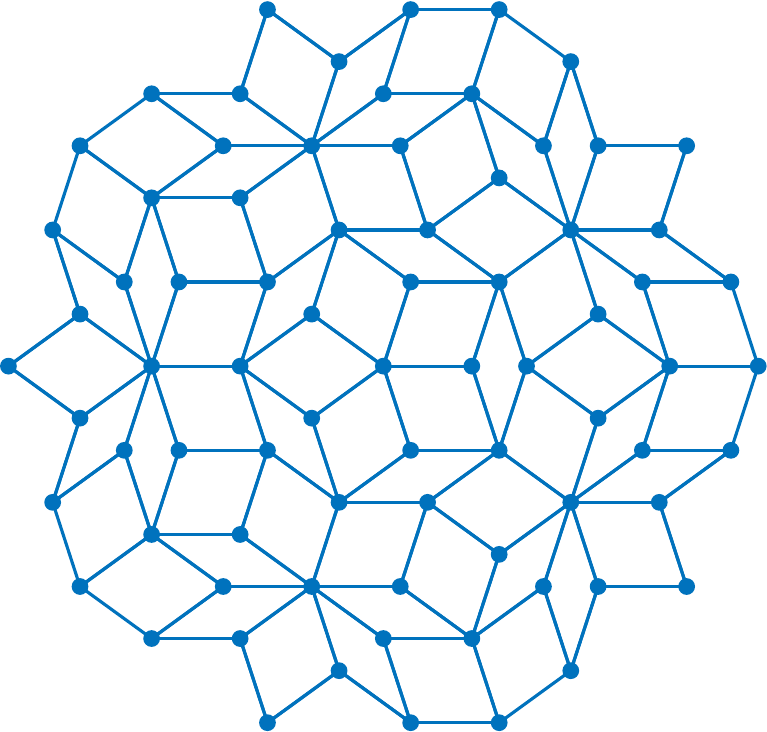}\\
\small $T_3$ tiling, level~2
\end{center}
\end{minipage}
\end{center}

\vspace*{-5pt}
\caption{Top: The 20-tile seed for the Penrose--Robinson triangle tiling, along with the 130 and 890 tile configurations that result from two and four iterations of the inflation rules. Bottom: The graphs for models $T_1$, $T_2$, and $T_3$ corresponding to the middle tiling in \cref{fig:circ_tri_tiling} (two iterations applied to the 20-tile seed). In this case, $T_1$ has 130~nodes and 185~edges; $T_2$ has 76~nodes and 205~edges;
$T_3$ has 76~nodes and 135~edges.}
\label{fig:circ_tri_tiling}
\end{figure}  

For each infinite graph with vertex set $\mathcal{V}$, we consider its graph Laplacian, 
which acts on $\ell^2(\mathcal{V}) = \{(\psi(i))_{i \in \mathcal{V}} : \sum |\psi(i)|^2<\infty\}$ via
$$
[H_k \psi](i) = \sum_{i \sim j} \left(\psi(i) - \psi(j)\right).
$$
Here $i \sim j$ means that the vertices $i$ and $j$ of the graph are connected by an edge, and the subscript $k=1,2,3$ corresponds to the choice of graph model.

\subsection{Computation of the spectrum}

\begin{figure}
\begin{center}
Penrose Tiling Model $T_1$
\bigskip

\includegraphics[width=4.5in]{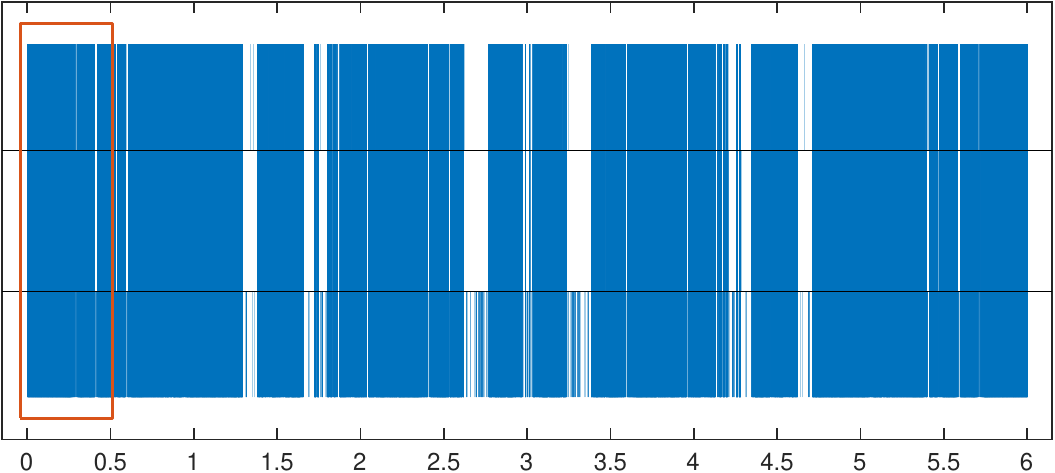}

\begin{picture}(0,0)
\put(-200,141){\rotatebox{90}{\makebox(0,0){\small finite}}}
\put(-190,141){\rotatebox{90}{\makebox(0,0){\small tiling}}}
\put(-200,99){\rotatebox{90}{\makebox(0,0){\small convergent}}}
\put(-190,99){\rotatebox{90}{\makebox(0,0){\small algorithm}}}
\put(-200,56){\rotatebox{90}{\makebox(0,0){\small finite}}}
\put(-190,56){\rotatebox{90}{\makebox(0,0){\small section}}}
\end{picture}
\end{center}
\vspace{-9mm}
\begin{center}
\includegraphics[width=4.5in]{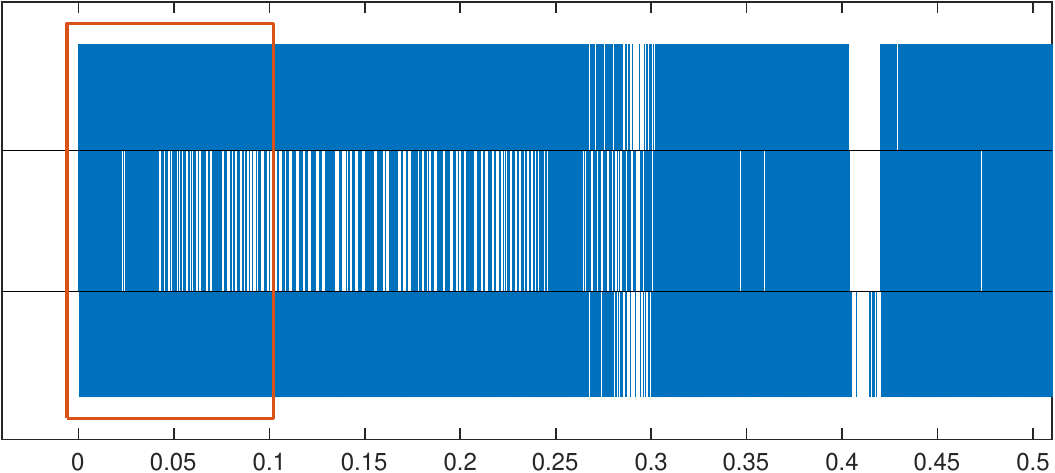}

\begin{picture}(0,0)
\put(-200,141){\rotatebox{90}{\makebox(0,0){\small finite}}}
\put(-190,141){\rotatebox{90}{\makebox(0,0){\small tiling}}}
\put(-200,99){\rotatebox{90}{\makebox(0,0){\small convergent}}}
\put(-190,99){\rotatebox{90}{\makebox(0,0){\small algorithm}}}
\put(-200,56){\rotatebox{90}{\makebox(0,0){\small finite}}}
\put(-190,56){\rotatebox{90}{\makebox(0,0){\small section}}}
\end{picture}
\end{center}
\vspace{-9mm}
\begin{center}
\includegraphics[width=4.5in]{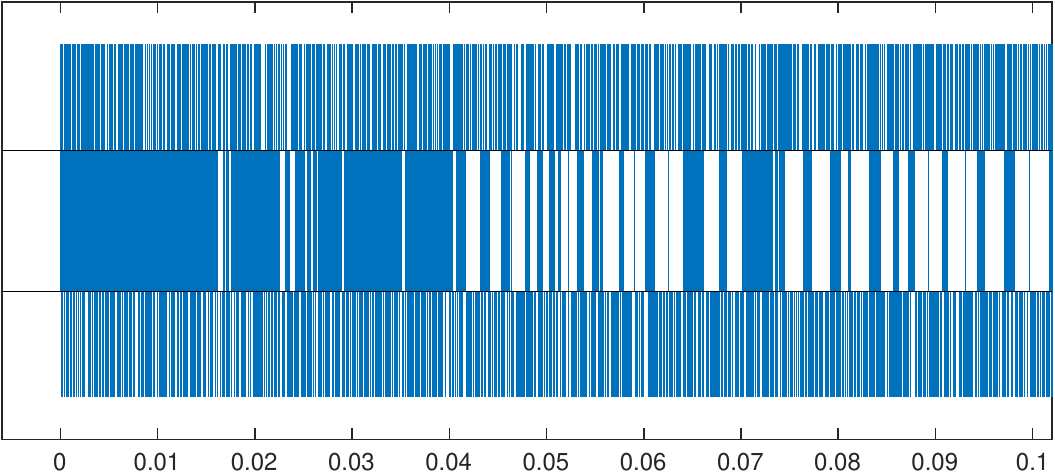}

\begin{picture}(0,0)
\put(-200,141){\rotatebox{90}{\makebox(0,0){\small finite}}}
\put(-190,141){\rotatebox{90}{\makebox(0,0){\small tiling}}}
\put(-200,99){\rotatebox{90}{\makebox(0,0){\small convergent}}}
\put(-190,99){\rotatebox{90}{\makebox(0,0){\small algorithm}}}
\put(-200,56){\rotatebox{90}{\makebox(0,0){\small finite}}}
\put(-190,56){\rotatebox{90}{\makebox(0,0){\small section}}}
\end{picture}
\end{center}\vspace{-5mm}
\caption{Comparison of approximations to the spectrum of Penrose tiling $T_1$, drawn with line widths equaling the true bin widths. The red boxes correspond to the zoom-in region in the plot below. 
The provably convergent algorithm \texttt{CompSpec} (``convergent algorithm'') uses a truncation to $10^6$ sites. The finite section and tiling methods use truncation to $109{,}460$ sites, but clearly exhibit spectral pollution in gaps of the spectrum.}
\label{fig_penrose_spec1}
\end{figure}
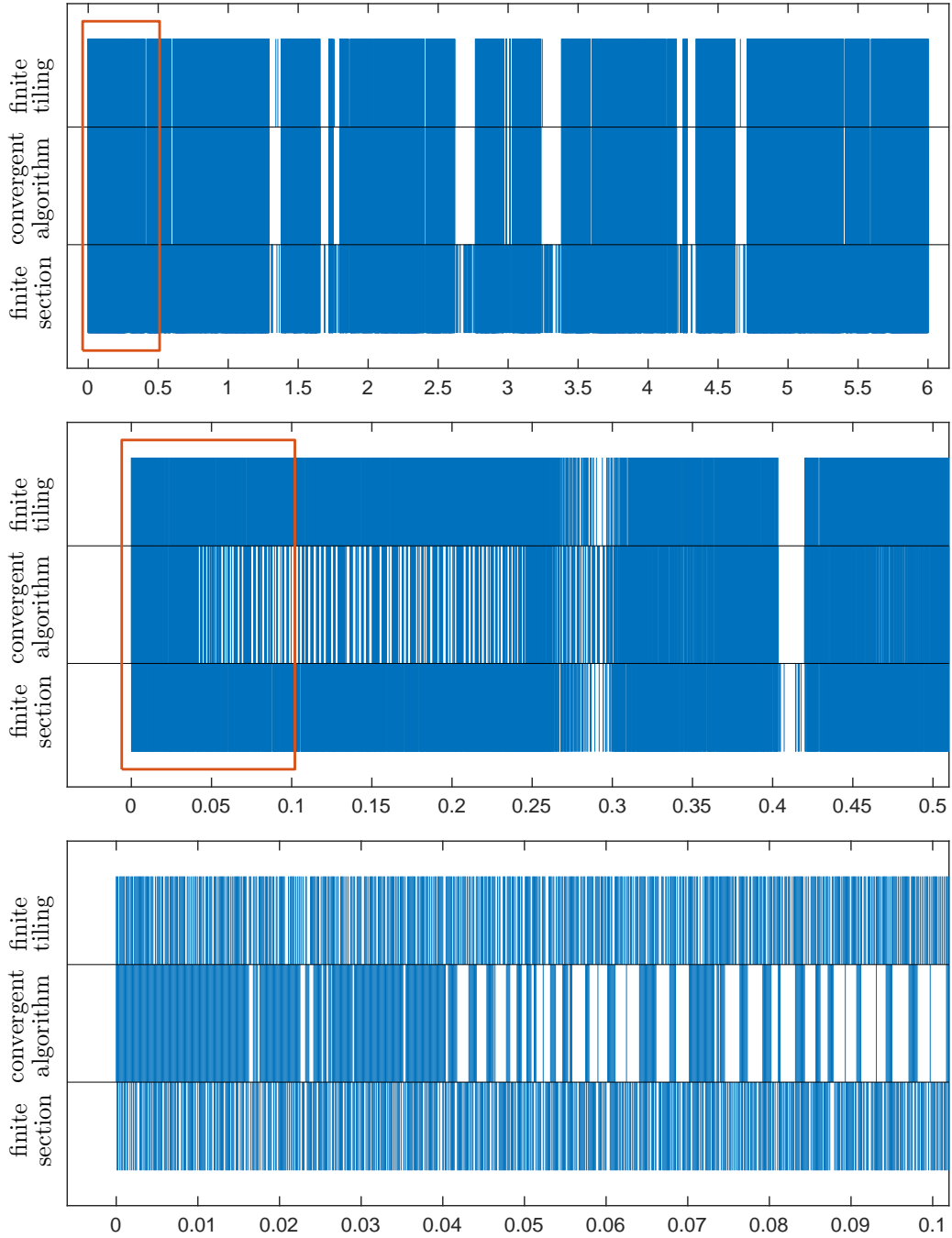

\begin{figure}
\begin{center}
Penrose Tiling Model $T_2$
\bigskip

\includegraphics[width=4.5in]{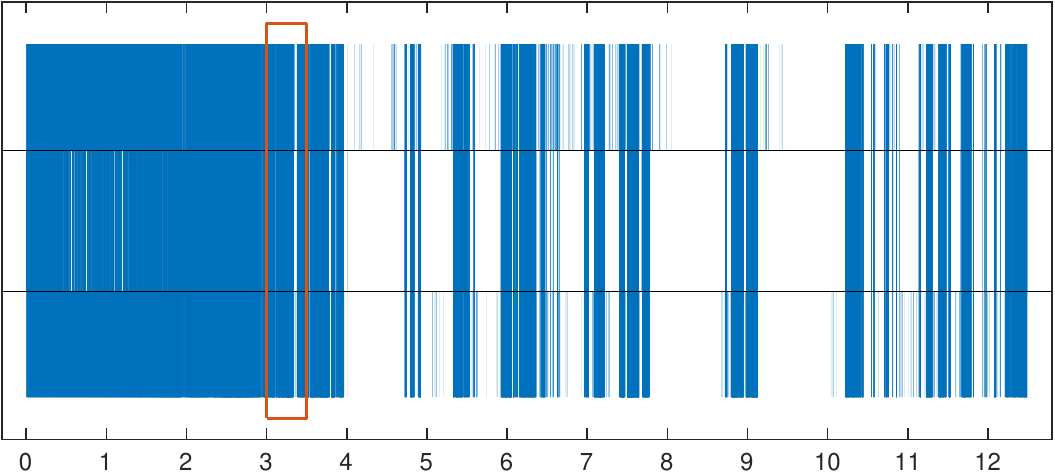}

\begin{picture}(0,0)
\put(-200,141){\rotatebox{90}{\makebox(0,0){\small finite}}}
\put(-190,141){\rotatebox{90}{\makebox(0,0){\small tiling}}}
\put(-200,99){\rotatebox{90}{\makebox(0,0){\small convergent}}}
\put(-190,99){\rotatebox{90}{\makebox(0,0){\small algorithm}}}
\put(-200,56){\rotatebox{90}{\makebox(0,0){\small finite}}}
\put(-190,56){\rotatebox{90}{\makebox(0,0){\small section}}}
\end{picture}
\end{center}
\vspace{-9mm}
\begin{center}
\hspace{1mm}\includegraphics[width=4.59in]{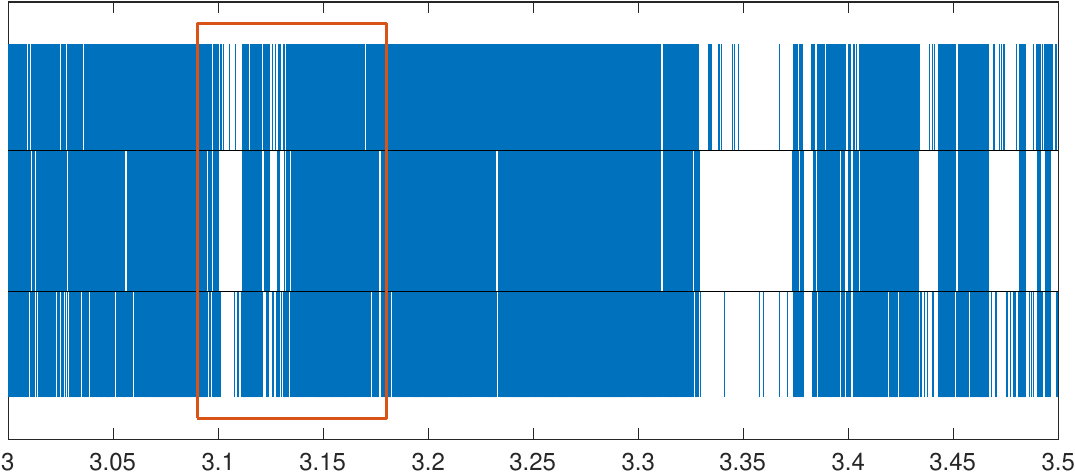}

\begin{picture}(0,0)
\put(-200,141){\rotatebox{90}{\makebox(0,0){\small finite}}}
\put(-190,141){\rotatebox{90}{\makebox(0,0){\small tiling}}}
\put(-200,99){\rotatebox{90}{\makebox(0,0){\small convergent}}}
\put(-190,99){\rotatebox{90}{\makebox(0,0){\small algorithm}}}
\put(-200,56){\rotatebox{90}{\makebox(0,0){\small finite}}}
\put(-190,56){\rotatebox{90}{\makebox(0,0){\small section}}}
\end{picture}
\end{center}
\vspace{-9mm}
\begin{center}
\includegraphics[width=4.7in]{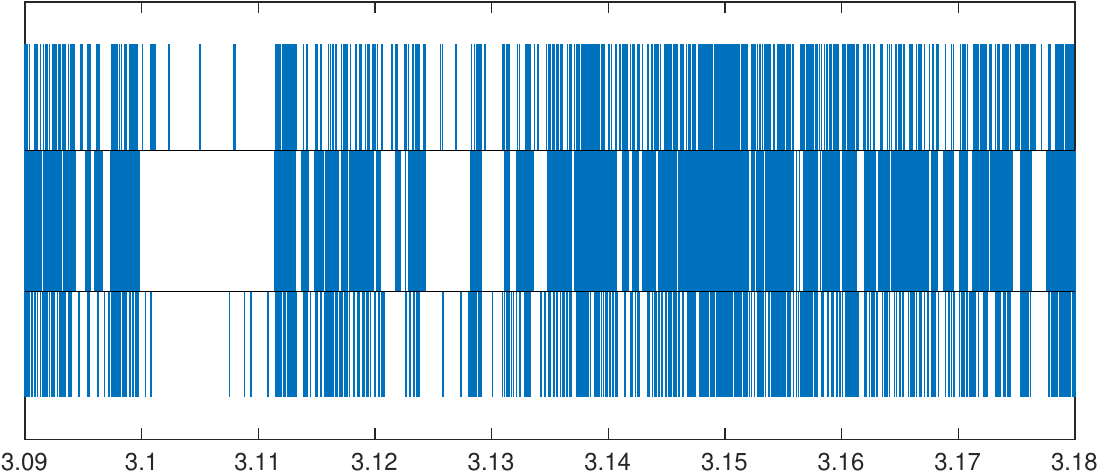}

\begin{picture}(0,0)
\put(-200,141){\rotatebox{90}{\makebox(0,0){\small finite}}}
\put(-190,141){\rotatebox{90}{\makebox(0,0){\small tiling}}}
\put(-200,99){\rotatebox{90}{\makebox(0,0){\small convergent}}}
\put(-190,99){\rotatebox{90}{\makebox(0,0){\small algorithm}}}
\put(-200,56){\rotatebox{90}{\makebox(0,0){\small finite}}}
\put(-190,56){\rotatebox{90}{\makebox(0,0){\small section}}}
\end{picture}
\end{center}\vspace{-5mm}
\caption{Same as \cref{fig_penrose_spec1} but for $T_2$. The provably convergent algorithm \texttt{CompSpec} (``convergent algorithm'') uses a truncation to $10^6$ sites. The finite section and tiling methods use truncation to $143{,}806$ sites, but again
clearly exhibit spectral pollution in gaps of the spectrum.}
\label{fig_penrose_spec2}
\end{figure}

\begin{figure}
\begin{center}
Penrose Tiling Model $T_3$
\bigskip

\includegraphics[width=4.5in]{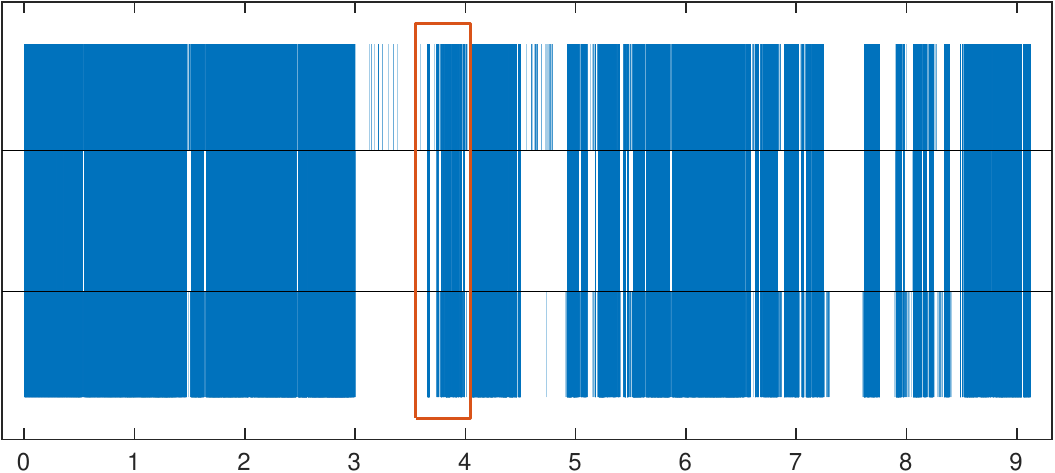}

\begin{picture}(0,0)
\put(-200,141){\rotatebox{90}{\makebox(0,0){\small finite}}}
\put(-190,141){\rotatebox{90}{\makebox(0,0){\small tiling}}}
\put(-200,99){\rotatebox{90}{\makebox(0,0){\small convergent}}}
\put(-190,99){\rotatebox{90}{\makebox(0,0){\small algorithm}}}
\put(-200,56){\rotatebox{90}{\makebox(0,0){\small finite}}}
\put(-190,56){\rotatebox{90}{\makebox(0,0){\small section}}}
\end{picture}
\end{center}
\vspace{-9mm}
\begin{center}
\includegraphics[width=4.69in]{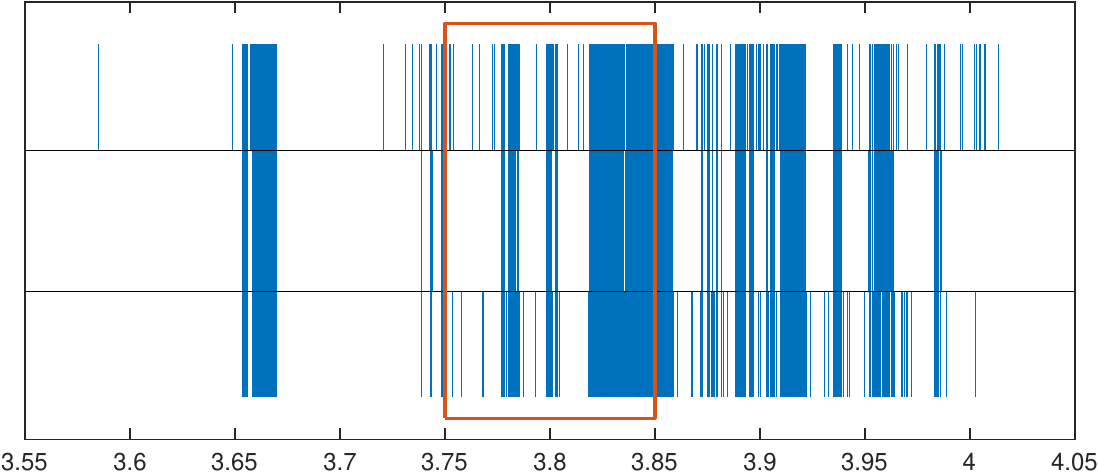}

\begin{picture}(0,0)
\put(-200,141){\rotatebox{90}{\makebox(0,0){\small finite}}}
\put(-190,141){\rotatebox{90}{\makebox(0,0){\small tiling}}}
\put(-200,99){\rotatebox{90}{\makebox(0,0){\small convergent}}}
\put(-190,99){\rotatebox{90}{\makebox(0,0){\small algorithm}}}
\put(-200,56){\rotatebox{90}{\makebox(0,0){\small finite}}}
\put(-190,56){\rotatebox{90}{\makebox(0,0){\small section}}}
\end{picture}
\end{center}
\vspace{-9mm}
\begin{center}
\includegraphics[width=4.69in]{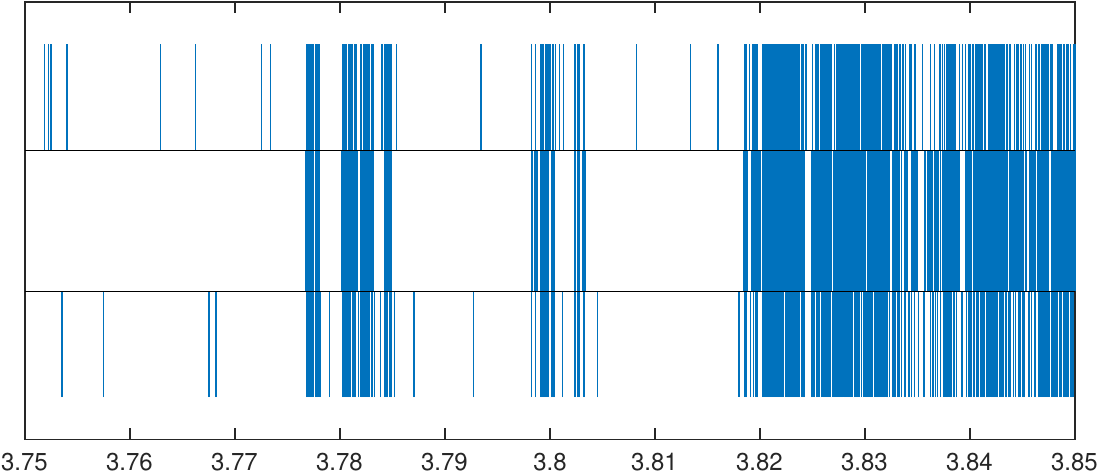}

\begin{picture}(0,0)
\put(-200,141){\rotatebox{90}{\makebox(0,0){\small finite}}}
\put(-190,141){\rotatebox{90}{\makebox(0,0){\small tiling}}}
\put(-200,99){\rotatebox{90}{\makebox(0,0){\small convergent}}}
\put(-190,99){\rotatebox{90}{\makebox(0,0){\small algorithm}}}
\put(-200,56){\rotatebox{90}{\makebox(0,0){\small finite}}}
\put(-190,56){\rotatebox{90}{\makebox(0,0){\small section}}}
\end{picture}
\end{center}\vspace{-8mm}
\caption{Same as \cref{fig_penrose_spec1,fig_penrose_spec2} but for $T_3$. The provably convergent algorithm \texttt{CompSpec} (``convergent algorithm'') uses a truncation to $10^6$ sites. The finite section and tiling methods use truncation to $143{,}806$ sites, and again exhibit spectral pollution in gaps of the spectrum.}
\label{fig_penrose_spec3}
\end{figure}

We first demonstrate why the functions $\Phi_n$ and the method outlined in \cref{sec:bounded_disp} are needed by showing how standard truncation methods do not approximate the spectrum. \cref{fig_penrose_spec1,fig_penrose_spec2,fig_penrose_spec3} show the approximated spectra of the operators $T_1$, $T_2$ and $T_3$, respectively. We have shown the output of: (i) the finite section method, corresponding to truncation of the tiles with free boundary conditions (akin to taking $f(n)=n$ in the framework of \cref{sec:bounded_disp}); (ii) the finite tiling method, corresponding to truncating the tiles and then forming the graph Laplacian of the truncated tiling; and (iii) the output of \texttt{CompSpec} (``convergent algorithm'') from \cite{colb1,colbrook2022foundations}, which performs a local minimization of the functions $\Phi_n$ constructed in \cref{sec:bounded_disp} to compute the spectrum. For each model, we see heavy spectral pollution of the finite section and tiling methods: areas where the ``convergent algorithm'' signals no spectrum, but the ``finite tiling'' or ``finite section'' methods produce spectrum. For example, such pollution is especially prevalent for the model $T_1$ in the plot that zooms in around $[0, 0.1]$; we see many points of pollution between $0.04$ and $0.1$ (bottom of \cref{fig_penrose_spec1}). This pollution destroys the apparent fractal structure of the spectrum.  Indeed, we found that such pollution prevented reliable results when using the finite section and tiling methods for computing properties such as the fractal dimension of the spectrum (e.g., they give the wrong scaling for approximating the box-counting dimension). The results for the finite section and tiling methods involve lengthy computation of all eigenvalues of matrices of dimension $109{,}460$ (model $T_1$) and $143{,}806$ (models $T_2$ and $T_3$); significantly larger truncations would require high-performance computing.

\begin{figure}
\begin{center}
\includegraphics[width=3.7in,trim={0mm 2mm 0mm 2mm},clip]{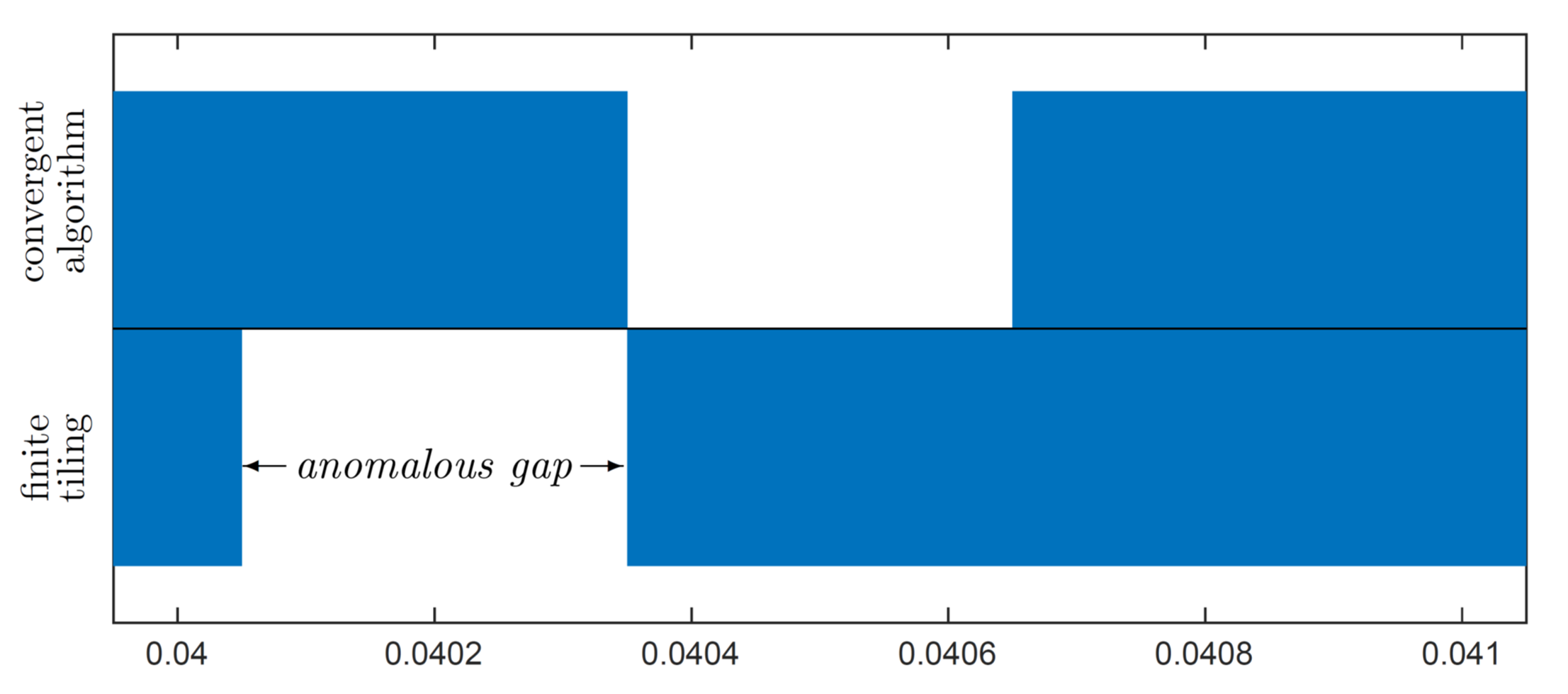}
\end{center}
\caption{Zoom-in of the computed spectrum for $T_1$. It seems curious that the finite tiling method does not show spectrum around $0.0402$, as the convergent algorithm does.}
\label{fig_penrose_spec_zoom}
\end{figure}

While spectral pollution is not surprising, if one looks closely, there are also a few other points where the ``convergent algorithm'' locates spectra, but the ``finite tiling'' or ``finite section'' methods do not. Let us investigate this anomaly in more detail. \cref{fig_penrose_spec_zoom} shows one such case, in the interval $[0.04,0.0406]$. The ``finite tiling'' spectrum is based on level~9 of $T_1$ (109,460~tiles).  The locations of the eigenvalues are ``rounded'' to bins that are analogous to those used by the ``convergent algorithm''; these bins have width $10^{-4}$. In the bins centered at 0.0401, 0.0402, and 0.0403, the ``finite tiling'' has no spectrum, while the ``convergent algorithm'' shows spectrum. We can investigate this situation in more detail using spectrum slicing~\cite[Section~3.3]{Par98}. Let $L_k$ denote the graph Laplacian for the level-$k$ tiling. By computing the inertia of $E-L_k$ for values of $E \in [0.040,0.041]$ spaced at $10^{-5}$ intervals, we can determine intervals of much larger finite tilings that have spectrum in the apparently anomalous region. \cref{fig_penrose_spec_slice} shows the results, up to level~14 (13,462,690 tiles). These computations reveal that the finite tiling eigenvalues eventually fill in the entire region $[0.040,0.041]$. Our expectations from the results of the ``convergent algorithm'' are confirmed.

\begin{figure}
\begin{center}
\includegraphics[width=4in,trim={0mm 2mm 0mm 2mm},clip]{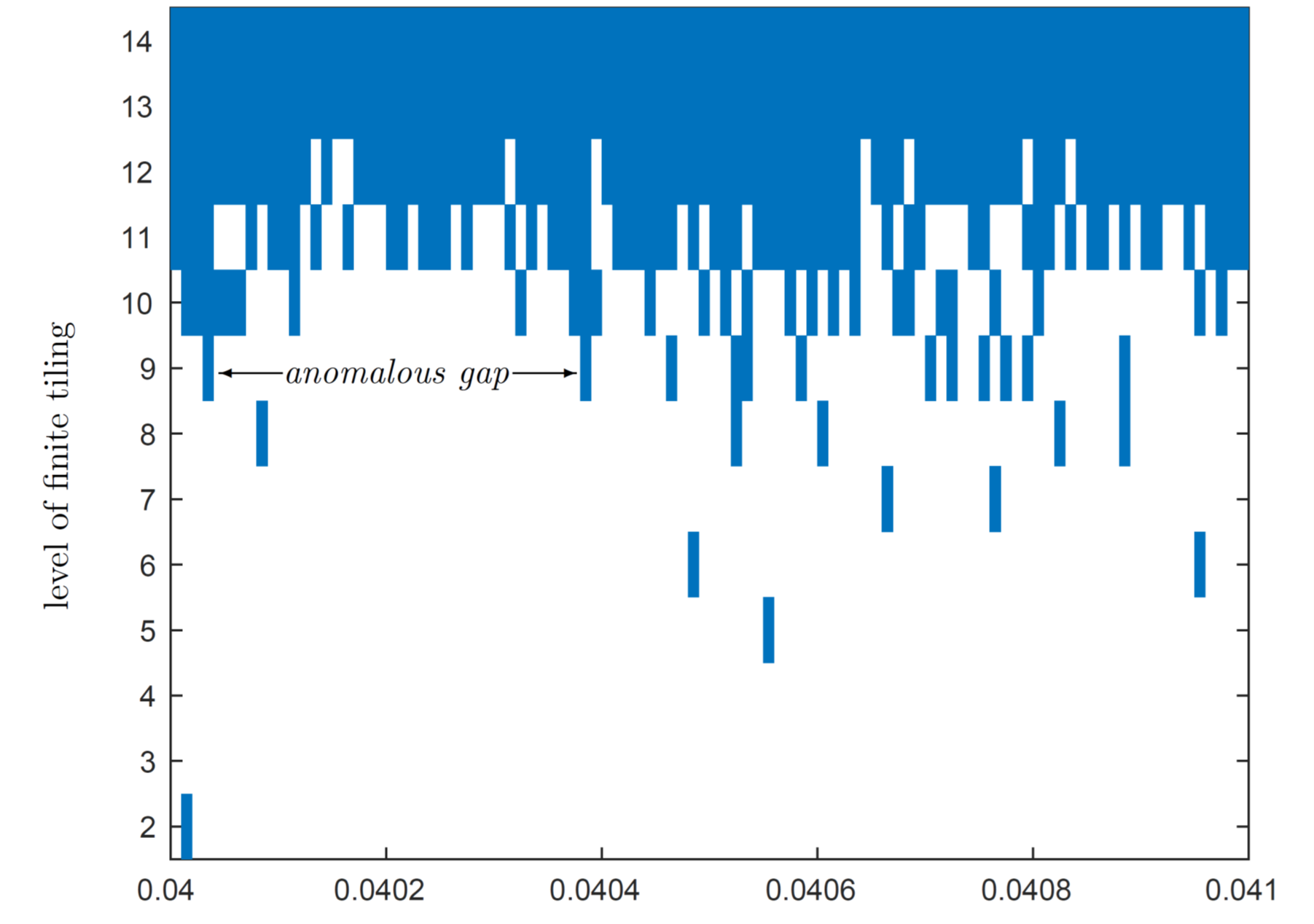}
\end{center}
\caption{Using spectrum slicing, we can compute spectrum in the region $[0.040,0.041]$ for 
the finite tiling method at much larger truncation sizes. At these levels, we see eigenvalues emerge 
in the ``anomalous gap'' in \cref{fig_penrose_spec_zoom}, consistent with our expectations from the
convergent algorithm.}
\label{fig_penrose_spec_slice}
\end{figure}

\begin{figure}
\centering
\raisebox{-0.5\height}{\includegraphics[width=0.441\textwidth,trim={0mm 0mm 0mm 0mm},clip]{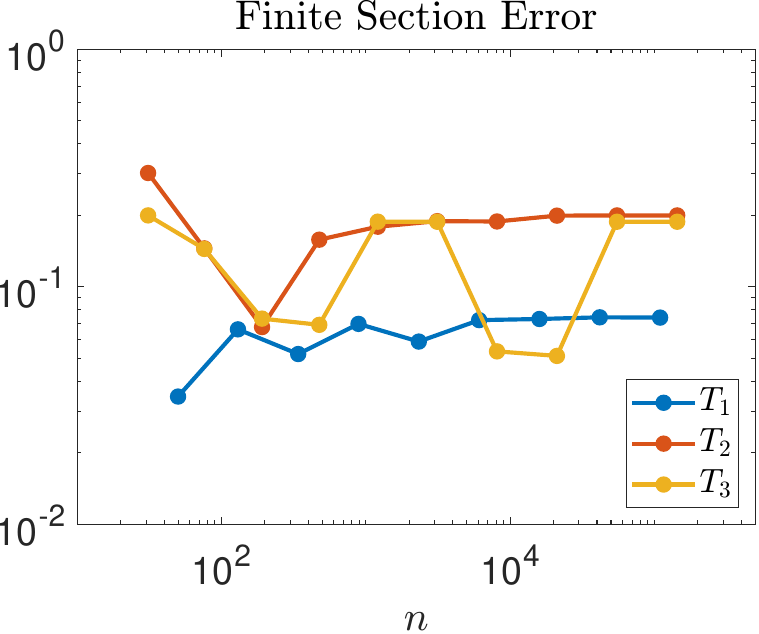}}\hfill
\raisebox{-0.5\height}{\includegraphics[width=0.441\textwidth,trim={0mm 0mm 0mm 0mm},clip]{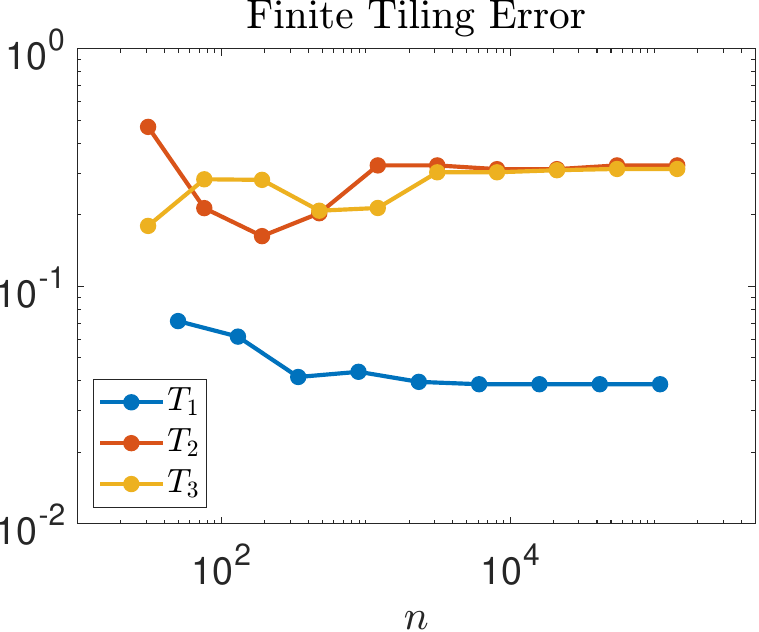}}
\caption{The finite section and finite tiling methods stagnate and do not converge due to spectral pollution. This pollution also means that the finite section and finite tiling methods do not enable us to compute fine spectral features, such as fractal dimensions. Instead, we use the method outlined in \cref{sec:bounded_disp} for more reliable results.}
\label{fig:Pen_FS_error}
\end{figure}

We can use the functions $\Phi_n$ constructed in \cref{sec:bounded_disp} to measure the error of the finite section and tiling methods. \cref{fig:Pen_FS_error} shows the maximum error for different truncation sizes of the tiling, where we have computed the error using $\Phi_{10^6}$, corresponding to a truncation size of $10^6$. We see that neither the finite section nor the tiling methods converge. In contrast, the maximum errors for \texttt{CompSpec} in \cref{fig_penrose_spec1,fig_penrose_spec2,fig_penrose_spec3} are
0.0012,
0.0053, and
0.0021,
respectively. Moreover, these errors provably converge to zero as the truncation size increases.

In summary, there are several advantages of our approach to approximate the spectrum:
\begin{itemize}
	\item Using the functions $\Phi_n$ gives approximations that converge without spectral pollution, avoiding misleading results concerning the size of the spectrum.
	\item The method provides error control, allowing us to approximate the resolution to which we have approximated the spectrum for a given truncation parameter. We shall see that this aids in implementing algorithms that require multiple limits.
\end{itemize}
All of these remarks are typical of $\Sigma_1^A$ convergence to the spectrum. Using $\Phi_n$, we can compute the spectrum and ensure that each portion of the output from our algorithm is within a controllable error of the spectrum.

\subsection{Computation of properties of the spectrum}

Given the truncation parameter $n=10^6$ for computing the functions $\Phi_n$ that satisfy \textbf{(S2)}, we estimate a suitable resolution $\delta=\delta_*$ of the cover $\Gamma_{n_2,n_1}^\spec$ in \eqref{S2_cover_alg} as follows. We run \texttt{CompSpec} to approximate the spectrum. The maximum value of $\Phi_n$ over the output indicates to what resolution the spectrum has been approximated.

\subsubsection{Spectral gaps}

\begin{figure}
\centering
\raisebox{-0.5\height}{\includegraphics[width=0.32\textwidth,trim={0mm 0mm 0mm 0mm},clip]{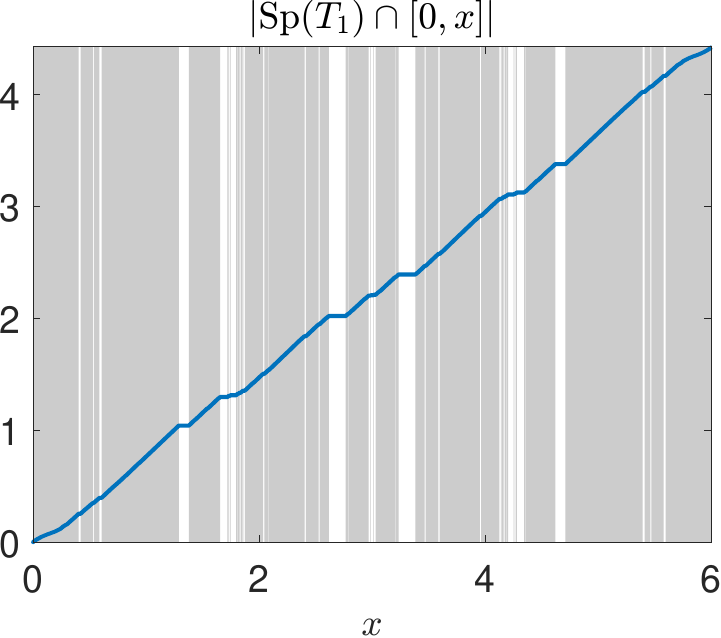}}
\hfill
\raisebox{-0.5\height}{\includegraphics[width=0.32\textwidth,trim={0mm 0mm 0mm 0mm},clip]{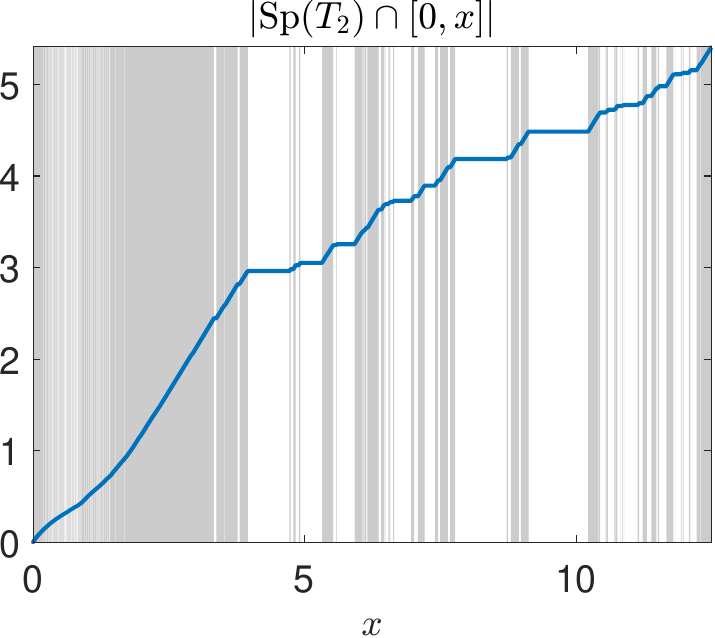}}
\hfill
\raisebox{-0.5\height}{\includegraphics[width=0.32\textwidth,trim={0mm 0mm 0mm 0mm},clip]{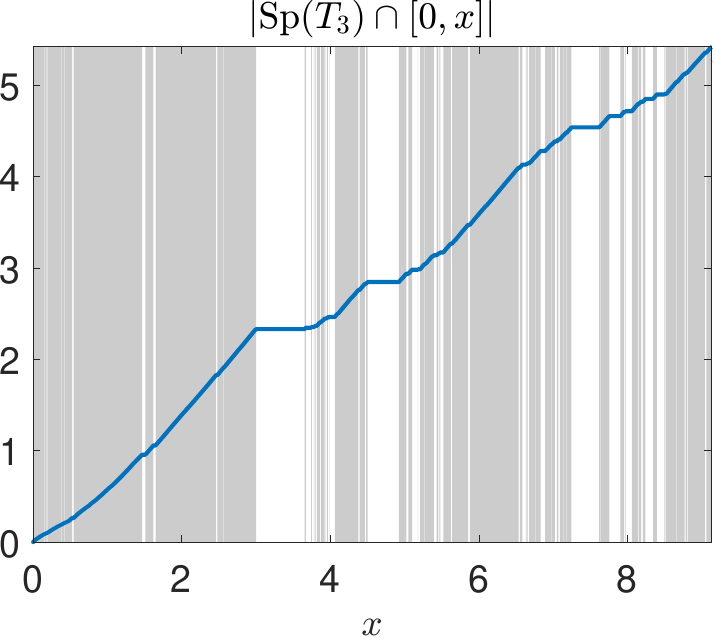}}
\caption{Cumulative Lebesgue measure of the covers $\Gamma_{n_2,n_1}^\spec$ computed using $1/n_2=\delta=3\delta_*$ and $n_1=10^6$. The shaded regions show the covers themselves.}
\label{fig:Pen_Leb_Meas}
\end{figure}

We first consider the number of gaps of the spectrum. Given a value of $\delta$, we compute the number of local maxima of $\Phi_{10^6}$ at least as large as $\delta$, ensuring that these maxima are at least $2\delta$-separated. This tells us the number of gaps of length $\geq 2\delta$. As well as using $\Phi_{10^6}$, we also perform this computation for the finite section and the finite tiling methods, using the functions that correspond to the distance to the set of eigenvalues of each method (with the same truncation parameters as in \cref{fig_penrose_spec1,fig_penrose_spec2,fig_penrose_spec3}). \cref{fig:Pen_gaps} shows the results, where the left vertical dashed line corresponds to $\delta_*$ and the right vertical dashed line corresponds to the maximum of $\Phi_n$ in the convex hull of the output of \texttt{CompSpec}. (This quantity gives an upper bound on half the maximum gap in the spectrum.) \cref{fig:Pen_gaps} provides evidence that each spectrum has infinitely many components. Furthermore, we see that the spectral pollution associated with the finite section and the finite tiling methods distorts the distribution of the size of the gaps.  Thus, the \texttt{CompSpec} approach provides a more reliable tool for investigating gap structure and integrated density of states, compared to the finite tilings used in~\cite[Section~8]{DEFM202XUEXM}.

\begin{figure}
\centering
\raisebox{-0.5\height}{\includegraphics[width=0.32\textwidth,trim={0mm 0mm 0mm 0mm},clip]{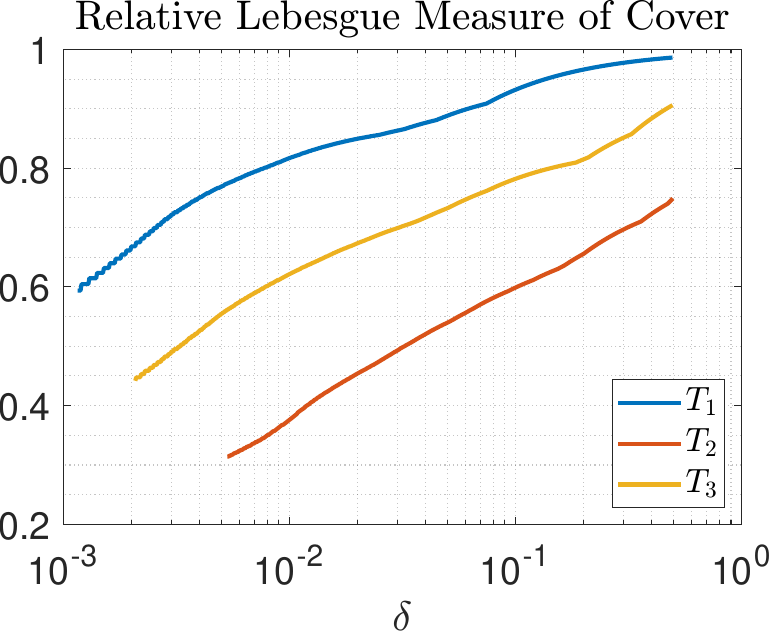}}
\hfill
\raisebox{-0.5\height}{\includegraphics[width=0.32\textwidth,trim={0mm 0mm 0mm 0mm},clip]{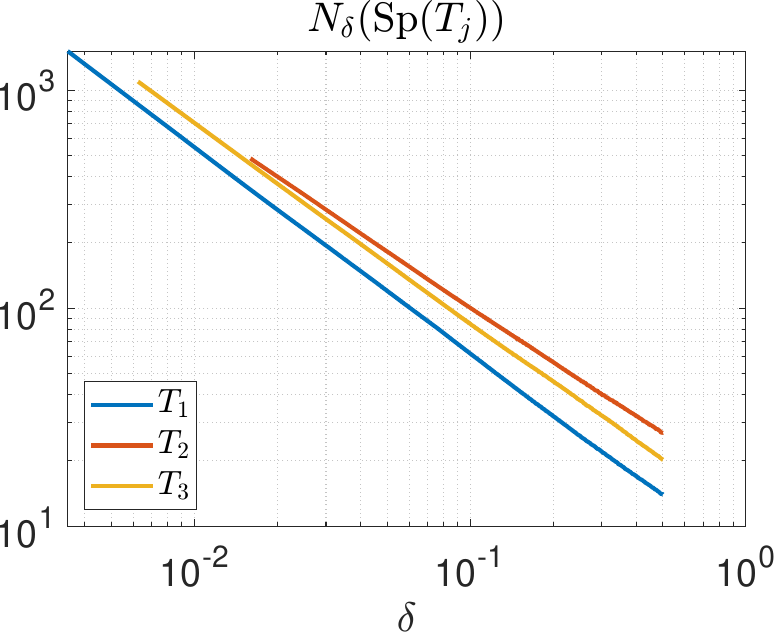}}
\hfill
\raisebox{-0.5\height}{\includegraphics[width=0.32\textwidth,trim={0mm 0mm 0mm 0mm},clip]{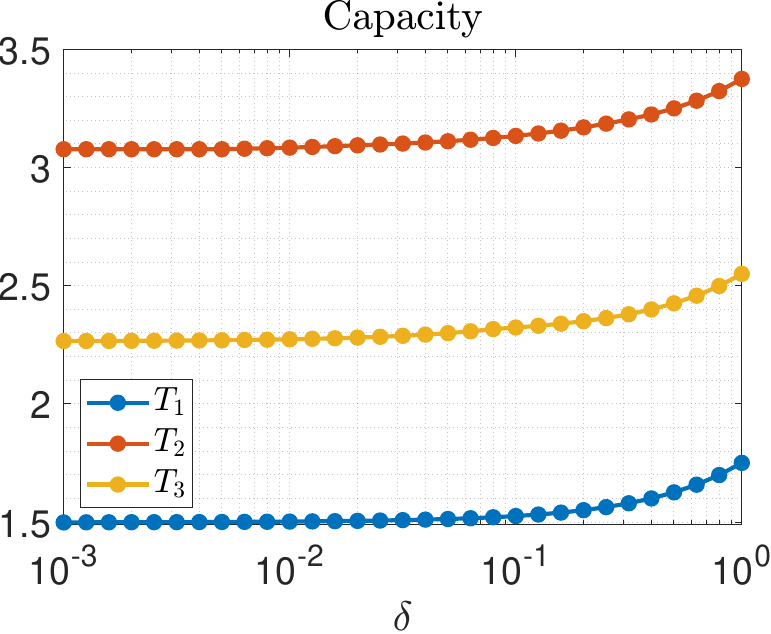}}
\caption{Left: Lebesgue measure of the covers $\Gamma_{n_2,n_1}^\spec$ computed using $n_1=10^6$ and various $\delta=1/n_2$. Middle: The number of $\delta$-mesh intervals that intersect the cover $\Gamma_{n_2,10^6}^\spec$ that has $1/n_2=3\delta_*$. Right: Capacities of the covers $\Gamma_{n_2,n_1}^\spec$ computed using $n_1=10^6$ and various $\delta=1/n_2$.}
\label{fig:Pen_Leb_Meas2}
\end{figure}

\subsubsection{Lebesgue measure}

Next, we consider the Lebesgue measure of the spectrum. We set $n_1=10^6$ and $\delta=1/n_2$ in the covers $\Gamma_{n_2,n_1}^\spec$. We first choose our grid to be points equally spaced at distance $\delta=3\delta_*$. \cref{fig:Pen_Leb_Meas} shows the corresponding cumulative Lebesgue measure for each model, where the covers are shown as the vertical gray boxes. Notice that the Lebesgue measure is much smaller than the width of the spectrum. To investigate this further, \cref{fig:Pen_Leb_Meas2} (left) plots the total Lebesgue measure (normalized by the spectrum width) against $\delta$. This is consistent with the $\Pi_2^A$ convergence and suggests that at the computed resolution (smallest $\delta$), the value of the Lebesgue measure has not converged and is likely smaller.

\subsubsection{Fractal dimension and capacity}

We also consider the box-counting dimension. \cref{fig:Pen_Leb_Meas2} (middle) shows the number of mesh intervals that intersect our cover  $\Gamma_{n_2,n_1}^\spec$. There is evidence of a scaling region with slopes approximately $0.95$, $0.86$, and $0.92$ for $T_1$, $T_2$, and $T_3$, respectively. These results support the conjecture that the Lebesgue measure of the spectrum is zero. Finally, \cref{fig:Pen_Leb_Meas2} (right) shows the $\Pi_2^A$ convergence of capacity, suggesting that the capacity of the spectrum is less than $1.50$, $3.08$, $2.27$ for $T_1$, $T_2$, and $T_3$, respectively.

\section{Final Remarks and Conjectures}
\label{sec:conjectures}

We have developed and demonstrated the first algorithms that are provably optimal for computing various notions of size or complexity of spectra and, more generally, compact sets. These algorithms rely on two sets of assumptions. Under \textbf{(S1)}, we assume access to covers of the set of interest, with control over the distance of these covers to the set in the Hausdorff metric. This assumption is satisfied, for example, by many one-dimensional quasicrystal and aperiodic models, and more broadly in settings where suitable cover constructions are available. Recent advances show that such constructions can, in principle, be realized in certain higher-dimensional settings under appropriate structural assumptions. Under \textbf{(S2)}, we assume access to functions that converge to the distance to the set. This condition is satisfied for a broad class of operators, including higher-dimensional quasicrystal models, but does not provide explicit enclosing covers. Consequently, computations under \textbf{(S2)} generally require an additional limiting step compared to those under \textbf{(S1)}. The SCI hierarchy formalizes this distinction, enabling us to classify the difficulty of computational problems and establish the optimality of our algorithms. For example, by considering the class of limit-periodic discrete Schr\"odinger operators, we proved that our algorithms under \textbf{(S1)} are optimal. Additionally, several of these classifications are suitable for the use of interval arithmetic and computer-assisted proofs.

\subsubsection*{Conjectures}

Some conjectures based on the computational examples of this paper are the following:
\begin{itemize}[leftmargin=0.8cm]
	\item \textbf{Almost Mathieu operator:} We conjecture that
	$$
	\mathrm{dim}_{\mathrm{H}}(\spec_{+}((\sqrt{5}-1)/2,1))=\mathrm{dim}_{\mathrm{B}}(\spec_{+}((\sqrt{5}-1)/2,1))=1/2,
	$$
	whereas
	$$
	\mathrm{dim}_{\mathrm{H}}(\spec_{+}(C,1))< \underline{\mathrm{dim}}_{\mathrm{B}}(\spec_{+}(C,1))=1/2<\overline{\mathrm{dim}}_{\mathrm{B}}(\spec_{+}(C,1))=2/3,
	$$
	where $C$ is Cahen's constant.\vspace{1mm}
	\item \textbf{Fibonacci Hamiltonian:} We conjecture that the box-counting and Hausdorff dimensions of the spectrum of the $d$-dimensional Fibonacci Hamiltonian are equal and satisfy
\begin{equation} \label{eq:dDimFibConj}
\mathrm{dim}_{\mathrm{H}}(\spec(H_\lambda^{(d)}))=\mathrm{dim}_{\mathrm{B}}(\spec(H_\lambda^{(d)}))=\min\left\{1,d\cdot\mathrm{dim}_{\mathrm{B}}(\spec(H_{\lambda}))\right\}.
\end{equation}
A corollary of this result would be that
$$
\lim_{\lambda\rightarrow\infty} \mathrm{dim}_{\mathrm{H}}(\spec(H_\lambda^{(d)}))\cdot\log(\lambda)=d\cdot\log(1+\sqrt{2}).
$$
We also found that $\mathrm{dim}_{\mathrm{B}}(\spec(H_\lambda))\approx 1-0.22\lambda$ for small $\lambda$, $\mathrm{dim}_{\mathrm{B}}(H_{\lambda}^{(2)})$ becomes smaller than one for $\lambda\approx 3.2$, and $\mathrm{dim}_{\mathrm{B}}(H_{\lambda}^{(3)})$ becomes smaller than one for $\lambda\approx 8.7$. Furthermore, we conjecture that the number of connected components of $\spec(H_{\lambda}^{(3)})$ becomes infinite around $\lambda\approx1+\sqrt{2}$ and that there is a region where the number of connected components is infinite, but the fractal dimension is 1.\vspace{1mm}
\item \textbf{Penrose tiles:} We conjecture that for the spectrum of the graph Laplacian of the three tilings, $T_1$, $T_2$, and $T_3$, the number of connected components is infinite and the Lebesgue measure is zero. We also conjecture that the box-counting dimensions are approximately $0.95$, $0.86$, and $0.92$, whilst the capacities are approximately $1.50$, $3.08$, $2.27$ for $T_1$, $T_2$, and $T_3$, respectively.
\end{itemize}

\subsubsection*{Extensions}

There are also several extensions of the methods of this paper for future work:
\begin{itemize}[leftmargin=0.8cm]
\item \textbf{(S1) covers for further two-dimensional quasicrystal models:} Developing explicit and computationally effective \textbf{(S1)} spectral covers for concrete two-dimensional quasicrystal Hamiltonians (beyond the cases currently understood) would lower the SCI classification of several problems considered here. In particular, constructing sharp and implementable covers for models such as the Penrose tiling would enable certified two-sided error control and facilitate high-resolution, computer-assisted spectral analysis.\vspace{1mm}
\item \textbf{Unitary operators:} It is straightforward to extend our techniques to unitary operators by considering the analogue of \textbf{(S1)} and \textbf{(S2)} on the unit circle. There is a growing interest in aperiodic unitary operators, such as the unitary almost Mathieu operator \cite{cedzich2021almost}.\vspace{1mm}
\item \textbf{Unbounded operators:} We can consider convergence in the Attouch--Wets topology, which generalises the Hausdorff metric to general non-empty closed sets and captures locally uniform convergence. It should be possible to extend the techniques of this paper to deal with differential operators on $L^2(\mathbb{R}^d)$, such as Schr\"odinger operators.\vspace{1mm}
\item \textbf{Higher-dimensional models:} The methods in this paper that are based on \textbf{(S2)} are entirely general and readily applicable to quasicrystal models in three and higher dimensions, which have been understudied due to their complexity. We anticipate that this approach will not only advance the understanding of higher-dimensional quasicrystal models but also pave the way for new algorithms to tackle their challenging spectral properties. 
\end{itemize}

\small

\renewcommand{\baselinestretch}{0.97}

\bibliography{aperiodic}
\bibliographystyle{abbrv}
\end{document}